\documentclass[10pt,a4paper]{article}
\usepackage{amsmath, amssymb, graphics, setspace}
\usepackage{empheq}
\usepackage{soul}
\usepackage{amssymb,amsbsy,amsmath,amsfonts,amssymb,amscd,amsthm, mathrsfs, bbm, bigints}
\usepackage{amssymb}
\usepackage{color}

\usepackage{hyperref}

\newcommand{\1}{\mathbbm{1}}

\newcommand{\CC}{\mathbb{C}}
\newcommand{\NN}{\mathbb{N}}
\newcommand{\RR}{\mathbb{R}}

\newtheorem{theo}{Theorem}
\newtheorem{prop}[theo]{Proposition}
\newtheorem{lem}[theo]{Lemma}
\newtheorem{cor}[theo]{Corollary}
\newtheorem{rem}[theo]{Remark}

\newcommand{\beqn}{\begin{equation}}
\newcommand{\eeqn}{\end{equation}}
\newcommand{\bear}{\begin{eqnarray}}
\newcommand{\eear}{\end{eqnarray}}
\newcommand{\bean}{\begin{eqnarray*}}
\newcommand{\eean}{\end{eqnarray*}}

\begin{document}

\title{Regularizing effects in a linear kinetic equation \\ for cubic interactions.}
\maketitle
\begin{center}
{\large M. Escobedo}\\
{\small Departamento de Matem\'aticas,} \\
{\small Universidad del
Pa{\'\i}s Vasco,} \\
{\small Apartado 644, E--48080 Bilbao, Spain.}\\
{\small E-mail~: {\tt miguel.escobedo@ehu.es}}
\end{center}
\noindent
{\bf Abstract}: 
We describe regularizing  effects in the linearization of a kinetic equation that arises  in study of a system of nonlinear waves satisfying  the Schr\"odinger equation in terms of weak turbulence and condensate. The problem is first considered in spaces of bounded functions with weights, where existence of solutions and some first regularity properties are proved. After a suitable change of variables the equation is written in terms of a pseudo differential operator. Homogeneity of the equation and  classical arguments of freezing of coefficients may then be used to prove regularizing  effect in local Sobolev type  spaces.



\section{Introduction}
\setcounter{equation}{0}
\setcounter{theo}{0}

The non linear  kinetic equation arising  in the study of a system of nonlinear waves satisfying  the Schr\"odinger equation in terms of weak turbulence if one  considers the  
fluctuations  of the solutions around a Dirac measure has been considered by several authors  (\cite{DPR},\cite{D},\cite{N},\cite{Zbook}). That equation has a family of stationary solutions $n_0(X)=CX^{-1}$, $C >0$ constant,  where $X$ is the frequency of the waves. When only the three wave interaction process are kept,  and under suitable assumptions on the physical system (in particular on its temperature; cf. \cite{D, Eckern, Kirkpatrick}), the evolution of an initial perturbation 
$n_0(1+v)$ of  the equilibrium $n_0$ may be described in terms of the function $v(t, X)$ itself, solution of the following problem (using scaled variables and $ C =1$  in order to get rid of  constants unnecessary for our analysis),
\begin{align}
&\frac {\partial v (t, X  )} {\partial t}=\mathscr L(v (t))(X )+N (v (t))(X ),t>0,\,\,X>0\label{S2Ewxi2}\\
&v(0, X)=v(X),\,X>0\\
&\mathscr L(v (t))(X )=\int _0^\infty (v (t ,  Y  )-v (t , X  ))   M(X , Y ) dY  \label{S3E23590L}\\
& M(X , Y)= \frac {1} {X  ^{1/2} }\left(\frac {1} { | X -Y |}-\frac {1} { (X +Y ) } \right) , \label{S3E23590M}
\end{align}
where the nonlinear operator $N$ in (\ref {S2Ewxi2}) reads, in terms of the function $\widetilde v(X)= X^{-1}v(X)$,

\begin{align*}
&N[\omega](X)=X^{1/2}\int _0^X \Bigg\{\widetilde v(X-X_1)\left(\widetilde v(X_1)-\widetilde v(X)\right)
+\left( \widetilde v(X_1)+\widetilde v(X)\right)\widetilde v(X+X_1)\Bigg\}dX_1-\nonumber\\
&\hskip 1.7cm -X^{1/2}\int _X^\infty \Bigg\{\widetilde v(X_1-X\left( \widetilde v(X)-\widetilde v(X_1)\right)
+\left( \widetilde v(X_1)+\widetilde v(X)\right)\widetilde v(X+X_1)\Bigg\}dX_1.
\end{align*}
(cf. \cite{m, EPV}). The same equation appears also as classical limit of the description of a weakly Bose gas in presence of a condensate when only the collisions involving the condensate are kept (\cite{Eckern},\cite{Kirkpatrick},\cite{LLPR}). The variable $X$ is then related to the energy of the particles.
 
The singularity of the kernel $M$ along the line $X=Y$ indicates that $\mathscr L$ must have have some regularizing effects. We  are then interested in the  regularity properties of  linearized problem,
\begin{align}
&\frac {\partial v(t, X  )} {\partial t}=\mathscr L(v(t))(X )+\nu (t, X),\,t>0, \, X>0 \label{S2Ewxi2L}\\
&v(0, X)=0,\,\,X>0. \label{S2Ewxi2Lb}
\end{align} 
{\bf {\large 1.1}}
The  Cauchy problem for (\ref{S2Ewxi2L}) with $\nu \equiv 0$ and initial data $v_0$ has been studied in \cite{m},   using  the variables $x=X^{1/2}$ and $u(t, x)=v(t, X)$ so that the equation reads,
\begin{align}
&\frac {\partial u (t, x)} {\partial t}= L(u (t))(x),\,\,t>0, x>0,\label{EL}\\
&L(u )(x)\int _0^\infty (u(y)-u(  x)) K(x, y) dy, \label{S3EMK2}\\
&K(x, y)=\left(\frac {1} {|x^2-y^2|}-\frac {1} {x^2+y^2} \right)\frac {y} {x},\,\,\forall x>0,\,\,\forall y>0,\,\,x\not =y.\label{S3E2359B}
 \end{align}
The existence of a fundamental solution $\Lambda(t, x)$ for (\ref{EL}) was shown in \cite{m}. It was also proved  that for all  initial data  $u _0\in L^1(0, \infty)\cap L^\infty _{ loc }(0, \infty)$ a solution $u \in C((0, \infty); L^1(0, \infty))\cap L^\infty((0, \infty); L^1(0, \infty))$ exists,  that satisfies $u _t(t, x)=L(u(t))(x)$ for $t>0$, $x>0$ and $u (t)\rightharpoonup v_0$ in $\mathscr D'(0, \infty)$ as $t\to 0$. These results apply  of course to equation (\ref{S2Ewxi2L}) with $\nu=0$,  using the change of variables $X=x^2$.

It is  proved in Section \ref{S7.1App} that the homogeneous equation (\ref{EL})is also well posed in the spaces $X _{ \theta, \rho  }$ defined, for $\theta\in \RR$ and $\rho \in \RR$, as
\begin{align*}
&X _{ \theta, \rho  }=\left\{ g \in C(0, \infty);\,\,||g|| _{ \theta, \rho  }<\infty \right\}\\
&||g|| _{ \theta, \rho  }=\sup _{ X>0 }X^{\theta}(1+X)^{\rho }|g(X)|.
\end{align*}
 and has a regularizing effect at the origin (cf. Corollary \ref{S7cor1} in Section \ref{S7.1App}  below). 
These properties are then used to solve the problem  (\ref{S2Ewxi2L}), (\ref{S2Ewxi2Lb}) and prove the following,
\begin{theo}
\label{S7cor2}
Suppose that  $\nu \in C  ((0, T);  X _{\theta, \rho })\cap  L^\infty((0, T); X _{ \theta, \rho  })$ for some $T>0$,  $ \theta\ge 0$, $\rho >0$ such that $\theta+\rho \in (0, 3/2)$ and  consider the function $v$ defined as
\begin{align}
\label{E1.10}
v(t, X)=\int _0^t (\mathscr S(t-s)\nu(s))(X)ds,\,\forall t\in (0, T),\,\forall X>0.
\end{align}
Then, for all $t\in (0, T)$ and $\theta'$, $\rho '$ satisfying (\ref{thetaP}),
\begin{align*}
&(i)\quad v\in L^\infty((0, T); X _{ \theta, \rho  })\cap C((0, T); X _{ \theta', \rho ' })\\
&\qquad\left|\frac {\partial v} {\partial t} \right|+\left| \mathscr L(v)\right|\in L^\infty _{ \text{loc} }((0, \infty)\times (0, \infty))\\
&\hskip 1cm ||v(t)|| _{ \theta, \rho  } \le Ct\,\sup _{ 0\le s\le t  }||\nu(s)|| _{ X _{ \theta, \rho  }},\,\forall t>0,\\
&(ii)\,\, \text{the function}\,\,v\,\text{satisfies (\ref{S2Ewxi2L}) for }\,\, t\in (0, T),\,X>0.
\end{align*}
\begin{align*}
&(iii)\quad \forall t\in (0, T),\,s\in (0, t),\, \exists L(t-s, \nu(s))\in \RR,\, \text{such that for}\, p\in (-1, 1/2), \\
&\qquad \int _0^{t-X^{1/2}} L(t-s; \nu (s)) ds\le  C\sup _{ 0\le s\le t }||\nu(s)|| _{ p, \rho  }t^{1-2p},\, \forall X\in (0, t^2),\nonumber\\
& \qquad \int _{t-X^{1/2}}^t L(t-s; \nu (s)) ds\le C\sup _{ 0\le s\le t }||\nu(s)|| _{ \theta, \rho  }X^{\frac {1} {2}-\theta},\, \forall X\in (0, t^2).\\
&(iv)\quad \text{For},\,t>0,\,\, X\in (0, \min(1,t^2))\\
&\hskip 1cm  v(t, X)=\int _0^{t } L(t-s; \nu (s))ds+R_3(t, \nu, X),\nonumber\\
&\hskip 1cm |R_3(t, \nu, X)|\le  C\sup _{ 0\le s\le t } ||\nu (s)|| _{ \theta, \rho  }X^{1/2}  \left( X^{- \theta }+t^{1-2\theta}\right)\\
&(v) \quad  \quad |v(t, X)|\le C\sup _{ 0\le s\le t } ||\nu (s)|| _{ \theta, \rho  }
X^{-3/2}t^{4-2\theta}(1+t)^{-2\rho },\,\forall X>t^2, t\in (0,  T), \nonumber
\end{align*}
\end{theo}
Theorem \ref{S7cor2} follows  from the properties of $\Lambda$ (the fundamental solution of  (\ref{EL})), using  similar arguments to those in \cite{m}.\\

\noindent
{\bf {\large 1.2.}}
On the other hand,  under the change of variables
\begin{align}
\label{S1ChV}
v(t, X)=w(t, \xi ),\,\,\nu(t, X)=Q(t, \xi ),\,\,\xi =\log X
\end{align}
the problem  (\ref{S2Ewxi2L}), (\ref{S2Ewxi2Lb}) is  transformed into
\begin{align}
&\frac {\partial w (t, \xi   )} {\partial t}=P(w (t))(\xi  )+Q (t, \xi ),\,\, t>0,\,\,\xi \in \RR, \label{S2Ewxi2Lc}\\
&w(0, \xi )=0,\,\xi \in \RR \label{S2Ewxi2Ld}
\end{align} 
 where $Q(t, \xi )=\nu(t, X)$ and,
\begin{align*}
&P(w)(\xi )=\int  _{ \RR }\widehat w(k)e^{ik\xi }p(\xi , k)dk,\,\,\forall \xi \in \RR,\,\forall k\in \RR,,\\
&p(\xi , k)=-e^{-\frac {\xi} {2} }\rho_0 (k)\\
&\rho _0(k)=\frac {1} {2} \left( \text{\rm Log}(4)+ \Psi {\left(\frac {1} {2} +\frac {ik} {2}\right) }+ \Psi \left(1+\frac {ik} {2}\right) +2\gamma _E\right),
\end{align*}
where $ \Psi $  is the Digamma function and $\gamma _E$ denotes the Euler's Gamma constant. 

The presence of the pseudo differential operator $P$ suggests to study the regularizing effects of problem (\ref{S2Ewxi2Lc}), (\ref{S2Ewxi2Ld}) using the classical method of freezing coefficients, and then to translate the result to  problem (\ref{S2Ewxi2L}), (\ref{S2Ewxi2Lb}), using the homogeneity properties of the linear operator $\mathscr L$.

We denote $H^\sigma (\RR)$ and $H^\sigma (I)$  the classical Sobolev spaces on $\RR$ and any open bounded interval $I$ respectively.
The properties of $\rho _0$ (cf. Proposition \ref{S2PP1} in Section \ref{Ppseudo}) suggest  the use of the  following spaces  (cf.  \cite{Ka}),
\begin{align*}
&H^\sigma _{ \log } (\RR)=\left\{w,\,\,\text{measurable on}\, \RR;\, ||w|| _{ H^\sigma _{ \log }  }<\infty \right\}\\
& ||w||^2 _{ H^\sigma  _{ \log } (\RR)}=\int  _{ \RR }|\widehat w(k)|^2(1+|k|^2)^\sigma (1+\log(1+|k|))dk,
\end{align*}
\begin{align*}
&H^\sigma _{ \log^{-1} } (\RR)=\left\{w,\,\,\text{measurable on}\, \RR;\, ||w|| _{ H^\sigma _{ \log^{-1} }  }<\infty \right\}\\
& ||w||^2 _{ H^\sigma  _{ \log^{-1} } (\RR)}=\int  _{ \RR }|\widehat w(k)|^2(1+|k|^2)^\sigma (1+\log(1+|k|))^{-1}dk.
\end{align*}
As proved in  Theorem 1.4 of \cite{Brue}, when $\sigma =0$,
\begin{align}
&H^0 _{ \log } (\RR)=\left\{w,\,\,\text{measurable on}\, \RR;\, ||w|| _{ H^0 _{ \log }  }<\infty \right\} \label{S1EHl1}\\
& ||w||^2 _{ H^0 _{ \log } (\RR) }= ||w|| _{ L^2(\RR) }^2+\int  _{ B _{ 1/3 } }\int  _{ \RR }\frac {|w(\xi +h)-w(\xi )|^2} {|h|}d\xi dh\label{S1EHl2}.
\end{align}
and then, for $\sigma =m\in \NN$, the norm $ ||w||^2 _{ H^m  _{ \log } (\RR)}$ is equivalent to
\begin{align*}
||w||^2 _{ H^m(\RR)}+\int  _{ B _{ 1/3 } }\int  _{ \RR }\frac {|w^{(m)}(\xi +h)-w^{(m)}(\xi )|^2} {|h|}d\xi dh
\end{align*}
Notice also,  for $\sigma =0$,
\begin{align*}
L^2(\RR)\subset H^0 _{ \log^{-1} }(\RR),\,\,\text{and}\,\,\,\,||w|| _{ H^0 _{ \log ^{-1} }(\RR) }\le ||w|| _{ L^2(\RR) },\,\,\,\forall w \in L^2(\RR).
\end{align*}
The corresponding local spaces, $H^\sigma  _{ \log }(I)$ for open, bounded intervals $I\subset \RR$ may then be defined in the usual way, only involving values of the function on $I$ (cf. \cite{Ka}).\\
Using the  change of variables (\ref{S1ChV}), we define the spaces $M _{ \sigma  }(\RR), M _{ \sigma , \log }(\RR)$, $M _{ \sigma , \log^{-1} }(\RR)$,
\begin{align}
&v\in M _{ \sigma  }(0, \infty)\Longleftrightarrow  w\in H^\sigma (\RR),\,\,\,||v|| _{ M _{ \sigma  } (0, \infty)}=||w|| _{ H^\sigma  (\RR)}\label{S1EMH}\\
&v\in M _{ \sigma, \log   }(0, \infty)\Longleftrightarrow  w\in H^\sigma _{ \log }(\RR),\,\, ||v|| _{ M _{ \sigma, \log   } (0, \infty)}=||w|| _{ H^\sigma  _{ \log } (\RR)}\label{S1EMHl}\\
&v\in M _{ \sigma, \log^{-1}  }(0, \infty)\Longleftrightarrow  w\in H^\sigma _{ \log^{-1}} (\RR),\,\,||v|| _{ M _{ \sigma, \log^{-1}   }(0, \infty) }=
||w|| _{ H^\sigma  _{ \log^{-1} }(\RR) }.\label{S1EMH}
\end{align}
Given an open bounded  intervals  $J\subset (0, \infty)$ and $I\subset \RR$ let us denote
\begin{align*}
\log J=\left\{\xi \in \RR;  \exists X\in J:\,\,\xi =\log X \right\}\\
\exp(I)=\left\{X>0;  \exists \xi \in I:\,\, X=e^\xi  \right\}
\end{align*}
and define the local  spaces,  $M _{ \sigma  }(I), M _{ m , \log }(I)$  for $m\in \{0, 1, 2, \cdots \}$ as follows
\begin{align}
&v\in M _{ \sigma  }(J)\Longleftrightarrow  w\in H^\sigma (\log J),\,\,\,||v|| _{ M _{ \sigma  } (J)}=||w|| _{ H^\sigma  ( \log J)}\label{S1EMHl}\\
&v\in M _{ m, \log   }( J)\Longleftrightarrow  w\in H^m _{ \log }(\log J),\,\, ||v|| _{ M _{ m, \log   } ( J)}=||w|| _{ H^m  _{ \log } (\log J)}. \label{S1EMHll}
\end{align}
It is proved in Section  \ref{back} that for $\sigma \in (0, 2)$,  for all open bounded interval $J\subset \overline J\subset (0, \infty)$, the space $M_\sigma (J)$ coincides with the usual Sobolev space $H^\sigma (J)$ in the sense of equivalence of the norms.

In general we shall denote $H^\sigma , H^\sigma  _{ \log },  H^\sigma  _{ \log^{-1} }$, $||\cdot|| _{ H^\sigma  }, ||\cdot|| _{ H^\sigma _{ \log }},   ||\cdot|| _{ H^\sigma _{ \log^{-1} }},$ the spaces $H^\sigma (\RR)$, $H^\sigma  _{ \log }(\RR), H^\sigma  _{ \log^{-1} }(\RR)$ and their respective norms. We denote the norms of the Lebesgue's spaces $L^p$, on $\RR$ or an interval $I$, as $||\cdot||_p$.

Let us denote, 
 $I_1=(1/8, 4)$, $I_2=(1/2, 2)$, $I_3=(5/8, 11/8)$ and  $I_4=(3/4, 5/4)$,
and define
\begin{align*}
&\eta _0(X )=
\begin{cases}
1,\,\,X \in I_3,\,\,\\
0,\,\,\,X \not \in I_2,
\end{cases}\\
&\eta _{ 0, R }(X)=\eta_0(X/R),\,\,\forall R>0.
\end{align*}
Our second result is on the smoothing effects of equation (\ref{S2Ewxi2L}) in terms of the  Sobolev norms defined above.
\begin{theo}
\label{mtheo}
Let $\theta\in (0, 1/4)$ and $\rho $ be such that $\theta+\rho\in (1/2, 3/2)$ and suppose that
\begin{align}
&\sup _ {\substack{ 0< t_0 < T^*\\ 0<R<1}}  \left( R   \int  _{ t_0 }^{\min(t_0+R^{1/2}, T_*)}||\chi _R \nu (t)||^2 _{ M _{ \sigma , \log^{-1} }   }dt  \right)^{1/2}+\nonumber\\
&+ \sup _{ R>1 }   \left( R \int  _{ 0 } ^{T_* } \left|\left |\chi _R \nu(t ))\right|\right|^2 _{  M _{ \sigma , \log^{-1} } }dt  \right)^{1/2}
+ \int _0^{T_*} t^{-4\theta }||\nu(t)|| _{ \theta, \rho  }^2dt
<\infty.  \label{mtheoE1}
\end{align}
Then, the function $v$ defined by  (\ref{E1.10}) satisfies,
\begin{align}
\hskip -1cm &(a)\quad (i) \sup _ {\substack{ 0< t_0 < T^*\\ 0<R<1}}\Bigg(\int  _{ t_0}^{\min(t_0+R^{1/2}, T_*)} ||v(t)||^2 _{ M_\sigma (RJ_3)}dt \Bigg)^{1/2} \le  C \left(1+T_*^{2(1-\theta)}\right)\times \nonumber\\
&\times \left( \int _0^{T_*}(1+t^{-4\theta}) ||\nu(t)|| ^2_{ \theta, \rho  } dt\right)^{1/2}\!\!\!\!\!\!\! + C
\!\!\!\!\!\!\!  \sup _ {\substack{ 0< t_0 < T^*\\ 0<R<1}}  \left( R   \int  _{ t_0 }^{\min(t_0+R^{1/2}, T_*)}\!\!\!\!||\eta _{ 0, R } \nu (t)||^2 _{ M _{ \sigma , \log^{-1} }   }dt  \right)^{1/2}, \label{mtheoE2a}
\end{align}
\begin{align}
\qquad & (ii) \,\,\sup _{ R>1 } \Bigg(\int _0^{ T_*} ||v(t)||^2 _{ M_\sigma (RJ_3) }dt \Bigg)^{1/2}\le  C \left(1+T_*^{2(1-2\theta)}\right)\times \nonumber \\
&\times \left( \int _0^{T_*}(1+t^{-4\theta}) ||\nu(t)|| ^2_{ \theta, \rho  } dt\right)^{1/2} +C \sup _{ R>1 }   \left( R \int  _{ 0 } ^{T_* } \left|\left |\eta _{ 0, R } \nu(t ))\right|\right|^2 _{  M _{ \sigma , \log^{-1} } }dt  \right)^{1/2}. \label{mtheoE2b}
\end{align}
\begin{align}
(b) & \sup _ {\substack{ 0< t_0 < T^*\\ 0<R<1}}\Bigg(\int  _{ t_0}^{\min(t_0+R^{1/2}, T_*)}  (N _{ R, \sigma  }[v(t) ])^2dt \Bigg)^{1/2} \le C\left(1+T_*^{2(1-\theta)}\right)\times \nonumber\\
& \times \left(\int _0^{T_*} \left(1+t^{-4\theta}\right) ||\nu(t)|| ^2_{ \theta, \rho  } dt\right)^{1/2}\!\!\!+C\!\!\! \sup _ {\substack{ 0< t_0 < T^*\\ 0<R<1}}  \left( R \int  _{ t_0 }^{\min(t_0+R^{1/2}, T_*)}\!\!\!\!\!\!\! ||\chi _R \nu (t)||^2 _{ M _{ \sigma } (I_2)  }dt  \right)^{1/2}   \label{mtheoE2c}\\ 
&\text{with}:\,(N _{ R, \sigma  }[v(t) ])^2=\frac {1} {R}\int  _{ \RR }\Big| \mathcal M\left( \eta_0\, v_R(t)\Big)(-ik)\right|^2(1+|k|^2)^\sigma \times\nonumber\\
&\hskip 7cm \times \left(1 _{ |k|<1 }+\1 _{ |k|>1 }\rho_0 (k)\right)dk \nonumber,
\end{align}
where $\mathcal  M$ denotes the Mellin transform, $v_R(t, x)=v(t, Rx)$ for $x>0$.\\
As a particular case, for $\sigma =0$,
 \begin{align}
& \sup _ {\substack{ 0< t_0 < T^*\\ 0<R<1}}\Bigg(\int  _{ t_0}^{\min(t_0+R^{1/2}, T_*)} \hskip -1cm \big[R^{-1}||v(t)||^2 _{ L^2(RJ_3) }+R^{-2}[[v (t )]]^2 _{RJ_3 }\big]dt \Bigg)^{1/2} \le  C\left(1+T_*^{2(1-2\theta)}\right)\times 
\nonumber \\
&\times \left(\int _0^{T_*} \left(1+t^{-4\theta}\right) ||\nu(t)|| ^2_{ \theta, \rho  } dt\right)^{1/2}
+C \sup _ {\substack{ 0< t_0 < T^*\\ 0<R<1}}  \left(  \int  _{ t_0 }^{\min(t_0+R^{1/2}, T_*)}||  \nu (t)||^2 _{ L^2(RJ_2)}dt  \right)^{1/2}. \label{sgfrd1}
\end{align}
where $[[v (t )]]^2 _{RJ_3 }=\iint  _{J_3^2 }\frac {|v(X)-v(Y)|^2} {|X-Y|}dXdY$.
\end{theo}
\noindent
For the condition $\theta\in (0, 1/4)$ see however Remark \ref{theta14}.
\begin{rem}
Property (b) in Theorem \ref{mtheo} shows that $\eta_0v_R$ belongs to some generalised Liouville space of functions $h$ such that $ \mathcal M (h)(-ik)(1+|k|^2)^{\sigma /2} (1+\mathscr Re(\rho _0(k)\1 _{ |k|>1 })\in L^2(\RR)$. No general inner description of generalized Liouville spaces is known  for $\sigma\not \in \NN$, and so only for $\sigma =0$, use of (\ref{S1EHl2}) allows to describe the regularizing effect in terms of the $H^0 _{ \log }(RJ_3)$ local norm.
\end{rem}

Theorem \ref{S7cor2} is proved in Section \ref{S7.1App}.
The deduction of some properties of the operator $P$ are given in Section \ref{Ppseudo}. The auxiliary result describing the smoothing effect of equation (\ref{S2Ewxi2L2}) are proved in Section \ref{thm3.1}. Theorem \ref{mtheo} is proved in Section \ref{back}.  Section \ref{somelemmas} contains some technical but necessary results.

Smoothing effects in  a kinetic equation strongly  related to  (\ref{S2Ewxi2L}) are considered  in \cite{BL}.

\section{Proof of Theorem \ref{S7cor2}.}
\label{S7.1App}
\setcounter{equation}{0}
\setcounter{theo}{0}
Under the change of variables $x=\sqrt X$ and $u(t, x)=v(t, X)$, and up to some trivial time rescaling to get rid of some constants,
the homogeneous equation (\ref{S2Ewxi2L}) with $Q=0$ reads
\begin{align}
&\frac {\partial u} {\partial \tau }(\tau , x)=L(u(t))(x)=\int _0^\infty (u(\tau , y)-u(\tau , x)) K(x, y) dy,\,\,\tau >0,\,x>0 \label{S3EMK2}\\
&K(x, y)=\left(\frac {1} {|x^2-y^2|}-\frac {1} {x^2+y^2} \right)\frac {y} {x},\,\,\forall x>0,\,\,\forall y>0,\,\,x\not =y. \label{S3E2359B}
\end{align}
The Cauchy problem for (\ref{S3EMK2}) was considered in \cite{m}. A  fundamental solution $\Lambda(t, x)$ was obtained, unique under the condition that its Mellin transform $\mathcal M(\Lambda(t))$  is analytic in the strip $\mathscr R e(s)\in (0, 2)$.

It was shown that for all initial data $g\in L^1(0, \infty)$ the function,
\begin{align}
\label{lpgn1}
u(t, x)=(S(t)g)(x)=\int _0^\infty \Lambda\left(\frac {t} {y}, \frac {x} {y} \right)g(y)\frac {dy} {y}, \,\,t>0,\,x>0
\end{align}
 satisfies $u\in L^\infty((0, \infty); L^1(0, \infty))\cap C((0, \infty); L^1(0, \infty))$, 
is a weak solution of (\ref{S3EMK2}) and $u(t)\to g$ in $\mathscr D'(0, \infty)$ as $t\to 0$. It was also proved that if  $g\in L^1(0, \infty)\cap L _{ loc }^\infty(0, \infty)$ then $u$ satisfies the equation   (\ref{S3EMK2})   pointwise  for $t>0$ and $x>0$.\\
We consider here the Cauchy problem for  (\ref{S3EMK2})  for initial data in the spaces $X _{ \theta, \rho  }$Notice that $X _{ \theta, \rho  }\subset L^1(0, \infty)$ (with continuous injection)   if $\theta<1$ and $\rho +\theta>1$, but $X _{ \theta, \rho  }\not \subset L^1(0, \infty)$ in general in all the other cases. On the other hand
 \begin{prop}
 \label{Xapprox}
 Suppose that $g\in X _{ \theta, \rho  }$ for some $\theta\in \RR, \rho \in \RR$ and define $g_n$ as $g_n(x)=g(x)\1 _{ \frac {1} {n}<x<1 }+g(x)\1 _{1< x<n }$. Then,
 \begin{align*}
 \lim _{ n\to \infty  }||g-g_n|| _{ \theta', \rho ' }=0,\,\,\text{for all}\,\,\,\theta' >\theta,\,\rho '<\rho +\theta-\theta'.
 \end{align*}
 \end{prop}
 
Using the estimates on $\Lambda$ the fundamental solution of (\ref{EL}) obtained in \cite{m}
it is not difficult to prove that, under suitable conditions on $\theta$ and $\rho $,  $S(t)$ is a continuous linear map from $X _{ \theta, \rho  }$ into itself   with norm independent of $t$. More precisely,

\begin{lem}
\label{lem1}
Suppose that   $\theta\ge 0$,  $\rho\ge 0 $ are such that  $0\le \theta +\rho  <3$. Then, there exists a constant $C>0$ such that for all $g\in X _{ \theta, \rho  }$ 
\begin{align}
\label{S7P141E}
|S(t)g(x)|\le C||g|| _{ \theta , \rho   } x^{-\theta }(1+x)^{-\rho } ,\,\,\forall t>0\,\,\,\forall x>0.
\end{align}
For all $\theta', \rho '$ such that
\begin{align}
\label{thetaP}
\theta' >\theta,\, \rho '\in [0, \rho +\theta-\theta'),\,\,0\le \theta'+\rho '<3
\end{align}
the map $t\to S(t)g$ belongs to $C((0, \infty); X _{ \theta', \rho ' })$.
\end{lem}
\begin{proof}
By definition,
\begin{align*}
|S(t)g(x)|\le \int _0^\infty  \left|\Lambda\left( \frac {t} {y}, \frac{x}{y}\right)g(y)\right|\frac {dy} {y}=I_1+I_2
\end{align*}
In the first term $I_1$ $t/y>1$ and then, by Proposition 3.1 and Proposition 3.2 of \cite{m},
\begin{align*}
\left| \Lambda\left(\frac {t} {y}, \frac {x} {y} \right)\right|\le C\max(t, x)^{-3}y^3
\end{align*}
 and then, since $3-\theta-\rho >0$,
 \begin{align*}
|I_1|&\le C \max(t, x)^{-3}\int _0^t|g(y)|y^2dy\le C\max(t, x)^{-3}||g|| _{ \theta,\rho   }\int _0^t y^{-\theta }(1+y  )^{-\rho }y^2dy\nonumber \\
&\le C\max(t, x)^{-3}||g|| _{ \theta ,  \rho  }\int _0^t y^{2-\theta} (1+y  )^{-\rho }dy \le C\max(t, x)^{-3}||g|| _{ \theta ,  \rho  }t^{3-\theta}(1+t)^{-\rho }.
\end{align*} 
Since for all $x\in (0, t)$,
\begin{align*}
\max(t, x)^{-3}t^{3-\theta }(1+t )^{-\rho }=\frac {t^{-\theta }} {(1+t)^\rho }\le x^{-\theta }(1+x   )^{-\rho }
\end{align*}
and for $x>t$,
\begin{align}
 \label{pgcd590}
\max(t, x)^{-3}t^{3-\theta }(1+t )^{-\rho }&=\frac {x^{-3}t^{3-\theta } }{(1+t)^\rho  }\\
&\le
\begin{cases}
\displaystyle{\frac {2} {(1+t)^\rho  }}x^{-\theta }(1+x   )^{-\rho }\,\,\hbox{if}\,\,0<x<1\nonumber\\
\displaystyle{\frac {1} {(1+t)^\rho  }}x^{-\theta }(1+x  )^{-\rho },\,\,\hbox{if}\,\,x>1
\end{cases}\le \frac {C} {1+t}x^{-\theta }(1+x  )^{-\rho }\nonumber
\end{align}
it follows that
\begin{align}
\label{ppcm52}
|I_1|\le C||g|| _{ \theta, \rho  } x^{-\theta }(1+x )^{-\rho }\,\,\forall t>0,\,\,\forall x>0.
\end{align}
In  the term $I_2$ one has $y>t$.  If $x<t$ then by Proposition  3.5 in \cite{m},
\begin{align}
&\left|\Lambda\left(\frac {t} {y}, \frac {x} {y} \right)\right|\le Cxt^5y^{-6}\nonumber\\
&|I_2|\le Cxt^5\int _t^\infty |g(y)|y^{-7}dy\le  C||g|| _{ \theta, \rho  }xt^5\int _t^\infty y^{-\theta-7}(1+y  )^{-\rho }dy\nonumber \\
&\le C||g|| _{ \theta , \rho  }xt^5 t^{-\theta -6}( 1+t )^{-\rho }\le C||g|| _{ \theta, \rho  } x t^{-1-\theta }(1+t  )^{-\rho }\nonumber \\
&|I_2|\le C||g|| _{ \theta, \rho  } x^{-\theta }(1+x  )^{-\rho },\,\forall x\in (0, t). \label{SAL2EX}
\end{align}
On the other hand, when $y>t$ if $x>t$  there exists $\delta >0$ small enough such that $x>(1+\delta )t$. The term $I_2$ must then be split,
\begin{align}
\label{S5E29BX}
I_2=\int _t^{\frac {x} {1+\delta }}[\cdots]\frac {dy} {y}+\int _{\frac {x} {1+\delta }}^{\frac {x} {1-\delta }}[\cdots]\frac {dy} {y}+\int _{\frac {x} {1-\delta }}^\infty[\cdots]\frac {dy} {y}=I _{ 2,1 }+I _{ 2,2 }+I _{ 2,3 }
\end{align}
Since $\frac {x} {y}>1+\delta $ in $I _{ 2, 1 }$ and $\frac {x} {y}<1-\delta $ in $I _{ 2, 3 }$, by Proposition 3.5 in \cite{m} for all $\varepsilon >0$ as small as desired there exists a constant $C_\varepsilon >0$ such that, for $y\in (t, x/(1+\delta ))$ and $y>x/(1-\delta )$,
\begin{align*}
\left|\Lambda\left(\frac {t} {y}, \frac {x} {y} \right)\right|\le C_\varepsilon \left(x^{-3+\varepsilon }t^{9-\varepsilon }y^{-6}+x^{-5}t^7y^{-2} \right)
\end{align*}
and then, since $-\theta-\rho <1$,
\begin{align}
|I _{ 2, 1 }|+|I _{ 2, 3 }|&\le  Cx^{-3+\varepsilon } t^{9-\varepsilon } \int _t^\infty y^{-7}|g(y)|dy+
Cx^{-5}t^7\int _t^\infty y^{-3}|g(y)|dy\nonumber\\
&\le C||g|| _{ \theta, \rho   }x^{-3+\varepsilon } t^{9-\varepsilon } \int _t^\infty y^{-7}y^{-\theta}(1+y )^{-^{\rho}}dy+\nonumber \\
&\qquad+C||g|| _{ \theta, \rho  }x^{-5}t^7\int _t^\infty y^{-3}y^{-\theta}(1+y )^{-\rho }dy\nonumber\\
&\le C||g|| _{ ^{\rho}, \rho  }\left(x^{-3+\varepsilon }t^{3-\varepsilon -\theta}+x^{-5}t^{5-\theta } \right)(1+t  )^{-\rho }.\label{pgcd31416}
\end{align}
Since,
\begin{align*}
x^{-3+\varepsilon }t^{3-\varepsilon -\theta }
&\le 
\begin{cases}
C \displaystyle{\frac {1} {(1+t)^\rho }}x^{-\theta}(1+x )^{-\rho }\,\,\hbox{if},\,\,x<1\\
C \displaystyle{\frac {t^\rho } {(1+t)^\rho }}x^{-\theta}(1+x )^{-\rho }\,\,\hbox{if},\,\,x>1
\end{cases} \le C x^{-\theta }(1+x )^{-\rho } 
\end{align*}
and by the same argument
\begin{align*}
x^{-5+\varepsilon }t^{5-\varepsilon -\theta }\le C x^{-\theta }(1+x^\rho )^{-1}, 
\end{align*}
it follows using (\ref{SAL2EX}) too,
\begin{align}
\label{SAL2E0BX}
 |I_1|+|I _{ 2, 1 }|+|I _{ 2, 3 }|\le C||g|| _{ \theta, \rho  }x^{-\theta }(1+x )^{-\rho },\,\,\forall t>0,\,\,\forall x>0.
\end{align}

In the second term $I _{ 2,2 }$ at the right hand side of (\ref{S5E29BX}) the regularising effect of $S(t)$ must be used.With the change of variables $z=x/y$,
\begin{align*}
I _{ 2, 2 }=\int  _{ 1-\delta  }^{1+\delta }\Lambda\left(\frac {tz} {x}, z \right)g\left(\frac {x} {z}\right) \frac {dz} {z}
\end{align*}
In order to use Corollary 3.13 in \cite{m} write now $I _{ 2,2 }$ as follows,
\begin{align*}
I _{ 2, 2 }=\int  _{ 1-\delta  }^ {1+\delta }\frac {tz} {x}|1-z|^{\frac {2tz} {x}-1}\psi (t, z; x)g(x/z)\frac {dz}{z}.
\end{align*}
where, for all $x>0$, 
\begin{align*}
&\psi (t, z; x)=\left(\Lambda\left( \frac {tz} {x}, z\right)
\frac {x} {tz}|1-z|^{1-\frac {2tz} {x}}\right).
\end{align*}
 By Corollary 3.13  and Corollary 3.14 in \cite{m} 
\begin{align*}
&(t, z)\mapsto \psi (t, z; x) \,\,\hbox{belongs to}\,\, C([0, 1]\times [1-\delta , 1+\delta ] ).\\
&C_*(\psi , \delta )=\sup\Big\{|\psi (t, z; x)|; x>t, z\in (1-\delta , 1+\delta ) \Big\}<\infty
\end{align*}
and then,
\begin{align*}
|I _{ 2, 2 }|\le C_*(\psi , \delta )\int  _{ 1-\delta  }^ {1+\delta }\frac {tz} {x}|1-z|^{\frac {2tz} {x}-1}|g(x/z)|\frac {dz}{z}.
\end{align*}

Under the change of variables $y=2\frac {t} {x}\log(1-z)$, $z=1-e^{\frac {xy} {2t}}$, $dy=-2tdz/(x(1-z))$
\begin{align*}
\int  _{ 1-\delta  }^ 1&\frac {tz} {x}|1-z|^{\frac {2tz} {x}-1}\left|g\left( \frac {x} {z}\right)\right|\frac {dz}{z}\le\\
&\le\frac {1} {2(1-\delta )}\int  _{ -\infty }^ {\frac {2t\log \delta } {x}} e^{ y\left(1-e^{\frac {xy} {2t}} \right)}  \left|g\left( \frac {x} {1-e^{\frac {xy} {2t}}}\right)\right|dy=J_1+J_2
\end{align*}
where,
\begin{align*}
&J_1=\frac {1} {2(1-\delta )}\int  _{ -\infty }^ {\frac {2t\log \delta } {x}} e^{ y\left(1-e^{\frac {xy} {2t}} \right)}\left| g_<\left( \frac {x} {1-e^{\frac {xy} {2t}}}\right)\right|dy\\
&J_2=\frac {1} {2(1-\delta )}\int  _{ -\infty }^ {\frac {2t\log \delta } {x}} e^{y\left(1-e^{\frac {xy} {2t}} \right)}\left| g_>\left( \frac {x} {1-e^{\frac {xy} {2t}}}\right)\right|dy\\
&g_<(z)=g(z)\1 _{ z<1 },\,\,g_>(z)=g(z)\1 _{ z>1 }.
\end{align*}
Since $||g|| _{ \theta, \rho  }<\infty$, 
Because condition $x<1-e^{\frac {xy} {2t}}$ for some $y\in (-\infty, 2t\log \delta /x)$ requires $x<1$, it follows that $J_1(t, x)=0$ for $x>1$. On the other hand,  if  $y\in (-\infty, 2t\log \delta /x)$ then $e^{\frac {xy} {2t}}<\delta $ and so condition $x>1-e^{\frac {xy} {2t}}$ requires   $x>1-\delta >1/2$ from where $J_2(t, x)=0$ for all $x\in (0, 1/2)$.
For $0<t<x<1$, 
\begin{align}
|J_1|&\le \frac {  ||g|| _{ \theta, \rho  }x^{-\theta}} {2(1-\delta )}\int  _{ -\infty }^ {\frac {2t\log \delta } {x}} e^{ y\left(1-e^{\frac {xy} {2t}} \right)}\left(1-e^{\frac {xy} {2t}}\right)^{\theta }dy\nonumber\\
&\le C  ||g|| _{ \theta, \rho  } x^{-\theta } \int  _{ -\infty }^ {\frac {2t\log \delta } {x}} e^{ y\left(1-e^{\frac {xy} {2t}} \right)}  dy\le C  ||g|| _{ \theta, \rho  }x^{-\theta }  \int  _{ -\infty }^ {\frac {2t\log \delta } {x}} e^{ y\left(1- \delta  \right)}dy\nonumber\\
&\le C  ||g|| _{ \theta, \rho  } x^{-\theta}e^{-\frac {2(-1+\delta )t\log \delta } {x}}\le C  ||g|| _{ \theta, \rho  } x^{-\theta}. \label{SAL2E0CX}
\end{align}
When $x>1/2$,

\begin{align}
|J_2|&\le \frac {  ||g|| _{ \theta, \rho  }} {2(1-\delta )x^{\theta+\rho }}\int  _{ -\infty }^ {\frac {2t\log \delta } {x}} e^{ y\left(1-e^{\frac {xy} {2t}} \right)}\left(1-e^{\frac {xy} {2t}}\right) ^{\theta+\rho }dy\nonumber\\
&\le  C    ||g|| _{ \theta, \rho }x^ {-\theta-\rho }\int  _{ -\infty }^ {\frac {2t\log \delta } {x}} e^{ y\left(1-e^{\frac {xy} {2t}} \right)}   dy 
\le C  ||g|| _{ \theta, \rho  }x^{-\theta-\rho  }  \int  _{ -\infty }^ {\frac {2t\log \delta } {x}} e^{ y\left(1- \delta  \right)}   dy\nonumber\\
&\le C   ||g|| _{ \theta, \rho  } x^{-\theta-\rho }e^{-\frac {2(-1+\delta )t\log \delta } {x}} \le C   ||g|| _{ \theta, \rho  } x^{-\theta-\rho }.\label{SAL2E0DX}
\end{align}
It follows that for $x>t$ ,
\begin{align*}
|I _{ 2,2 }|\le C ||g|| _{\theta, \rho }  x^{-\theta}(1+x )^{-\rho },
\end{align*}
and then, by (\ref{pgcd31416}),
\begin{align}
\label{ppcm78}
 \left( |I_1|+|I_2|\right)\le C ||g|| _{ \theta, \rho  } x^{-\theta  }(1+x )^{-\rho },\,\forall x>t.
\end{align}
By (\ref{ppcm78}), (\ref{ppcm52}) and (\ref{SAL2EX}),
\begin{align}
\label{abcP1}
\left(|I_1|+|I_2|\right)\le C||g|| _{\theta, \rho  } x^{-\theta}(1+x )^{-\rho },\,\,\forall t>0,\,\forall x>0.
\end{align}

On the other hand for all $\theta', \rho '$ satisfying the hypothesis,
\begin{align*}
&||S(t_1)g-S(t_2)g|| _{ \theta', \rho ' }\le 
||S(t_1)g-S(t_1)g_n|| _{ \theta', \rho ' }+\\
&+||S(t_1)g_n-S(t_2)g_n|| _{ \theta', \rho ' }+||S(t_2)g_n-S(t_2)g|| _{ \theta', \rho ' }\\
&\le ||g-g_n|| _{ \theta', \rho ' }
+||S(t_1)g_n-S(t_2)g_n|| _{ \theta', \rho ' }+||g_n-g|| _{ \theta', \rho ' }
\end{align*}
from where, by Proposition \ref{Xapprox} the map $t\to S(t)g$ belongs to $C((0, \infty); X _{ \theta', \rho ' })$.
 \end{proof}
\begin{rem}
If  $\theta<3$ and $\theta+\rho >3$ then  $I_1(t, x)\le C\max(t, x)^{-3}||g|| _{ \theta, \rho  }$.
\end{rem}

\begin{prop}
\label{ptdm1}
Suppose that $g\in X _{ \theta, \rho  }$ with   $\theta\ge 0$,  $\rho\ge 0 $ are such that  $0\le \theta +\rho  <3$ and denote $u$ the function defined for all $t>0$ as $u(t)=S(t)g$. \\
(i) Then $u\in L^\infty (0, \infty; X _{ \theta, \rho  })\cap C((0, \infty); X _{ \theta', \rho ' })$ for $\theta', \rho '$ satisfying (\ref{thetaP}) and
\begin{align}
&\left|\frac {\partial u} {\partial t}\right|+|  L(u)|\in 
 L ^\infty_{ \text{loc} }((0, \infty)\times (0, \infty))  \label{ptdm1E1}\\
&\frac {\partial u(t, x)} {\partial t}=   L(u(t))(x),\,\,t>0, x>0. \label{ptdm1E01}\\
&\forall x>0,\,\,\lim _{ t\to 0 }u(t, x)=g(x).\label{ptdm1E0}
\end{align}
(ii) If $f\in C(0, T); X _{ \theta, \rho  })\cap L^\infty((0, T); X _{ \theta, \rho  })$ then the function $u$ defined as

\begin{align}
u(t)=S(t)g+\int _0^tS(t-s)f(s)ds \label{ptdm1E19}
\end{align}
is such that $u\in C ((0, T); X _{ \theta', \rho'  })$ for $\theta'$, $\rho '$ satisfying (\ref{thetaP}), it satisfies (\ref{ptdm1E1}) and
\begin{align}
\frac {\partial u(t, x)} {\partial t}=   L(u(t))(x)+f(t, x),\,\,t>0, x>0.\label{ptdm1E24}
\end{align}
\end{prop}
\begin{proof} (i)  If $\theta\in [0, 1)$ and $\rho >1$ this follows as Theorem 1.4 of \cite{m}.  The argument of the proof is that, for each $t_0>0$ and small open interval $I$ containing $t_0$ it is possible to find a  function $H(x, z)$, integrable with respect to $z$   that dominates $\left|\frac {\partial \Lambda} {\partial t} \left( \frac {t} {z}, \frac {x} {z}\right)\frac {f_0(z)} {z^2}\right|$ for each $x>0$,  uniformly for $t\in I$. For $\theta$ and $\rho $ satisfying the hypothesis the same proof shows that there exists a positive constant $C$ such that for $t>0$, $x>0$,
\begin{align}
&\left|\frac {\partial u(t, x)} {\partial t} \right|\le 
Cx^{-4}\int _0^{2t}|g(z)|z^2dz+Cx^{-3}t^6\left(\int  _{ 2t }^\infty \frac {|g(z)|dz} {z^3}\right)\1 _{ x>3t }+\nonumber\\
&+Cxt^4\left(\int  _{ 2t }^\infty \frac {|g(z)|dz} {z^7}\right)\1 _{ 0<x<\frac {2t}{3} }+
C\left(\sup _{ z\in (2t, 3x) }|g(z)|\right)\1 _{ \frac {2t} {3}<x<3t }.\label{blb}
\end{align}
The proof of (\ref{ptdm1E0}) requires some further results from \cite{m}.
In order to see this in some detail write  (\ref{lpgn1}) as follows,  using the change of variables $x/y=z$ and $a>0$ arbitrary,
\begin{align}
\label{S3ESPM1}
S(t)g(x)=\int  _{ 0 }^ {x/a}\Lambda\left( \frac {tz} {x}, z\right)g(x/z)\frac {dz} {z}+ \int  _{x/a}^\infty \Lambda\left( \frac {tz} {x}, z\right)g(x/z)\frac {dz} {z}=P_1+P_2.
\end{align}
Both terms $P_i$ are treated with similar arguments. Consider first $P_1$. For $x<a$ then $z$ is away from $z=1$ and two cases are possible. If $t>a$ then $x/t<x/a$ and
\begin{align}
\label{S3ESPM2}
|P_1|\le C||g|| _{\theta, \rho  } \left(\int  _{ 0 }^ {x/t}\left|\Lambda\left( \frac {tz} {x}, z\right)\right| p (x/z)\frac {dz} {z}
+\int  _{x/t}^{x/a}\left|\Lambda\left( \frac {tz} {x}, z\right)\right|p (x/z)\frac {dz} {z}\right)
\end{align}
where we have denoted
$$
p(x)=x^{-\theta}(1+x)^{-\rho }.
$$
In the first term at the right hand side of (\ref{S3ESPM2}) $tz/x<1$ and in the second  $tz/x<1$ while $x/t>1$ in both. By Proposition 3.5 of \cite{m}, using the change of variables $x/z=y$, for $\varepsilon >0$ as small as desired there is a constant $C_\varepsilon >0$ such that,
\begin{align}
&\int  _{ 0 }^ {x/t}\left|\Lambda\left( \frac {tz} {x}, z\right)\right|p (x/z)\frac {dz} {z}\le C_\varepsilon \int _0^{x/t}
\left(z^{-3+\varepsilon }\left( \frac {tz} {x}\right)^{9-\varepsilon }+z^{-5}\left( \frac {tz} {x}\right)^7 \right)p (x/z)\frac {dz} {z} \nonumber\\
&\!\le C_\varepsilon \Bigg( \frac{t^{9-\varepsilon }}{x^{3-\varepsilon }}\int _t^\infty \frac {y^{-7}dy} {y^\theta+y^{\theta+\rho }} +\frac{t^7}{x^{5}}\int _t^\infty  \frac {y^{-3}dy} {y^\theta+y^{\theta+\rho }}
\Bigg) \le C\left(\frac{ t^{3-\varepsilon -\theta}}{x^{3-\varepsilon }}+\frac{t^{5-\theta}}{x^{5}}\right)
\le C\frac{ t^{3-\varepsilon -\theta}}{x^{3-\varepsilon }}
\label{S3ESPM3}
\end{align}
and
\begin{align}
\int  _{x/t}^{x/a}\left|\Lambda\left( \frac {tz} {x}, z\right)\right|p (x/z)\frac {dz} {z}\le C\int  _{ x/t }^{x/a}\frac {z^{-4}dz} {(x/z)^{\theta}}
\le Cx^{-3}t^{3-\theta}.
\label{S3ESPM4}
\end{align}
Then by (\ref{S3ESPM2})--(\ref{S3ESPM4}),
\begin{align}
|P_1|\le C_\varepsilon ||g|| _{\theta, \rho  } t^{3-\varepsilon -\theta}x^{-3+\varepsilon },\,\,\forall (t, x);\,0<x<a<t.
\end{align}
If $t\in (0, a)$, then $tz/x<1$ for all $z\in (0, x/a)$ and by  Proposition 3.5, if $x>t$,

\begin{align}
 |P_1|&\le C\int _0^{x/a} \left(z^{-3+\varepsilon }\left( \frac {tz} {x}\right)^{9-\varepsilon }+z^{-5}\left( \frac {tz} {x}\right)^7 \right)p(x/z)\frac {dz} {z} \nonumber\\
&\le \int _0^{x/t} \left(z^{-3+\varepsilon }\left( \frac {tz} {x}\right)^{9-\varepsilon }+z^{-5}\left( \frac {tz} {x}\right)^7 \right)\frac {dz} {z ( x/z)^\theta}
\le C_\varepsilon \frac{ t^{3-\varepsilon -\theta}}{x^{3-\varepsilon }},
\label{S3ESPM5}
\end{align}
and,
\begin{align}
 |P_1|\le C\int _0^{x/a} z \left( \frac {tz} {x}\right)^5
 \frac {dz} {z (x/z)^\theta }\le Ca^{-6-\theta}t^5x,\,\,\forall x\in (0, t).
\label{S3ESPM6}
\end{align}
From (\ref{S3ESPM5}), (\ref{S3ESPM6}),
\begin{align}
 |P_1|\le C\left(   t^{3-\varepsilon -\theta}x^{3-\varepsilon }\1 _{a> x>t }+ a^{-6-\theta}t^5x\1 _{ x<t <a}\right).
\label{S3ESPM7}
\end{align}
Similar arguments show for $x>a$,
\begin{align}
\label{S3ESPM8}
|P_2|\le C\left( t^{3-\theta}x^{-5}+x^{-3+\varepsilon }t^{3+\theta} \1 _{ t<a<x } \right).
\end{align}

The cases $x>a$ in $P_1$ and $x<a$ in $P_2$ must now be treated. Consider first $x>a$ in $P_1$.
Since $a>0$ and $t\to 0$,  it may be assumed that $0<t<a$, from where $tz/x<t/a<1$. If $ x>a$ then $1\in (0, x/a)$ and
\begin{align*}
P_1=\int  _{0 }^ {x/a}\Lambda\left( \frac {tz} {x}, z\right)g(x/z)\frac {dz}{z}=\int  _{ 0 }^ {1-\delta }+\int  _{ 1-\delta  }^ {1+\delta }+\int _{ 1+\delta  }^{x/a}
\end{align*}
By Proposition 3.5 of \cite{m}, arguing as  for (\ref {S3ESPM3}) or (\ref{S3ESPM5}), 
\begin{align}
 \int _{0}^{1-\delta }[\cdots]dz+\int _{1+\delta }^{ x/a}[\cdots]dz\le  C_\varepsilon ||g|| _{\theta, \rho }\Big( t^{ 7} x^{-7 -\theta}+t^7x^{-5}a^{-2-\theta}\Big)\le 
\frac{  C_\varepsilon ||g|| _{\theta, \rho }t^7}{x^{5}a^{2+\theta} }\label{S5EWWW3}
\end{align}
Consider next,
\begin{align*}
\int  _{ 1-\delta  }^{1+\delta }\Lambda\left( \frac {tz} {x}, z\right)g(x/z)\frac {dz}{z}
=\int  _{ 1-\delta  }^ {1+\delta }\frac {tz} {x}|1-z|^{\frac {2tz} {x}-1}\psi (t, z; x)g(x/z)\frac {dz}{z}.
\end{align*}
with, for all $x>0$, 
\begin{align*}
&\psi (t, z; x)=\left(\Lambda\left( \frac {tz} {x}, z\right)
\frac {x} {tz}|1-z|^{1-\frac {2tz} {x}}\right)\\
&(t, z)\mapsto \psi (t, z; x) \,\,\hbox{belongs to}\,\, C((0, 1)\times (1-\delta , 1+\delta ) )
\end{align*}
Under the change of variables $y=2t\log(1-z)$, $z=1-e^{\frac {y} {2t}}$
\begin{align*}
\int  _{ 1-\delta  }^ 1\frac {tz} {x}|1-z|^{\frac {2tz} {x}-1}&\psi (t, z; x)g(x/z)\frac {dz}{z}=\\
&=\frac {1} {2x}\int  _{ -\infty }^ {2t\log \delta } e^{\frac {y} {x}\left(1-e^{\frac {y} {2t}} \right)}\psi (t, 1-e^{\frac {y} {2t}}; x) g\left( \frac {x} {1-e^{\frac {y} {2t}}}\right)dy.
\end{align*}
Similarly, under the change  $y=2t\log(z-1)$, $z=1+e^{\frac {y} {2t}}$
\begin{align*}
\int  _1^{ 1+\delta  }\frac {tz} {x}|1-z|^{\frac {2tz} {x}-1}&\psi (t, z; x)g(x/z)\frac {dz}{z}=\\
&=\frac {1} {2 x}\int  _{ -\infty }^ {2t\log \delta } e^{\frac {y} {x}\left(1+e^{\frac {y} {2t}} \right)}\psi (t, 1+e^{\frac {y} {2t}}; x) g\left( \frac {x} {1+e^{\frac {y} {2t}}}\right)dy.
\end{align*}
Let us write,
\begin{align*}
&\frac {1} {2x}\int  _{ -\infty }^ {2t\log \delta } e^{\frac {y} {x}\left(1-e^{\frac {y} {2t}} \right)}\psi (t, 1-e^{\frac {y} {2t}}; x) g\left( \frac {x} {1-e^{\frac {y} {2t}}}\right)dy-\frac {g(x)} {2}=\\
&\qquad =\frac {1} {2x}\int  _{ -\infty }^ {0} e^{\frac {y} {x}\left(1-e^{\frac {y} {2t}} \right)}\psi (t, 1-e^{\frac {y} {2t}}; x) g\left( \frac {x} {1-e^{\frac {y} {2t}}}\right)dy-\frac {g(x)} {2}+\\
& \hskip 2.4cm +\frac {1} {2x}\int  _ {2t\log \delta }^0 e^{\frac {y} {x}\left(1-e^{\frac {y} {2t}} \right)}\psi (t, 1-e^{\frac {y} {2t}}; x) g\left( \frac {x} {1-e^{\frac {y} {2t}}}\right)dy\\
&\qquad=\frac {1} {2}\left(\int  _{ -\infty }^ {0} \frac {1} {x}e^{\frac {y} {x}\left(1-e^{\frac {y} {2t}} \right)} dy-1\right)g(x)+\\
& \hskip 2.4cm+\frac {1} {2}\int  _{ -\infty }^ {0} \frac {1} {x}e^{\frac {y} {x}\left(1-e^{\frac {y} {2t}} \right)}
\left(\psi (t, 1-e^{\frac {y} {2t}}; x) g\left( \frac {x} {1-e^{\frac {y} {2t}}}\right) -g(x)\right)dy+\\
& \hskip 2.4cm+\frac {1} {2x}\int  _ {2t\log \delta }^0 e^{\frac {y} {x}\left(1-e^{\frac {y} {2t}} \right)}\psi (t, 1-e^{\frac {y} {2t}}; x) g\left( \frac {x} {1-e^{\frac {y} {2t}}}\right)dy
\end{align*}
and finally,
\begin{align*}
&\frac {1} {2x}\int  _{ -\infty }^ {2t\log \delta } e^{\frac {y} {x}\left(1-e^{\frac {y} {2t}} \right)}\psi (t, 1-e^{\frac {y} {2t}}; x) g\left( \frac {x} {1-e^{\frac {y} {2t}}}\right)dy-\frac {g(x)} {2}=I_1+I_2+I_3+I_4\\
&I_1=\frac {1} {2}\left(\int  _{ -\infty }^ {0} \frac {1} {x}e^{\frac {y} {x}\left(1-e^{\frac {y} {2t}} \right)} dy-1\right)g(x)\\
&I_2=\frac {1} {2}\int  _{ -\infty }^ {0} \frac {1} {x}e^{\frac {y} {x}\left(1-e^{\frac {y} {2t}} \right)}
\left(\psi (t, 1-e^{\frac {y} {2t}}; x)-1\right) g\left( \frac {x} {1-e^{\frac {y} {2t}}}\right)dy
\end{align*}
\begin{align*}
&I_3=\frac {1} {2}\int  _{ -\infty }^ {0} \frac {1} {x}e^{\frac {y} {x}\left(1-e^{\frac {y} {2t}} \right)}
 \left(g\left( \frac {x} {1-e^{\frac {y} {2t}}}\right) -g(x)\right)dy\\
&I_4=\frac {1} {2x}\int  _ {2t\log \delta }^0 e^{\frac {y} {x}\left(1-e^{\frac {y} {2t}} \right)}\psi (t, 1-e^{\frac {y} {2t}}; x) g\left( \frac {x} {1-e^{\frac {y} {2t}}}\right)dy.
\end{align*}
In the first term, the change of variables $y=x\, z$ gives
\begin{align*}
x|I_1(t, x)|\le C||g|| _{\theta, \rho  } \left| \int _{-\infty }^0 e^{z\left( 1-e^{\frac {x z} {2t}}\right)}dz-1\right|.
\end{align*}
Since
\begin{align*}
e^{z\left( 1-e^{\frac {x z} {2t}}\right)}>e^z,\,\,\forall z<0\Longrightarrow   \int _{-\infty }^0 e^{z\left( 1-e^{\frac {x z} {2t}}\right)}dz> \int _{-\infty }^0 e^{z}dz=1,
\end{align*}
it follows,
\begin{align*}
x|I_1(t, x)|\le C||g|| _{\theta, \rho  }\left( \int _{-\infty }^0 e^{z\left( 1-e^{\frac {x z} {2t}}\right)}dz-1\right).
\end{align*}
Since $x>a$, it follows, $e^{-ze^{\frac { xz} {2t}}}<e^{-ze^{\frac {az} {2t}}}$
and,
\begin{align*}
x|I_1(t, x)|\le C||g|| _{\theta, \rho  }\left( \int _{-\infty }^0 e^{z\left( 1-e^{\frac {a z} {2t}}\right)}dz-1\right).
\end{align*}
In the second term $I_2$, using that $p $ is a decreasing function,

\begin{align*}
\left|\frac {1} {2}\int  _{ -\infty }^ {0} \frac {1} {x}e^{\frac {y} {x}\left(1-e^{\frac {y} {2t}} \right)}
\left(\psi (t, 1-e^{\frac {y} {2t}}; x)-1\right) g\left( \frac {x} {1-e^{\frac {y} {2t}}}\right)dy\right|\le\\
\le C||g|| _{\theta, \rho  }
\int  _{ -\infty }^ {0} \frac {1} {x}e^{\frac {y} {x}\left(1-e^{\frac {y} {2t}} \right)}
\left|\psi (t, 1-e^{\frac {y} {2t}}; x)-1\right| p\left( \frac {x} {e^{\frac {y} {2t}}}\right)dy\\
\le  C||g|| _{\theta, \rho  }p (x) \int  _{ -\infty }^ {0} e^{z\left(1-e^{\frac {xz} {2t}} \right)}
\left|\psi (t, 1-e^{\frac {xz} {2t}}; x)-1\right| dz\\
\le  C||g|| _{\theta, \rho  } p (x)\int  _{ -\infty }^ {0}e^{z\left(1-e^{\frac {az} {2t}} \right)}
\left|\psi (t, 1-e^{\frac {xz} {2t}}; x)-1\right|dz
\end{align*}
By Corollary 3.13 in \cite{m}
\begin{align}
&\lim _{ t\to 0 }\, t^{-1}\left|e^{-1/t} Y\right|^{1-2t}\Lambda \left(t, 1+e^{-1/t}Y\right)=1 \label{S3Cor17E1}
\end{align}
uniformly  for $Y$ on bounded subsets of $\RR$. Since $x>a$, when $t\to 0$,  $0<tz/xtz/a\to 0$ uniformly for $x>a$, and $z\in (1-\delta , 1)$ and
\begin{align*}
\lim _{ t\to 0 }\sup _{ z\in (1-\delta (t), 0) }\left|\Lambda\left( \frac {tz} {x}, z\right)
\frac {x} {tz}|1-z|^{1-\frac {2tz} {x}}-1\right|=0.
\end{align*}
In the  term $I_3$, by the continuity of $g$ at $x$,
\begin{align*}
\lim _{ t\to 0 }|I_3(t, x)|=0.
\end{align*}
In the last term $I_4$, using again that $\rho $ is decreasing
\begin{align*}
&|I_4|\le \frac {1} {2x}\int  _ {2t\log \delta }^0 e^{\frac {y} {x}\left(1-e^{\frac {y} {2t}} \right)}\psi (t, 1-e^{\frac {y} {2t}}; x) \left|g\left( \frac {x} {1-e^{\frac {y} {2t}}}\right)\right|dy.\\
&\le C  ||g|| _{\theta, \rho  } \frac {1} {x}\int  _ {2t\log \delta }^0 e^{\frac {y} {x}\left(1-e^{\frac {y} {2t}} \right)}|\psi (t, 1-e^{\frac {y} {2t}}; x)|p \left( \frac {x} {e^{\frac {y} {2t}}}\right)dy\\
&\le C  ||g|| _{\theta, \rho  } p (x) \frac {1} {x}\int  _ {2t\log \delta }^0 e^{\frac {y} {x}\left(1-e^{\frac {y} {2t}} \right)}dy =
C||g|| _{\theta, \rho  }p (x)\int  _ {\frac {2t\log \delta} {x} }^0 e^{z\left(1-e^{\frac {xz} {2t}} \right)}dz
\end{align*}
and since $\frac {2t\log \delta} {x} > \frac {2t\log \delta} {a} $, $
|I_4|\le  ||g|| _{\theta, \rho  }\displaystyle{\int  _ {\frac {2t\log \delta} {a} }^0 e^{z\left(1-e^{\frac {az} {2t}} \right)}dz}$.
All this shows,
\begin{align}
\label{S3Em24}
\lim _{ t\to 0 } x\left|\int  _{ 1-\delta  }^ 1\frac {tz} {x}|1-z|^{\frac {2tz} {x}-1}\psi (t, z; x)g(x/z)\frac {dz}{z}-\frac {g(x)} {2}
 \right|=0.
\end{align}
A similar argument shows,
\begin{align}
\label{S3Em25}
\lim _{ t\to 0 } \, x\left| \int _1^{ 1+\delta  }\frac {tz} {x}|1-z|^{\frac {2tz} {x}-1}\psi (t, z; x)g(x/z)\frac {dz}{z}-\frac {g(x)} {2}
 \right|=0
\end{align}
and then, for $x>a$,
\begin{align}
\label{S3ESPM20}
\lim _{ t\to 0 } \left|S(t)g(x)-g(x) \right|=0.
\end{align}
The term  $P_2$ when $x<a$  is treated in the same way as $P_1$ for $x>a$, writing first,
\begin{align}
\label{S3Eq234}
|P_2|\le \int  _{x/a}^\infty \left|\Lambda\left( \frac {tz} {x}, z\right)g(x/z)\right|\frac {dz}{z}=
\int  _{ x/a }^ {1-\delta }[\cdots]dz+\int  _{ 1-\delta  }^ {1+\delta }[\cdots]dz+\int _{ 1+\delta  }^{\infty}[\cdots]dz.
\end{align}
In the first  and third integrals, two cases arise depending on whether $x>t$ or  $x<t$. In both cases, use of  Proposition 3.1, Proposition 3.2 and Proposition 3.5 of \cite{m}  yields the existence of a constant $C>0$ such that,
\begin{align}
\int  _{ x/a }^ {1-\delta }[\cdots]dz+\int _{ 1+\delta  }^{\infty}[\cdots]dz\le C ||g|| _{ X_\theta } \left(t^{-\theta}\1 _{ 0<x<t }+x^{-3}t^{3-\theta}\1 _{ t<x<a }
\right),\,\,0<x<a. \label{S3ESPM9}
\end{align}
The estimate of the second integral in the right hand side of (\ref{S3Eq234})  follows from the same arguments that lead to  (\ref{S3Em24}), (\ref{S3Em25}) and give,
\begin{align}
\lim _{ t\to 0 }p (x)\int  _{ 1-\delta  }^{1+\delta }[\cdots]dz=0. \label{S3ESPM12}
\end{align}
It follows from (\ref{S3ESPM9})--(\ref{S3ESPM12}) that for $x\in (0, a)$,
\begin{align}
\label{S3ESPM21}
\lim _{ t\to 0 } \left|S(t)g(x)-g(x) \right|=0,
\end{align}
and then (\ref{ptdm1E0}) follows from (\ref{S3ESPM20}) and (\ref{S3ESPM21}).

Suppose now that $f\in C(0, T; X _{ \theta, \rho  })$ and denote
$$
m(t, x)= \int _0^\infty S(t-s)f(s)(x) dx,\,\,t\in (0, T), x>0.
$$
Then, for $|h|>0$,
\begin{align}
\frac {1} {h}(m(t+h)-m(t))=\frac {1} {h}\int _t^{t+h}S(t+h-s)f(s)(x)dx+\nonumber\\
+\frac {1} {h}\int _0^t  \Big(S(t+h-s)f(s)(x)-S(t-s)f(s)(x)\Big) dx  \label{blb2}
\end{align}

In the second term at the right hand side of  (\ref{blb2}), for all   $s\in (0, t)$ and $x>0$,
\begin{align*}
\lim _{ h\to 0 }\frac {1} {h}\ \Big(S(t+h-s)f(s)(x)-S(t-s)f(s)(x)\Big) dx
= \frac {\partial } {\partial t}\bigg(S(t-s)f(s)\bigg)(x).
\end{align*}

Estimate  (\ref{blb}) gives now a suitable dominant function to apply the Lebesgue's convergence Theorem. To this end write first,
\begin{align*}
&\left|\left( \frac {\partial } {\partial t} S(t-s)f(s)(x)\right) \right|\le 
C\Bigg(x^{-4}\int _0^{2(t-s)}|f(s, z)|z^2dz+\\
&+Cx^{-3}(t-s)^6\left(\int  _{ 2(t-s) }^\infty \frac {|f(s, z)|dz} {z^3}\right)\1 _{ x>3(t-s) }+Cx(t-s)^4\left(\int  _{ 2(t-s) }^\infty \frac {|f(s, z)|dz} {z^7}\right)\1 _{ 0<x<\frac {2(t-s)}{3} }+\\
&\hskip 7.4cm +C\left(\sup _{ z\in (2(t-s), 3x) }|f(s, z)|\right)\1 _{ \frac {2(t-s)} {3}<x<3(t-s) }\Bigg).
\end{align*}
Since $t\in (0, T)$, $s\in (0, t)$ and $g\in X _{ \theta, \rho  }$,
\begin{align*}
\int _0^{2(t-s)}|f(s, z)|z^2dz\le C\sup _{ 0\le s\le T }||f(s)|| _{ \theta, \rho  }\int _0^{2T}z^{-\theta+2}dz,
\end{align*}
\begin{align*}
(t-s)^6\left(\int  _{ 2(t-s) }^\infty \frac {|f(s, z)|dz} {z^3}\right)\le 
\begin{cases}
C\sup _{ 0\le s\le T }||f(s)|| _{ \theta, \rho  }((t-s)^{6}+(t-s)^{4-\theta})\le\\
\le C\sup _{ 0\le s\le T }||f(s)|| _{ \theta, \rho  }T^{4-\theta},\,\, (t-s)<1,\\
C\sup _{ 0\le s\le T }||f(s)|| _{ \theta, \rho  }(t-s)^{4-\theta-\rho }\le\\
\le C\sup _{ 0\le s\le T }||f(s)|| _{ \theta, \rho  }T^{4-\theta-\rho},\,\,(t-s)>1,
\end{cases}
\end{align*}
\begin{align*}
x(t-s)^4\left(\int  _{ 2(t-s) }^\infty \frac {|f(s, z)|dz} {z^7}\right)\1 _{ 0<x<\frac {2(t-s)} {3} }&\le  C\sup _{ 0\le s\le T }||f(s)|| _{ \theta, \rho  }\, x(t-s)^{-2-\theta}\1 _{ 0<x<\frac {2(t-s)} {3} }\\\
&\le  C\sup _{ 0\le s\le T }||f(s)|| _{ \theta, \rho  }\, x^{-\theta-\rho -1}.
\end{align*}
All these estimates show that there exists a non negative function  $\psi \in L^1(0, \infty)$ such that 
\begin{align}
\label{blbpsi}
\left|\left( \frac {\partial } {\partial t} S(t-s)f(s)(x)\right)\right|\le \psi (x),\,\,\forall x>0,\,\forall s\in (0, t).
\end{align}
By Lebesgue's convergence theorem it then follows, for $t\in (0, T)$ and $x>0$,
\begin{align*}
\lim _{ h\to 0 }\frac {1} {h}\int _0^t \Big(S(t+h-s)f(s)(x)-S(t-s)f(s)(x)\Big) ds=\int _0^\infty \frac {\partial } {\partial t} S(t-s)f(s)(x)ds.
\end{align*}
By property (\ref{ptdm1E01}) already proved, 
\begin{align*}
\int _0^\infty \frac {\partial } {\partial t} S(t-s)f(s)(x)ds=\int _0^\infty \mathscr L(S(t-s)f(s))(x)ds
=\mathscr L\left(\int _0^\infty S(t-s)f(s)ds\right)(x).
\end{align*}

In order to pass to the limit as $h\to 0$ in the first term at the right hand side of  (\ref{blb2}) let $C>0$ be the constant given by (\ref{S7P141E}) in Lemma \ref{lem1}.  Since $f\in C([0, T); X _{ \theta, \rho  })$, for all $t\in (0, T)$ and $\varepsilon >0$ there exists $\delta >0$ there exists $\delta >0$ such that
\begin{align*}
&|s-t|<\delta \Longrightarrow ||f(s)-f(t)|| _{ \theta, \rho  }<\varepsilon C^{-1} x_0^\theta(1+x_0)^\rho.
\end{align*}
Then by Lemma \ref{lem1}, for all $s\in (t, t+h)$ with $|h|<\delta $,
\begin{align*}
\Longrightarrow ||S(t+h-s)(f(s)-f(t))|| _{ \theta, \rho  }\le||f(s)-f(t)|| _{ \theta,  \rho  } <\varepsilon  x_0^\theta(1+x_0)^\rho
\end{align*}
and in particular,
\begin{align*}
 |S(t+h-s)(f(s)-f( t))(x_0)|\le \varepsilon.
\end{align*}
On the other hand, since  $f(s)\in X _{ \theta, \rho  }\subset C(0, \infty)$ for all $s\in (0, T)$, by property  (\ref{ptdm1E0})  already proved, for each $x_0>0$ and $\varepsilon >0$ there exists $\delta' >0$ such that,
\begin{align*}
|h|<\delta' \Longrightarrow |S(t+h-s)f(t)(x_0)-f(t, x_0)|<\varepsilon,\,\text{for all}\, s\, \text{between}\, t, \text{and}\, t+h,
\end{align*}
then, for $|h|<\min(\delta , \delta') $ and  $s$ between $t$ and $t+h$,
\begin{align*}
&\left| (S(t+h-s)f(s))(x_0)-f(t) \right|\le  |S(t+h-s)(f(s)-f( t))(x_0)|+\\
&\hskip 5cm +|S(t+h-s)f(t)(x_0)-f(t, x_0)|\le 2\varepsilon ,\\
&\left|\frac {1} {h} \int _t^{t+h} (S(t+h-s)f(s))(x)ds-f(t, x_0)\right|\le 2\varepsilon, 
\end{align*}
and 
\begin{align*}
\lim _{ h\to 0 }\frac {1} {h}\int _t^{t+h}S(t+h-s)f(s)(x)dx=f(t, x_0)
\end{align*}
It then follows that the function $u$ defined in (\ref{ptdm1E19}) satisfies (\ref{ptdm1E24})
\end{proof}

The map $g\to S(t)g$ has  some regularizing effect at the origin as shown in the next result.
\begin{prop} 
\label{S7P14}
There exists a constant $B_1>0$ such that, if for some $\theta>0$, $\rho  >0$, $0\le \theta+\rho<3$,  $g\in X _{ \theta, \rho  }$ then for  $\delta>0$ small, $t>0$, $x>0$:
\begin{align}
&|S (t)g(x)|\le C_\delta ||g|| _{ \theta, \rho  } t^{-\theta}(1+t)^{-\rho } \left(1+\frac {x} {t}+x\right),\,x\in (0, t),\label{S7P14D1}\\
&|S (t)g(x)|\le  C||g|| _{ \theta, \rho  }t^{-\theta} (1+t)^{-\rho } \left(\frac {t} {x}\right)^3\,0<t<x.\label{S7P14D1B}
\end{align}
Moreover,
\begin{align}
&S (t)g)(x )-\ell (t; g)=R_0(t, g, x),\,x\to 0 \label{S7P14E1}\\
&|R_0(t, g, x)| \le C||g|| _{ \theta, \rho  } t^{-\theta}(1+t)^{-\rho }  \left(\frac {x} {t}+\left(\frac {x} {t}\right)^{1+\delta } +x\right), 0<x<\min (1, t)  \label{S7P14E1B} \\
&\ell (t; g)=\frac {6 B(1)}{\pi ^2}\frac {1} {2i\pi }\int  _{ \mathscr Re(r )=\beta  }\frac {\Gamma(r )} {B(r )}t^{-r }
\left(\int _0^{t} g(\zeta )\zeta ^{-1+ r }d\zeta  \right)dr,\,\,\beta \in (0, 2). \label{S7P14E2}
\end{align}
For all $T>0$ and all $p>-2$, there exists $C>0$ such that,
\begin{align}
\label{S7P14E3}
\left|\ell (t; g) \right| \le C||g|| _{ p, \rho }t^{-p},\,\,\forall t\in (0, T).
\end{align}
If $g(x)=\alpha x^{-\theta}(1+\mathcal O(x)^\varepsilon )$ as $x \to 0$ for some $\varepsilon >0$, then,
\begin{align}
\label{S7P14E5}
\ell (t; g)=\frac {12B(1)\Gamma (\theta)} {\pi ^2B(\theta)}t^{-\theta}+\mathcal O(t)^{-\theta+\varepsilon },\,\, t\to 0.
\end{align}
\end{prop}
\begin{proof}  In order to prove   (\ref{S7P14D1}) and  (\ref{S7P14D1B}) we write, as given in \cite{m},
\begin{align}
(S(t)g)(x)&=\int _0^{t} g(y)\Lambda\left(\frac {t} {y}, \frac {x} {y} \right)\frac {dy} {y}+\int _{t}^\infty g(y)\Lambda\left(\frac {t} {y}, \frac {x} {y} \right)\frac {dy} {y}\nonumber \\
&=\frac {1} {2\pi i}\int  _{  \mathscr Re s=c }x^{-s}
\int _0^{t} g(y)U\left(\frac {t} {y}, s \right)y^{s-1}dyds+\int _{t}^\infty g(y)\Lambda\left(\frac {t} {y}, \frac {x} {y} \right)\frac {dy} {y}  \label{S7P14E2A0}
\end{align}
where $c\in (0, 2)$. We start considering the second term in the right hand side of  (\ref{S7P14E2A0}) where $y>t$. If we assume  $x<t$  then, by Proposition 3.5 in \cite{m},
\begin{align*}
\left|\int _{t}^\infty g(y)\Lambda\left(\frac {t} {y}, \frac {x} {y} \right)\frac {dy} {y}\right|\le Cxt^5\int  _{ t }^\infty |g(y)|\frac {dy} {y^6}\le C||g|| _{ \theta, \rho  }xt^{5}t^{-\theta-5}(1+t)^{-\rho }
\end{align*}
If  $0<t<x$, then by  Proposition 3.5 of \cite{m}, for all $\varepsilon >0$ as small as desired, and  $0<t<x, y>t$,
\begin{align*}
|\Lambda \left(\frac {t} {y}, \frac {x} {y}\right)|\le Cx^{-3+\varepsilon }t^{9-\varepsilon }y^{-6}+Cx^{-5}t^7y^{-2}
\end{align*}
from where for $0<t<x$,
\begin{align*}
&\left|\int _{t}^\infty g(y)\Lambda\left(\frac {t} {y}, \frac {x} {y} \right)\frac {dy} {y}\right|\le Ct^{-\theta}(1+t)^{-\rho }\left(\left(\frac {t} {x} \right)^3+C\left(\frac {t} {x} \right)^5\right)
\end{align*}
It follows that
\begin{align}
\label{S7P14E2D2}
&\left|\int _{t}^\infty g(y)\Lambda\left(\frac {t} {y}, \frac {x} {y} \right)\frac {dy} {y}\right|\le Ct^{-\theta}(1+t)^{-\rho }
\min (1, x) \left(\frac {t} {x} \right)^3
\end{align}
On the other hand, in the first term at the right hand side of (\ref{S7P14E2A0}),  for all $0<x<t$, $y>0$ the function $U(t/y, \cdot)$ is analytic for $\mathscr Re s\in (-1, 0)\cup(0, 2)$ and has  simple poles at $s=-1$ and $s=0$. Then for $x$ small and $\delta>0$ small,

\begin{align}
&\frac {1} {2\pi i} \int  _{  \mathscr Re s=c }x^{-s}
\int _0^{t} g(y)U\left(\frac {t} {y}, s \right)y^{s-1}dyds=\int _0^t g(y)y^{-1} \text{Res} \left(U\left(\frac {t} {y}, s\right); s=0  \right)dy+
\label{S7P14E2A}\\
&+x\int _0^t g(y)y^{-2} \text{Res} \left(U\left(\frac {t} {y}, s\right); s=-1  \right)dy+\frac {1} {2\pi i}\int  _{  \mathscr Re s=-1-\delta  }x^{-s}
\int _0^t g(y)U\left(\frac {t} {y}, s \right)y^{s-1}dyds. \nonumber
\end{align}
In the last term at the right hand side of (\ref{S7P14E2A}),  by Proposition 2.10 in \cite{m},
\begin{align*}
&\Big| \frac {1} {2\pi i}\int  _{  \mathscr Re s=-1-\delta  }x^{-s}
\int _0^t g(y)U\left(\frac {t} {y}, s \right)y^{s-1}dyds \Big|\le \nonumber\\
&\le Cx^{1+\delta }\int _{-\infty}^\infty ((1+\delta) ^2+v^2)^{-1} \int _0^t g(y)((1+\delta) ^2+v^2)^{-\frac {t} {y}}y^{-2-\delta }dydv\\
&\le Cx^{1+\delta}  t^{-1-\delta}   \int _0^1 g(t z)(1+\delta)^{-\frac {2} {z}}z^{-2-\delta }dz
\end{align*}
If $0<t<1$,
\begin{align*}
&\le Cx^{1+\delta} t^{-1-\delta-\theta}||g|| _{ \theta, \rho  } \int _0^1(1+ \delta)^{-\frac {2} {z}}z^{-\theta-\delta -2}dy=C||g|| _{ \theta, \rho  }x^{1+\delta} t^{-1-\delta-\theta}.
\end{align*}
but if $t>1$,
\begin{align*}
\int _0^1 g(t z)(1+\delta)^{-\frac {2} {z}}z^{-2-\delta }dz\le 
Ct^{-\theta} \int _0^{1/t} (1+\delta)^{-\frac {2} {z}}z^{-2-\delta-\theta }dz+\\
+Ct^{-\theta-\rho } \int _{1/t} ^1(1+\delta)^{-\frac {2} {z}}z^{-2-\delta-\theta-\rho  }dz.
\end{align*}
Use of Mathematica gives,
\begin{align*}
 \int _0^{1/t} (1+\delta)^{-\frac {2} {z}}z^{-2-\delta-\theta }dz=\Gamma \left(1+\delta +\theta, t\log(1+\delta ) \right)(\log(1+\delta ))^{-1-\delta -\theta}\\
 \le C  (1+\delta )^{-t}(\log(1+\delta ))^{-1}t^{\delta +\theta}\le Ce^{-\varepsilon t},\,\,t>1.
\end{align*}
for some $\varepsilon >0$ small. Since, 
\begin{align*}
 \int _{1/t} ^1(1+\delta)^{-\frac {2} {z}}z^{-2-\delta-\theta-\rho  }dz\le  \int _{0} ^1(1+\delta)^{-\frac {2} {z}}z^{-2-\delta-\theta-\rho  }dz\le C
\end{align*}
it follows that for $t>1$,
\begin{align*}
&\Big| \frac {1} {2\pi i}\int  _{  \mathscr Re s=-1-\delta  }x^{-s}
\int _0^t g(y)U\left(\frac {t} {y}, s \right)y^{s-1}dyds \Big|\le
C||g|| _{ \theta, \rho  }x^{1+\delta} t^{-1-\delta-\theta-\rho }.
\end{align*}
and then, for all $x\in (0, \min (1, t))$,  $t>0$,
\begin{align}
&\Big| \frac {1} {2\pi i}\int  _{  \mathscr Re s=-1-\delta  }x^{-s}
\int _0^t g(y)U\left(\frac {t} {y}, s \right)y^{s-1}dyds \Big|\le
C||g|| _{ \theta, \rho  }x^{1+\delta} t^{-1-\delta-\theta}(1+t)^{-\rho }. \label{S7P14D4}
\end{align}
In the first and second  terms at the right hand side of (\ref{S7P14E2A}), by (2.21) and (2.27) in \cite{m},

\begin{align}
\label{S7P14E2B}
&\text{Res}\left( U\left(\frac {t} {y}, s\right)s=0\right)=\frac {12 B(1)} {\pi ^2}\frac {1} {2i\pi }\int  _{ \mathscr Re (r) =\beta  }\left(\frac {t} {y} \right)^{-r }\frac {\Gamma (r )} {B(r ) }dr\\
\label{S7P14E2C}
&\text{Res}\left( U\left(\frac {t} {y}, s\right)s=-1\right) =\text{Res}\left(B(s); s=-1) \right)\frac {1} {2i\pi }\int  _{ \mathscr Re (r) =\beta  }\left(\frac {t} {y} \right)^{-r-1 }\frac {\Gamma (r+1 )} {B(r ) }dr
\end{align}
where $B(1)>0$ by (2.7) in \cite{m}. And since  $0<y<t$ in these  terms, the estimates of the integrals in the right hand side of  (\ref{S7P14E2B}) and  (\ref{S7P14E2C})  are obtained for large values of $\beta $. Since the function $\Gamma $ is analytic for $\beta >0$, and the function $B$ is meromorphic,  analytic for $\beta \in (0, 3)$ and has an isolated zero at $r=3$ it follows,
\begin{align*}
\left|\text{Res}\left( U\left(\frac {t} {y}, s\right)s=0\right)\right|\le C\left(\frac {t} {y} \right)^{-3 },\,\,\,
\left|\text{Res}\left( U\left(\frac {t} {y}, s\right)s=-1\right)\right|\le C\left(\frac {t} {y} \right)^{-4 }
\end{align*}
from where, usig that $\rho +\theta<\rho $,
\begin{align}
\label{S7P14D6}
\left|\int _0^t g(y)y^{-1} \text{Res} \left(U\left(\frac {t} {y}, s\right); s=0  \right)dy\right|\le Ct^{-3}||g|| _{ \theta, \rho  }
\int _0^t y^{2-\theta}(1+y)^{-\rho } dy\nonumber\\
\le C||g|| _{ \theta, \rho  }t^{-\theta}(1+t)^{-\rho }
\end{align}
\begin{align}
\label{S7P14D6b}
\left|x\int _0^t g(y)y^{-2} \text{Res} \left(U\left(\frac {t} {y}, s\right); s=-1 \right)dy\right|\le Cxt^{-4}||g|| _{ \theta, \rho  } \int _0^t y^{2-\theta}(1+y)^{-\rho } dy\nonumber \\
\le Cxt^{-1-\theta}(1+t)^{-\rho }||g|| _{ \theta, \rho  }
\end{align}
Estimate (\ref{S7P14D1})    follows from  (\ref{S7P14E2A0})--(\ref{S7P14D4}), (\ref{S7P14D6}) and (\ref{S7P14D6b}).

For $0<t<x$ and $y\in (0, t)$, by Proposition 3.1 and Proposition 3.2 of \cite{m}, for all $\delta >0$ as small as desired,
\begin{align*}
\Lambda \left(\frac {t} {y}, \frac {x} {y}\right)=C_1\left(\frac {t} {y} \right)^{-3}\left(\left(\frac {x} {t} \right)^{-3} +
\mathcal O\left(\frac {x} {t} \right)^{-4+\delta } \right)=C_1x^{-3}y^3+y^3\mathcal O\left(t^{1-\delta } x^{-4+\delta} \right).
\end{align*}
from where,
\begin{align}
\label{S7P14E2DF27}
&\left|\int _{0}^t g(y)\Lambda\left(\frac {t} {y}, \frac {x} {y} \right)\frac {dy} {y}\right|\le Ct^{-\theta}(1+t)^{-\rho }\left(\left(\frac {t} {x} \right)^3+\left(\frac {t} {x} \right)^{4-\delta }\right)
\end{align}
and (\ref{S7P14D1B}) follows   from (\ref{S7P14E2A0}), (\ref{S7P14E2D2}), (\ref{S7P14E2DF27})

On the other hand, for $0<x<t$ and $x<1$ if we pass the first term in the right hand side of (\ref{S7P14E2A}) to the left hand side of   (\ref{S7P14E2A0})  it follows by  (\ref{S7P14E2A0})--(\ref{S7P14D4}), (\ref{S7P14D6}) and (\ref{S7P14D6b}) that for $x\in (0, 1)$,
\begin{align*}
\left|S(t)g(x)-\ell(t, g)\right|\le  C||g|| _{ \theta, \rho  }\left(t^{-\theta}  \left(\frac {x} {t}+\left(\frac {x} {t}\right)^{1+\delta } +x\right)+xt^5\right),
\end{align*}
and that shows (\ref {S7P14E1}).

In order to prove (\ref{S7P14E5}) notice that by  hypothesis,
\begin{align*}
g(y)y^{-1+r }=\alpha x^{-1+r-\theta}+s(x),\,\,|s(x)|\le C x^{-1+r -\theta+\varepsilon },\,\,\forall x\in (0, t).
\end{align*}
Then,
\begin{align*}
\ell (t, g)=\frac {12\alpha B(1)t^{-\theta}}{\pi ^2}\frac {1} {2i\pi }\int  _{ \mathscr Re(r )=\beta  }\frac {\Gamma(r )} {B(r )}\frac {dr } {(\sigma -\theta)}+\mathcal O(t)^{-\theta+\varepsilon },\,\,t\to 0
\end{align*}
from where (\ref{S7P14E5}) follows.
\end{proof}

\begin{lem}
\label{S7P141B}
For all $\theta \ge 0, \rho \ge 0$, $\theta+\rho <3$, $p\in (-2,1)$ there exists a constant $C>0$ such that, for all $g\in L^\infty ((0, T; X _{\theta, \rho  })$,

\begin{align}
\label{S7P141ABE}
(i) \quad & \int _0^{t-x}\Big(\ell(t-s, g(s))\Big)ds\le  C\sup _{ 0\le s\le t }||g(s)|| _{ p, \rho  }t^{1-p},\,\forall x\in (0, t).\\
\label{S7P141ABCE}
(ii)  \quad &  \int _{t-x}^t\Big(\ell(t-s, g(s))\Big)ds\le C\sup _{ 0\le s\le t }||g(s)|| _{ \theta, \rho  }x^{1-\theta},\,\forall x\in (0, t)\\
\label{S7P141BE}
(iii) \quad &\forall t\in (0, T),\,\,\forall x\in (0,t),\, x<1,\nonumber \\
& \int _0^t (S(t-s)g (s))(x)ds=\int _0^{t}\Big(\ell(t-s, g(s))\Big)ds+R_1(t, g, x) \nonumber \\
&|R_1(t, g, x)| \le C\sup _{ 0\le s\le t } ||g(s)|| _{ \theta, \rho  }x \left( x^{-\theta }+t^{1-\theta}\right).\\
\label{S7P141ABCEP}
(iv) \quad &\forall t>0,\,\,\forall x>t,\nonumber \\
& \left|\int _0^t (S(t-s)g (s))(x)ds\right|\le C\sup _{ 0\le s\le t } ||g(s)|| _{ \theta, \rho  }
x^{-3}t^{4-\theta}(1+t)^{-\rho }
\end{align}
\end{lem}
\begin{proof}
Property (i) follows from  (\ref{S7P14E3}). By (\ref{S7P14D1B}),
\begin{align*}
\left|\int  _{ t-x }^tS(t-s)g(s)(X)ds\right|\le C\sup _{ 0\le s\le t } ||g(s)|| _{ \theta, \rho  }x^{-3}\int  _{ t-x }^t(t-s)^{3-\theta}ds\\
\le C\sup _{ 0\le s\le t } ||g(s)|| _{ \theta, \rho  }x^{-3}t^{4-\theta}\le C\sup _{ 0\le s\le t } ||g(s)|| _{ \theta, \rho  }t^{1-\theta}.
\end{align*}

On the other hand, in order to prove (iii) write
\begin{align*}
 \int _0^t S(t-s)g (s)(x)ds =\int _0^{t-x}S(t-s)g (s)(x)ds+  \int _{t-x}^tS(t-s)g (s)(x)ds=I_1+I_2.
\end{align*}
In  the term $I_2$, if $t<1$, then $t-s<1$. If $t>1$, since $x<1$ it follows that $t-1<t-x<s$. In both cases $0<t-s<1$  and then by (\ref{S7P14D1B}),  
\begin{align*}
I_2\le C x^{-3} \int _{t-x}^t ||g(s)|| _{ \theta, \rho  } (t-s)^{3-\theta} ds\le \sup _{ 0\le s\le t } ||g(s)|| _{ \theta, \rho  } x^{1-\theta}.
\end{align*}
Similarly,
\begin{align*}
 \int _0^t \ell(t-s, g(s))ds =\int _0^{t-x}\ell(t-s, g(s))ds+  \int _{t-x}^t\ell(t-s, g(s))ds=J_1+J_2.
\end{align*}
and, by (\ref{S7P14E3}), for $\theta>-1$,
\begin{align*}
|J_2|\le C\sup _{ 0<s<T }||g(s)|| _{ \theta, \rho  }\int  _{ t-x }^t(t-s)^{-\theta}ds\le C\sup _{ 0<s<T }||g(s)|| _{ \theta, \rho  }x^{1-\theta}
\end{align*}
By (\ref{S7P14E1}),(\ref{S7P14E1B}), if $0<x<t<1$, then $0<t-s<1$ and
\begin{align}
&|I_1-J_1|=\int _0^{t-x}\Big(S(t-s)g (x)-\ell(t-s, g(s))\Big)ds \nonumber\\
&\le C_\delta \int _0^{t-x} ||g(s)|| _{ \theta, \rho  }(t-s)^{-\theta} \left(\frac {x} {t-s}+\left(\frac {x} {t-s}\right)^{1+\delta }+x \right)ds \label{E7.43}\\
&\le C \sup _{ 0\le s\le t } ||g(s)|| _{ \theta, \rho  } x \left( x^{-\theta }+t^{1-\theta}\right).\nonumber
\end{align}
and this proves (iii). 
For the last point of the Lemma, since $x>t$, $x>t-s$ for every $s\in (0, t)$ and by (\ref{S7P14D1B}) again

\begin{align*}
\left|\int _0^t S(t-s)g (s)(x)ds\right|\le C \sup _{ 0\le s\le t } ||g(s)|| _{ \theta, \rho  }
x^{-3}\int _0^t(t-s)^{3-\theta}(1+(t-s))^{-\rho }ds
\end{align*}
and $(iv)$ follows.
\end{proof}

\begin{cor}
\label{S7cor1}
For all $v_0\in in X _{\theta,  \rho  }$ with $\theta \ge 0$, $\rho >0$ and $\theta+\rho <3/2$ consider the function $v$ defined as
\begin{align*}
&v(t, X)= ( S(t)g)(X^{1/2})\\
&g(x)=w_0(x^2).
\end{align*}
and denote $v(t, X)=( \mathscr S(t)\omega_0 )(X)$. Then, for $\theta'$ and $\rho '$ satisfying (\ref{thetaP}),
\begin{align*}
&(i) \quad  v\in L^\infty((0, \infty); X _{ \theta, \rho  })\cap C((0, \infty); X _{ \theta', \rho ' }))\\
&\qquad \exists C>0;\,\,||v(t)|| _{  \theta, \rho  }\le C ||\omega_0|| _{\theta, \rho  },\,\,\forall t>0,\\
&\qquad\left|\frac {\partial v} {\partial t} \right|+\left| \mathscr L(v)\right|\in L^\infty _{ \text{loc} }((0, \infty)\times (0, \infty))\\
&\qquad\forall X>0,\,\,\lim _{t\to 0 }v(t, X)=v_0(X),\\
&(ii) \quad \frac {\partial v (t, X  )} {\partial t}=\mathscr L(v (t))(X ),\,\,\forall t>0,\,\,\forall X>0.
\end{align*} 
There exists positive constants $C>0$  such that
\begin{align*}
&(iii)\quad|v(t, X)|\le C ||\omega_0|| _{  \theta, \rho  }  t^{-2\theta}(1+t)^{-2\rho } \left(1+\frac {\sqrt X} {t}+\sqrt X \right)\,0<X<t^2,\\
&(iv)\quad |v(t, X)|\le  C||v_0|| _{ \theta, \rho  }t^{-2\theta} (1+t)^{-2\rho }\left(\frac {t} {\sqrt X}\right)^3,\,0<t^2<X.
\end{align*}
Moreover,
\begin{align*}
&(v)\quad v(t, X)=L (t; v_0)+R_2(t, \omega, X),\,\forall X\in (0, \min(1,t^2))\\
&|R_2(t, \omega, X)|\le  ||\omega_0|| _{  \theta, \rho  }t^{-2\theta}(1+t)^{-2\rho } \left(\frac {\sqrt X} {t}+\sqrt X \right),\,0<X<\min (1, t^2)\\
&L(t; v_0)= \frac {6 B(1)}{\pi ^2}\frac {1} {2i\pi }\int  _{ \mathscr Re(r )=\beta  }\frac {\Gamma(r )} {B(r )}t^{-r }
\left(\int _0^{t} v_0(\zeta^2 )\zeta ^{-1+ r }d\zeta  \right)dr,\,\,\beta \in (0, 2)\\
&(vi)\quad \forall T>0, \forall p>-1,\,\,\exists C>0;\,\,|L(t, v_0)|\le C||v_0|| _{ p, \rho  }t^{-2p}\\
&(vii)\quad \text{If for some}\,\, \varepsilon >0:\,\,v_0(X)=\alpha X^{-\theta}(1+\mathcal O(X)^\varepsilon ),\,X\to 0:\\
&\qquad \quad  L(t; v_0)=\frac {12B(1)\Gamma (2\theta)} {\pi^2 B(2\theta)}t^{-2\theta}+\mathcal O(t)^{-2\theta+\varepsilon },\,t\to 0.
\end{align*}
\end{cor}
\begin{proof}
 For any function $w_0\in X _{ \theta, \rho }$, with $ \theta \ge 0$ and $ \theta+\rho \in (0, 3/2)$, consider the function $g$ defined as $g(x)=w_0(x^2) \in X _{2 \theta, 2\rho  }$. Then (i) and (ii) follow from  (\ref{ptdm1E1})-(\ref{ptdm1E0}) in Proposition \ref{ptdm1}.
 Properties  $(iii)$ to $(vii)$ follow from Proposition \ref{S7P14} .
\end{proof}

\begin{proof}
[\upshape\bfseries{Poof of Theorem \ref{S7cor2}}]
Theorem \ref{S7cor2}  follows from Proposition \ref{ptdm1}, (ii) and Lemma \ref{S7P141B}.
\end{proof}

\section{The pseudo differential operator $P$.}
\label{Ppseudo}
\setcounter{equation}{0}
\setcounter{theo}{0}
After the change of variables $X=e^\xi , Y=e^\zeta , U(X)=w(\xi ), dY=e^\zeta d\zeta $,
\begin{align*}
&\mathscr L(v)(X)=\frac {e^{-\frac {\xi } {2}}} {2}\int  _{ \RR } (w(\zeta )-w(\xi ) )\left(\frac {1} {|e^\xi-e^\zeta |}-\frac {1} {e^\xi+e^\zeta } \right) e^{\zeta } d\zeta\\
&=\frac {e^{\frac {\xi } {2}}} {2}\int  _{ \RR } (w(\zeta )-w(\xi ) )\left(\frac {1} {|e^\xi-e^\zeta |}-\frac {1} {e^\xi+e^\zeta } \right) e^{\zeta -\xi  } d\zeta.
\end{align*}
Use of the change of variable $\zeta =\xi -h$, 

\begin{align*}
&=\frac {e^{\frac {\xi } {2}}} {2}\int  _{ \RR }(w(\xi -h )-w(\xi ) )\left(\frac {1} {|e^\xi-e^{\xi -h} |}-\frac {1} {e^\xi+e^{\xi -h} } \right) e^ { -h} 
dh\\
&=\frac {e^{-\frac {\xi } {2}}} {2}\int  _{ \RR }(w(\xi -h )-w(\xi ) )\left(\frac {1} {|1-e^{ -h} |}-\frac {1} {1+e^{ -h} } \right) e^ { -h} 
dh
=:P (w)(\xi )
\end{align*}
and since,
\begin{align*}
w(\xi -h )-w(\xi )=\int  _{ \RR }e^{ik\xi }\left(e^{-ikh}-1\right)\widehat w(k)dk
\end{align*}
it follows,
\begin{align*}
\mathscr L(v)(X)&=P(w)(\xi )\\
P(w)(\xi )&=\frac {e^{-\xi /2}} {2}\int  _{ \RR }\widehat w(k)e^{ik\xi }\int  _{ \RR }\left(\frac {1} {|1-e^{ -h} |}-\frac {1} {1+e^{ -h} } \right) e^{-h}\left(e^{-ikh}-1\right) dhdk\\
&=\int  _{ \RR }\widehat w(k)e^{ik\xi }p(\xi , k)dk\\
p(\xi , k)&=\frac {e^{-\xi /2}} {2}\int  _{ \RR }\left(\frac {1} {|1-e^{ -h} |}-\frac {1} {1+e^{-h} } \right) e^{- h}\left(e^{-ikh}-1\right) dh=-e^{-\frac {\xi} {2} }\rho_0 (k)\\
\rho _0(k)&=\frac {1} {2} \left( \text{\rm Log}(4)+ \Psi {\left(\frac {1} {2} +\frac {ik} {2}\right) }+ \Psi \left(1+\frac {ik} {2}\right) +2\gamma _E\right)
\end{align*}
where $ \Psi (s)$  is the Digamma function and $\gamma _E$ denotes the Euler's Gamma constant. 
The operator $P$ is  a pseudo differential operator, whose symbol is $p(\xi , k)=-e^{-\frac {\xi } {2}}\rho _0(k)$. 
The linear equation (\ref{S2Ewxi2L}) reads in these new variables,
\begin{align}
\label{S2Ewxi2L2}
\frac {\partial w (t, \xi )} {\partial t  }=P(w (t))(\xi )+Q(t, \xi ).
\end{align}

The following Proposition directly follows from the well known asymptotic behavior of the Digamma function $\Psi $.
\begin{prop}
\label{S2PP1}
The function $\rho_0 $ is meromorphic, analytic on the strip $\mathscr Re(k)\in (-1, 1)$. Moreover,
\begin{align*}
\rho_0 (k)&=  (\gamma _E+\log |k|)-\frac {i} {2k}+\mathcal O(k)^{-2},\,\,|k|\to \infty \\
\rho_0  (k)&=Z(3) k^2-\frac {i\pi ^2k} {12}+\mathcal O (k)^3,\,\,k\to 0,\\
|\rho _0(k)|&> 0,\,\,\forall k\in \RR\setminus\{0\}
\end{align*}
There exists $C^*>0$ and $C_*>0$ such that
\begin{align*}
C_*\left(k^2\1 _{ |k|<1}+ 2\1 _{ |k|>1}\log |k|\right)\le \mathscr Re( \rho_0 (k))\le C^*\left(k^2\1 _{ |k|<1 }+\1 _{ |k|>1 }\log |k|\right),\,\,\forall k\in \RR.
\end{align*}
\end{prop}

It follows that for all compact $K\subset \RR$, there exists $C>0$ and $c>0$ such that $-p(\xi , k)\ge c\log |k|$ for all $\xi \in K$ and $|k|\ge C$.

Then $-P$  may  be  seen as an elliptic pseudo differential operator of  logarithmic order. It is of constant strength and for any $\xi_0\in \RR$ fixed, the operator $-P _{ \xi _0 }$ with symbol $e^{-\xi _0/2}\rho _0(k)$ is hypoelliptic. It follows that $-P$ is also hypoelliptic by Theorem 7.4.1 in \cite{H}.  However, $-P$ is not uniformly elliptic since  the previous bound from below on $-p(\xi , k)$ can not be uniform for all $\xi \in \RR$. 
\begin{prop}
\label{S2PP2}
If $w\in \mathscr S'(\RR)$ is such that $\widehat w\, \rho _0\in L^2(\RR)$, then  $e^{\xi/2} P(w)\in L^2(\RR)$ and
\begin{align*}
&P(\xi )=-e^{-\xi/2 }\int  _{ \RR } e^{ik\xi }\widehat w(k)\rho _0(k)dk\\
&\widehat{e^{\xi/2 }P}(k)=-\widehat w(k)\rho _0(k),\,\,\forall k\in \RR.\\
&e^{\xi/2 }P(w)\in L^2(\RR) \Longleftrightarrow \widehat w\, \rho_0 \in L^2(\RR),\,\,\, ||e^{\xi/2 }P(f)||_2=|| \widehat w\, \rho_0 ||_2.
\end{align*}
\end{prop}
\begin{proof}
If $\widehat w\, \rho_0 \in L^2(\RR)$, then $\mathscr F^{-1}\left( \widehat w\, \rho_0\right)\in L^2(\RR)$. Therefore,  by definition of $P$,
\begin{align*}
e^{\xi /2}P(w)(\xi )&=-\int  _{ \RR }\widehat w(k)e^{ik\xi }\rho_0 (k) dk\equiv\mathscr F^{-1}\left( \widehat w\, \rho_0\right)(\xi )\in L^2\\
P(w)(\xi )&=\int  _{ \RR }\widehat w(k)e^{ik\xi }p(\xi , k)dk=-e^{-\xi/2 }\int  _{ \RR }\widehat w(k)e^{ik\xi }\rho_0 (k) dk\\
\widehat{e^{\xi/2 }P(w)}(k)&=- \widehat w(k)\rho_0 (k).
\end{align*}
and the  Proposition follows from the properties of the Fourier transform on $L^2(\RR)$.
\end{proof}
In order to describe the regularizing properties of the  equation (\ref{S2Ewxi2Lc})
the classical method of localization and freezing of coefficients may now be used.
\subsection{Freezing the coefficient in $P$.}
For $\xi _0\in \RR$, define,
\begin{align*}
&M _{ \xi _0 }h (\xi )=e^{-\frac {\xi _0} {2}}h (\xi )\\
&P_0h (\xi )=\frac {1} {2}\int  _{ \RR }(h(\xi -h )-h(\xi ) )\left(\frac {1} {|1-e^{ -h} |}-\frac {1} {1+e^{ -h} } \right) e^ { -h}dh\\
&P _{ \xi _0 }h (\xi )=e^{-\frac {\xi_0 } {2}}( P_0h )(\xi )=M _{ \xi _0 }P_0(w)(\xi )
\end{align*}
and denote
\begin{align*}
e^{-\frac {\xi _0} {2}}=\kappa _0.
\end{align*}
The operator $P _{ \xi _0 }$ satisfies then,
\begin{align*}
\widehat {P _{ \xi _0 }h }(k)= - e^{-\xi _0/2} \rho _0(k)\widehat h (k)= - \kappa _0 \rho _0(k)\widehat h (k),\,\,\text{if}\,\,  \rho _0(k) \widehat h (k)\in \mathscr S'(\RR).
\end{align*}
It immediately follows from Proposition \ref{S2PP1} that  $P_0$ is linear and continuous from $H^\sigma $ to $H^\sigma  _{ \log^{-1 }}$ and from $H^\sigma  _{ \log }$ into $H^\sigma $ for all $\sigma>0$, and that  $||P_0 w||^2 _{ H _{ \log^{-1} }^\sigma  }$ defines a norm  in $H^\sigma$, equivalent  to $||w||^2 _{ H^\sigma  }$.
\begin{rem}
The solution $h$ of the Cauchy problem for the homogeneous equation
\begin{align*}
\frac {\partial h (t, \xi )} {\partial t}= P _{ \xi _0 }h (t) (\xi ),\,\,h (0, \xi )=h _0(\xi )
\end{align*}
such that $\widehat h\,\rho _0\in \mathscr S'(\RR)$, is given as follows,
\begin{align*}
h (t, \xi )=S _{ \xi _0}( t)(h _0)(\xi )=\mathscr F^{-1}\left(\exp\left(- t   \kappa _0 \rho _0 \right) \widehat h _0 \right)(\xi ),\,\,\forall t>0,\, \forall \xi \in \RR.
\end{align*}
If $h _0\in H^\sigma (\RR)$,
\begin{align*}
\int  _{ \RR }|\widehat{ S _{ \xi _0 }(t)h_0}(k)|^{2} (1+|k|^2)^{\sigma +t \kappa _0}dk=
\int  _{ \RR }|\widehat h_0 (k)|^{2} (1+|k|^2)^{\sigma + t\kappa _0 } \exp(-2t \kappa _0\mathscr Re(\rho _0(k)))dk\\
\le C\int  _{ \RR }|\widehat h_0 (k)|^{2} (1+|k|^2)^{\sigma } dk
\end{align*}
and the following regularizing effect of $S _{ \xi _0 }$ follows,
\begin{align*}
\left|\left|S _{ \xi _0 }(t)(h _0) \right|\right| _{H^{\sigma + t \kappa _0  }  }\le C||h _0|| _{ H^\sigma  },\,\forall t>0.
\end{align*}
\end{rem}
For the non homogeneous equation the following holds,
\begin{prop}
\label{nheh1}
Suppose that $  Q\in L^2((0, T); H^0 _{ \log^{-1} })$ 
and $h_0\in L^2$. Then, for every $\xi _0\in \RR$, the Cauchy problem 

\begin{align*}
&\frac {\partial h} {\partial t}(t, \xi )=P _{ \xi _0 }(h(t))(\xi )+Q(t, \xi )\\
&h(0, \xi )=h_0(\xi )
\end{align*}
has a unique solution $h$ such that, for each $t>0$, $\partial _t \hat h$ and $\hat h(t)$ are well defined  functions for almost every $t>0$ and $\xi >0$,  and is given by
\begin{align}
\label{nheh2}
\hat h(t, k )=e^{-\kappa _0\rho _0(k)t}\hat h_0(k)+\int _0^t e^{-\kappa _0\rho _0(k)(t-s)}\widehat Q(s, k)ds\\
\int _0^T||h(t)||^2_2dt\le CT||h_0||_2^2+\int _0^T||Q(s)||^2 _{ H^0 _{ \log^{-1} } }ds
\end{align}
\end{prop}
\begin{proof}
If for all $t>0$  $\partial _t \hat h(t)$ and $\hat h(t)$are two functions, we may Fourier transform the equation to obtain, for all $t\in (0, T)$,
\begin{align*}
\frac {\partial \hat h(t, k)} {\partial t }=-\kappa _0\rho _0(k)\hat h(t, k)+\widehat Q(t, k)
\end{align*}
and then, integration in time gives (\ref{nheh2}) for $t\in (0, T)$ and $k\in \RR$. 
By Proposition \ref{S2PP1}, there exists a constant $C>0$ such that
\begin{align*}
 \int  _{ \RR} e^{-2\kappa _0\mathscr Re(\rho _0(k))t} |\hat h_0(k)|^2dk\le C ||h_0||_2^2
\end{align*}
By the hypothesis on $Q$, for all $t\in (0, T)$ and $s\in (0, t)$
\begin{align*}
&\int  _{ \RR }\left|\widehat Q(s, k)e^{-\kappa _0\rho _0(k)(t-s)} \right|^2dk
\le\int  _{ \RR }\left|\widehat Q(s, k) \right|^2e^{-2\kappa _0 \mathscr Re(\rho _0(k)(t-s)}dk\\
&=\int  _{|k|<R }\left|\widehat Q(s, k) \right|^2e^{-2\kappa _0 \mathscr Re(\rho _0(k)(t-s)}dk+\int  _{ |k|>R }\left|\widehat Q(s, k) \right|^2e^{-2\kappa _0 \mathscr Re(\rho _0(k)(t-s)}dk
\end{align*}
By,  Proposition \ref{S2PP1}
\begin{align*}
\int  _{|k|<R }\left|\widehat Q(s, k) \right|^2e^{-2\kappa _0 \mathscr Re(\rho _0(k)(t-s)}dk \le C
\int  _{|k|<R }\left|\widehat Q(s, k) \right|^2e^{-2\kappa _0 |k|^2(t-s)}dk,
\end{align*}
and
\begin{align*}
\int  _{|k|>R }\left|\widehat Q(s, k) \right|^2e^{-2\kappa _0 \mathscr Re(\rho _0(k)(t-s)}dk \le C
\int  _{|k|>R }\left|\widehat Q(s, k) \right|^2e^{-2\kappa _0 (\log |k|)(t-s)}dk\\
=\int  _{|k|>R }\left|\widehat Q(s, k) \right|^2 |k|^{-2\kappa _0(t-s)}dk
\end{align*}
from where 
\begin{align*}
||h(t)||^2_2\le  C||h_0||^2_2+\int _0^t\int  _{|k|<R }\left|\widehat Q(s, k) \right|^2e^{-2\kappa _0 |k|^2(t-s)}dkds+\\
+\int _0^t\int  _{|k|>R }\left|\widehat Q(s, k) \right|^2 |k|^{-2\kappa _0(t-s)}dkds
\end{align*}
and then,
\begin{align*}
\int _0^T||h(t)||^2_2dt\le  CT||h_0||^2_2+\int _0^T\int _0^t\int  _{|k|<R }\left|\widehat Q(s, k) \right|^2dkdsdt+\\
+\int _0^T\int _0^t\int  _{|k|>R }\left|\widehat Q(s, k) \right|^2 |k|^{-2\kappa _0(t-s)}dkdsdt
\end{align*}

\begin{align*}
\int _0^T\int _0^t  \left|\widehat Q(s, k) \right|^2 |k|^{-2\kappa _0(t-s)}dsdt
&=\int _0^T\left|\widehat Q(s, k) \right|^2\int _s^T  |k|^{-2\kappa _0(t-s)}dtds\\
&=\int _0^T\left|\widehat Q(s, k) \right|^2\left(\frac {1} {2\kappa _0(\log |k|)}-\frac {e^{-2\kappa _0 (\log |k|)(T-s)}} {2\kappa _0(\log |k|)}\right)ds\\
&< \frac {1} {2\kappa _0(\log |k|)}\int _0^T\left|\widehat Q(s, k) \right|^2ds
\end{align*}
and then
\begin{align*}
&\int _0^T\int _0^t\int  _{|k|>R }\left|\widehat Q(s, k) \right|^2 |k|^{-2\kappa _0(t-s)}dkdsdt
\le \frac {1} {2\kappa _0}\int _0^T\int  _{ |k|>R } \left|\widehat Q(s, k) \right|^2 (\log |k|)^{-1}dkds\\
&||h(t)||^2_2\le CT||h_0||_2+C\int _0^T||Q(s)||^2 _{ H^0 _{ \log^{-1} } }ds.
\end{align*}
\vskip -0.5cm
\end{proof}

\section{Regularising result in  $\xi, w$ variables.}
\label{thm3.1}
\setcounter{equation}{0}
\setcounter{theo}{0}
\subsection{Localization.}
Let us denote, $J_i=\log I_i$, for $i=1, \cdots, 4$ where, as in the Introduction,
 $I_1=(1/8, 4)$, $I_2=(1/2, 2)$, $I_4=(3/4, 5/4)$ and $I_3=(5/8, 11/8)$
and define
\begin{align*}
&\chi _0(\xi )=\eta_0(e^{\xi})
\begin{cases}
1,\,\,\xi \in J_3,\,\,\\
0,\,\,\,\xi \not \in J_2,
\end{cases}\\
&\widetilde w(t, \xi )=\chi_0(\xi )w(t, \xi )
\end{align*}

\begin{theo}
\label{S2Th3.1} 
(i)
Suppose that  $\chi _0Q\in L^2((0, 2); H^{\sigma } _{ \log^{-1} }(\RR))$ for some $\sigma\ge 0 $  and $w\in L^\infty((0, 1)\times I_1)\cap H^1((0, 1); L^2(I_2))$ is such that $w=0$ if $\xi \not \in I_1$ and satisfies

\begin{align}
\label{S2Th3.1E1}
&\frac {\partial w(t)} {\partial t}= P(w(t))+Q,\,\,\forall \xi \in I_2, t\in (0, 1)\\
&w(0, \xi )=0,\,\forall \xi\in I_2\nonumber
\end{align}
Then if $\widetilde v=\chi _0\,v$,
\begin{align*}
\int _0^1||\widetilde w(t)|| _{ H^\sigma  }^2dt\le Ce^{\frac {\sigma B} { \kappa _1 }}\int _0^1||w(s)|| _{ L^\infty ( I_1)}^2ds
+C \int _0^1||  \chi _0 Q(s) ||^2 _{ H^\sigma _{ \log^{-1} }  }ds.
\end{align*}\\

\noindent
(ii) Suppose moreover that for some $T _{ \max }>0$, $\chi _0Q\in L^2((0, T _{ \max }); H^{\sigma} _{ \log^{-1} }(\RR))$, $w\in L^\infty((0, T _{ \max })\times I_1)\cap H^1((0, T _{ \max }); L^2(I_2))$ for some $\sigma\ge 0$ satisfies
\begin{align}
\label{S2Th3.1E2}
&\frac {\partial w(t)} {\partial t}=P(w(t))+Q,\,\,\forall \xi \in I_2, t\in (0, 1)\\
&w(0, \xi )=0,\,\forall \xi\in I_2.\nonumber
\end{align}
Then, for all $t\in [0, T _{ \max }-1]$,
\begin{align}
&\sup _{ 0\le T\le T _{ \max } }\left(\int _T^{\min(T+1, T _{ \max })}||\widetilde w(t)||^2 _{ H^\sigma  }dt\right)^{1/2}\le C\left(\int _0^{T _{ \max }}||w(t) ||^2 _{ L^\infty(I_1) }dt\right)^{1/2}+ \nonumber\\
&+C\sup _{ 0\le T\le T _{ \max } }\left(\int  _{ T } ^{\min(T+1, T _{ +1\max })}\left|\left |\chi _0 Q(s))\right|\right|^2 _{ H^{\sigma} _{ \log^{-1} }  }ds \right)^{1/2}.\label{S4PUE3}
\end{align}
\\

\noindent
(iii) Suppose  that for some $T _{ \max }>0$, $\chi _0Q\in L^2((0, T _{ \max }); H^{\sigma} (\RR))$, $w\in L^\infty((0, T _{ \max })\times I_1)\cap H^1((0, T _{ \max }); L^2(I_2))$  for some $\sigma\ge 0$ satisfies
\begin{align}
\label{S2Th3.1E2}
&\frac {\partial w(t)} {\partial t}=P(w(t))+Q,\,\,\forall \xi \in I_2, t\in (0, 1)\\
&w(0, \xi )=0,\,\forall \xi\in I_2\nonumber.
\end{align}
Then,
\begin{align}
&\sup _{ 0\le T\le T _{ \max } }\left(\int _T^{\min(T+1, T _{ \max })} ||T_1\widetilde w(s)|| _{ H^\sigma  }^2ds\right)^{1/2}\le C\left(\int _0^{T _{ \max }}|| w(t)||^2 _{ L^\infty(I_1) }dt\right)^{1/2}+\nonumber \\
&+ C\sup _{ 0\le T\le T _{ \max } }\left(\int  _{ T } ^{\min(T+1, T _{ +1\max })}\left|\left |\chi _0 Q(s))\right|\right|^2  _{ H^{\sigma}  }ds \right)^{1/2}, \label{S3emc8}
\end{align}
where $T_1$ is the operator defined by $\widehat {-T_1(h)}(k)=\widehat h(k)\mathscr Re(\rho _0(k))\1 _{ |k|>1 }$.
\end{theo}
\begin{rem}
Property (iii) in Theorem \ref{S2Th3.1} shows that $\widetilde w(t)$ belongs to the generalised Liouville space of functions $h$ such that $ \widehat h(k)(1+|k|^2)^{\sigma /2} (1+\mathscr Re(\rho _0(k)\1 _{ |k|>1 })\in L^2(\RR)$.  No general inner description of such spaces is known. It can be deduced however that
\begin{align*}
&\sup _{0< |h|<1 }|h|^{-\sigma }\log |h|\int  _{ \RR }|\widetilde w(x+h)-2\widetilde w(x)+\widetilde w(x-h)|^2dx<\infty\\
&\int  _{ \RR }\left|\int  _{ \RR } \frac {\Delta _h^{[\sigma ]}\widetilde w(x)}{h^{1+\sigma} } |\log |h||dh\right|^2dx<\infty.
\end{align*}
where $\Delta _hf(x)=f(x+h/2)-f(x-h/2), \Delta_h ^2 =\Delta _h\Delta _h $, (cf. \cite{BelT}).
\end{rem}
In order to prove Theorem \ref{S2Th3.1}, multiply the equation (\ref{S2Ewxi2Lc})by $\chi_0$,
\begin{align}
&\frac {\partial \widetilde w (t, \xi )} {\partial t}=\chi_0(\xi ) P(w(t))(\xi )+\chi _0(\xi )Q \label{S2Ewxi2L4i}\\
&= P( \widetilde w (t))(\xi )+ \widetilde Q_1+\widetilde Q_2 \label{S2Ewxi2L4ii}\\
&\widetilde Q_1= \int  _{ \RR }w(\xi -h)(\chi _0(\xi )-\chi _0(\xi -h)) \left(\frac {1} {|1-e^{-h}|}- \frac {1} {|1+e^{-h}|}\right)e^{-h}dh\label{S2Ewxi2L4iii}\\
&\widetilde Q_2=\chi _0\,Q. \label{S2Ewxi2L4iiii}
\end{align}
After the change of variables $\xi -h\to \zeta $ in $\widetilde Q_1$, its derivative may be computed  by differentiation of  the function $\left(\frac {1} {|1-e^{\zeta -\xi }|}- \frac {1} {|1+e^{\zeta -\xi }|}\right)e^{\zeta -\xi }$ and using the fact that  $\chi _0(\xi )-\chi _0(\zeta )$ vanishes if $\zeta $ is near $\xi $. This gives the existence of a constant $C>0$ such that,
\begin{align}
\label{estTQ1}
||\widetilde Q_1|| _{H^1 }\le C||v|| _{ L^\infty( I_1) }.
\end{align}
\subsection{Proof of Theoremm. \ref{S2Th3.1}, (i).}
Let $\delta _0>0$ be such that,
\begin{align*}
|\xi -\xi _0|\le 3\delta_0 \Longrightarrow  \left|e^{-\xi +\xi _0} -1\right|\le 3|\xi -\xi _0|.
\end{align*}
For all $\varepsilon _0>0$ small denote then 
\begin{align*}
\delta  _{ \xi _0 }=\min \left(\delta _0, \varepsilon _0 e^{-1+\xi _0/2}\right).
\end{align*}
Consider now  the function $\widetilde w$ that satisfies  equation, (\ref{S2Ewxi2L4ii}).\\

\textbf{\underline{For $\xi _0\in I_2$}} define $\chi \in \mathscr D(\RR)$ such that
\begin{align*}
\chi (\xi )=
\begin{cases}
1,\,\,|\xi -\xi _0|\le \delta _{ \xi _0 } \\
0,\,\,|\xi -\xi _0|>2\delta _{ \xi _0 }.
\end{cases}
\end{align*}
If we multiply the  non homogeneous equation, (\ref{S2Ewxi2L4ii})
\begin{align*}
&\frac {\partial \widetilde w (t, \xi )} {\partial t}=  P( \widetilde w (t))(\xi )+\widetilde Q_1+\widetilde Q_2
\end{align*}
by $\chi $ and  denote $g=\chi \widetilde w $,
\begin{align}
\frac {\partial g(t, \xi )} {\partial t}=\chi (\xi )   P(\widetilde v(t) )(\xi )+\chi  (\widetilde Q_1+\widetilde Q_2) \label{S2Ewxi2L5}
\end{align}
with,
\begin{align*}
\chi (\xi ) P(\widetilde v(t) )(\xi )=\frac {\chi (\xi )\kappa _0} {2}
\int  _{ \RR } (\widetilde v(\zeta )-\widetilde v(\xi ) )\left(\frac {1} {|e^\xi-e^\zeta |}-\frac {1} {e^\xi+e^\zeta } \right) e^{\zeta } d\zeta+\\
+\frac {\chi (\xi )(e^{-\frac {\xi } {2}}-e^{-\frac {\xi_0 } {2}})} {2}
 \int  _{ \RR } (\widetilde v(\zeta )-\widetilde v(\xi ) )\left(\frac {1} {|e^\xi-e^\zeta |}-\frac {1} {e^\xi+e^\zeta } \right) e^{\zeta } d\zeta 
\end{align*}
and
\begin{align*}
&\chi (\xi )\int  _{ \RR } (\widetilde v(\zeta )-\widetilde v(\xi ) )\left(\frac {1} {|e^\xi-e^\zeta |}-\frac {1} {e^\xi+e^\zeta } \right) e^{\zeta } d\zeta=\\
&=\int  _{ \RR } (\chi (\zeta )\widetilde v(\zeta )-\chi (\xi )\widetilde v(\xi ) )\left(\frac {1} {|e^\xi-e^\zeta |}-\frac {1} {e^\xi+e^\zeta } \right) e^{\zeta } d\zeta-\\
&-\int  _{ \RR } (\chi (\zeta ) -\chi (\xi ) )\widetilde v(\zeta )\left(\frac {1} {|e^\xi-e^\zeta |}-\frac {1} {e^\xi+e^\zeta } \right) e^{\zeta } d\zeta
\end{align*}
and then,
\begin{align*}
 \chi (\xi )P( \widetilde vv((t) )(\xi )
 &= P _{ \xi _0 }(g(t))(\xi )-\frac {  \kappa _0} {2}
 \int  _{ \RR } (\chi (\xi )-\chi (\zeta ))\widetilde w(\zeta )\left(\frac {1} {|e^\xi-e^\zeta |}-\frac {1} {e^\xi+e^\zeta } \right) e^{\zeta } d\zeta+\\
 &+ \frac {\chi (\xi )(e^{-\frac {\xi } {2}}-e^{-\frac {\xi_0 } {2}})} {2}
 \int  _{ \RR } (\widetilde w(\zeta )-\widetilde w(\xi ) )\left(\frac {1} {|e^\xi-e^\zeta |}-\frac {1} {e^\xi+e^\zeta } \right) e^{\zeta } d\zeta
\end{align*}
and the equation (\ref{S2Ewxi2L5}) reads,
\begin{align*}
\frac {\partial g(t, \xi )} {\partial t}-& P _{ \xi _0 }(g(t))(\xi )=-\frac {  \kappa _0} {2}
 \int  _{ \RR } (\chi (\xi )-\chi (\zeta ))\widetilde w(\zeta )\left(\frac {1} {|e^\xi-e^\zeta |}-\frac {1} {e^\xi+e^\zeta } \right) e^{\zeta } d\zeta+\nonumber\\
 &+ \frac {\chi (\xi )(e^{-\frac {\xi } {2}}-e^{-\frac {\xi_0 } {2}})} {2}
 \int  _{ \RR } (\widetilde w(\zeta )-\widetilde w(\xi ) )\left(\frac {1} {|e^\xi-e^\zeta |}-\frac {1} {e^\xi+e^\zeta } \right) e^{\zeta } d\zeta+\nonumber\\
 &+ \chi  (\widetilde Q_1+\widetilde Q_2) 
\end{align*}

If $g(0, \xi )=0$ we have then,
\begin{align*}
&g(t, \xi )=-\int _0^t S _{ \xi _0 } (t-s)  \frac { \kappa _0} {2}
\int  _{ \RR } (\chi (\xi )-\chi (\zeta ))\widetilde v(s, \zeta )\left(\frac {1} {|e^\xi-e^\zeta |}-\frac {1} {e^\xi+e^\zeta } \right) e^{\zeta } d\zeta ds+\nonumber\\
&+\frac {1 } {2}\int _0^t S _{ \xi _0 }(t-s) \chi (\xi )(M _{ \xi  }-M _{ \xi _0 })P_0(\widetilde w(s))(\xi )ds+\int _0^t S _{ \xi _0 } (t-s) \chi  (\widetilde Q_1(s)+\widetilde Q_2(s))ds
\end{align*}
Use of the commutator, $[A, B](h)=A(B(h))-B(A(h))$,
\begin{align}
g(t, \xi )&= \int _0^t S _{ \xi _0 }( t-s) \chi  (\widetilde Q_1(s)+\widetilde Q_2(s))ds\nonumber\\
&- \int _0^t S _{ \xi _0 }( t-s)  \frac {\kappa _0} {2}\int  _{ \RR } (\chi (\xi )-\chi (\zeta ))\widetilde v(s, \zeta )\left(\frac {1} {|e^\xi-e^\zeta |}-\frac {1} {e^\xi+e^\zeta } \right) e^{\zeta } d\zeta ds+\nonumber\\
&+ \int _0^t S _{ \xi _0 }(t-s)P_0\Big( \chi (M _{ \xi  }-M _{ \xi _0 }) \widetilde w(s)\Big)(\xi )ds+\nonumber\\
&+ \int _0^t S _{ \xi _0 }(t-s)\Big[\chi (M _{ \xi  }-M _{ \xi _0 }),P_0 \Big](\widetilde w(s))(\xi )ds\nonumber \\
&=g_1(t, \xi )+g_2(t, \xi )+g_3(t, \xi )+g_4(t, \xi ).\label{lzm68}
\end{align}
\begin{align}
||g_1(t)||^2 _{ H^\sigma  } \le \int _0^t||S _{ \xi _0 }(t-s)\chi \widetilde Q_1(s)||^2  _{ H^\sigma  } ds+
\left|\left|\int _0^tS _{ \xi _0 }( t-s)\chi \widetilde Q_2(s) ds\right|\right|^2 _{ H^\sigma  }\label{S3g1T}
\end{align}
By the estimate on $Q_1$,
\begin{align}
 \int _0^t||S _{ \xi _0 }(t-s)\chi \widetilde Q_1(s)||^2 _{ H^\sigma  } ds\le C\int _0^t||w(s)||^2 _{ L^\infty (I_1)}ds. \label{lzm175}
\end{align}
In order to estimate the second term in the right hand side of (\ref{S3g1T}), use of the notation $e^{-\frac {\xi _0} {2}}=\kappa _0$ gives,
\begin{align*}
&\left|\left|\int _0^tS _{ \xi _0 }(t-s)h(s) ds\right|\right|^2 _{ H^\sigma  }=
\left|\left|\int _0^t e^{\kappa _0P_0(t-s)} h(s) ds \right|\right| _{ H^\sigma  }^2\\
&=\int  _{ \RR }\left(\int _0^t e^{-\kappa _0(t-s_1)\rho _0(k)} \widehat h(s_1, k)ds_1\int _0^t e^{-\kappa _0(t-s_2)\overline {\rho _0}(k)} \overline{\widehat h(s_2, k)}ds_2\right)(1+|k|^2)^\sigma dk\\
&\le \int  _{ \RR }\left(\int _0^t e^{-\kappa _0 (t-s_1)\mathscr Re \rho _0(k)} |\widehat h(s_1, k)|^2ds_1\right)^2(1+|k|^2)^\sigma dk.
\end{align*}
By Lemma \ref{SapLPGCD1},
\begin{align*}
&\int _0^1\left|\left|\int _0^tS _{ \xi _0 }(t-s)h(s) ds\right|\right|^2 _{ H^\sigma  }dt\le  C(1+\kappa _0^{-2}) \int _0^1||\chi \widetilde Q_2(s) ||^2 _{ H^\sigma _{ \log^{-1} }  }ds
\end{align*}

We deduce
\begin{align*}
\int _0^1||g_1(s)|| ^2_{ H^\sigma  }ds\le C\int _0^1||w(t)|| _{ L^\infty (I_1)}^2dt+C(1+\kappa _0^{-2})\int _0^1||\chi \widetilde Q_2(s) ||^2 _{ H^\sigma _{ \log^{-1} }  }ds.
\end{align*}
By Lemma \ref{SapLPGCD2},
\begin{align*}
\int _0^1||g_1(s)|| ^2_{ H^\sigma  }ds\le C\int _0^1||w(t)|| _{ L^\infty (I_1)}^2dt+C(1+\kappa _0^{-2})\int _0^1||\chi_0  Q(s) ||^2 _{ H^\sigma _{ \log^{-1} }  }ds.
\end{align*}
\subsubsection{Estimate of $g_3$.}
Denote $\psi (t, \xi )=\chi (\xi )(M_\xi -M _{ \xi _0 })\widetilde w(t, \xi )$. Then
\begin{align*}
&|\hat g_3(t, k )|^2 \le  \int _0^t\int _0^t  \exp \left( - (t-s_1)\kappa _0\rho _0(k)\right)  \widehat{ P_0\psi (s_1)}(k)\times \\
&\times \exp \left( - (t-s_2) \kappa _0\overline{\rho _0(k)}\right) \overline{ \widehat{ P_0\psi (s_2)}(k)}ds_1ds_2\\
&\le \left(\int _0^t  \exp \Big( - (t-s) \kappa _0\mathscr Re(\rho _0(k))\Big)  \left|\widehat{ P_0\psi (s)}(k)\right|ds\right)^2\\
&=  \left(\int _0^t  \exp \Big( - (t-s) \kappa _0\mathscr Re(\rho _0(k))\Big)\left| \kappa _0 \rho _0(k)\right| |\widehat \psi (s, k)| ds\right)^2\\
&\le  \kappa _0^2\left(\int _0^t  \exp \Big( - (t-s) \kappa _0\mathscr Re(\rho _0(k))\Big)\left| \rho _0(k)\right| |\widehat \psi (s, k) |ds\right)^2
\end{align*}

After multiplication by $(1+|k|^2)^\sigma $,
\begin{align*}
|\hat g_3(t, k )&|^2(1+|k|^2)^\sigma\le \\ 
&\le \kappa _0^2 \left(\int _0^t  \exp \Big( - (t-s)\kappa _0\mathscr Re(\rho _0(k))\Big)\left| \rho _0(k)\right|
(1+|k|^2)^{\sigma /2}| \widehat  \psi (s, k)| ds\right)^2
\end{align*}
By Proposition \ref{S2PP1},
\begin{align*}
|\rho _0(k)|\le C\left(\mathscr Re(\rho _0(k))\1 _{ |k|\ge 1 }+\1 _{ |k|\le 1 }\right).
\end{align*}
Let us then define the  operators, $T_1$, $T_2$, $M_\sigma $ and $M$ such that
\begin{align*}
&\widehat {-T_1(h)}(k)=\hat h(k)\mathscr Re(\rho _0(k))\1 _{ |k|>1 }\\
&\widehat {T_2(h)}(k)=\hat h(k)\1 _{ |k|\le 1 }\\
&\widehat {M_\sigma h}(k)=|k|^{\sigma }\\
&M(\varphi )(\xi )=\frac {1} {\sqrt{2\pi }}\int  _{ \RR }|\widehat \varphi (k)|e^{ik\xi }=\mathscr F^{-1}(|\widehat \varphi |)(\xi ).
\end{align*}
Then,
\begin{align*}
|\rho _0(k)||\widehat\psi (s, k)|\le \mathscr F(T_1(M\psi (s)))(k)+\mathscr F(T_2(M\psi (s)))(k)
\end{align*}

and
\begin{align*}
&|\hat g_3(t, k )|^2(1+|k|^2)^\sigma \le C   \kappa _0^2 \times\\
&\times  \left(\int _0^t  \exp \Big( - (t-s)  \kappa _0\mathscr Re(\rho _0(k))\Big) 
(1+|k|^2)^{\sigma /2} \mathscr F(T_1(M\psi (s)))(k) ds \right)^2+\\
&+C   \kappa _0^2\left(\int _0^t  \exp \Big( - (t-s)\kappa _0\mathscr Re(\rho _0(k))\Big) 
(1+|k|^2)^{\sigma /2}| \mathscr F(T_2(M\psi (s)))(k)ds \right)^2\\
&=\rho _1(t, k)+\rho _2(t, k).
\end{align*}
In the second term,
\begin{align*} 
(1+|k|^2)^{\sigma /2} \mathscr F(T_1(M\psi (s)))(k)\le (1+|k|^2)^{\sigma /2} |\widehat \psi (s, k)|\1 _{ |k|\le 1 }\le 2^{\sigma/2}  |\widehat \psi (s, k)|
\end{align*}
from where,
\begin{align*}
&\rho _2(t, k)\le  C   \kappa _0^2 \left(\int _0^t 
(1+|k|^2)^{\sigma /2} |\widehat{ \psi (s, k)}|   ds \right)^2
\le C t\,  \kappa _0^2\int _0^t  |\widehat{ \psi (s, k)}|^2 ds\\
&\Longrightarrow \int  _{ \RR }\rho _2(t, k) dk \le C t\,   \kappa _0^2\int _0^t ||\psi (s)|| _ 2^2 ds.
\end{align*}

On the other hand, in the first term $\rho _1(t, \xi )$ it may be written
\begin{align*}
\left(\int _0^t  \exp \Big( - (t-s) \kappa _0\mathscr Re(\rho _0(k))\Big) 
(1+|k|^2)^{\sigma /2}  \mathscr F(T_1(M\psi (s)))(k)  ds \right)^2=\\
=\left(\int _0^t  e^{\kappa _0  T_1( t-s)}
\mathscr F(T_1(M\psi (s)))(k) ds \right)^2(1+|k|^2)^{\sigma } 
\end{align*}
from where,
\begin{align*}
\int  _{ \RR } \rho _1(t, k)dk=\left|\left|   \int _0^t   e^{\kappa _0  T_1(t-s)}
T_1(M\psi (s)) ds  \right|\right| ^2 _{ H\sigma  }
\end{align*}
By Lemma \ref{SapPGCD0},
 we have then,
\begin{align*}
\int  _{ \RR }\rho _1(t, k)dk\le C  \kappa _0^{-2} \int _0^t ||M\psi(s) || _{ H^\sigma  }^2dt
\end{align*}
and therefore
\begin{align}
 ||g_3(t)|| ^2_{ H^\sigma  }\le  C(1+t) \kappa _0^{-2} \int _0^t ||\psi(s)|| ^2_{ H^\sigma  }ds\label{S3.3G3a}\\
 \int _0^1 ||g_3(t)|| ^2_{ H^\sigma  }\le C\kappa _0^{-2}\int _0^1 ||\psi(s)|| ^2_{ H^\sigma  }ds.\label{S3.3G3}
\end{align}
The last term in the right hand side of (\ref{S3.3G3}), where $\psi (t, \xi )=\chi (\xi)(M_\xi -M _{ \xi _0 })w(t, \xi )$ must now be estimated. 

Let then $\chi  _1\in C_c^\infty ( \RR)$ such that  $\chi  _1(\xi )=1$ for $|\xi -\xi _0|\le 2\delta _{ \xi _0 } $ and $\chi  _1(\xi )=0$ if $|\xi -\xi _0|>3\delta _{ \xi _0 }$. By construction,
\begin{align*}
\chi (\xi)(M_\xi -M _{ \xi _0 })\widetilde w(t, \xi )=\chi  _1(\xi )(M_\xi -M _{ \xi _0 })\chi (\xi )\widetilde w(t, \xi )
=\chi  _1(\xi )(M_\xi -M _{ \xi _0 })g(t, \xi ).
\end{align*}
Indeed, if $|\xi -\xi _0|\le 2 \delta  _{ \xi _0 }$ then $\chi  _1(\xi )\chi (\xi)(M_\xi -M _{ \xi _0 })\widetilde w(t, \xi )=\chi (\xi)(M_\xi -M _{ \xi _0 })\widetilde w(t, \xi )$ and if $|\xi -\xi _0|>2\delta  _{ \xi _0 }$ then $\chi (\xi)(M_\xi -M _{ \xi _0 })\widetilde w(t, \xi )=0$ and therefore $\chi  _1(\xi )(M_\xi -M _{ \xi _0 })\chi (\xi )\widetilde w(t, \xi )=0$ too.

If we denote 
\begin{align}
\label{S3alfa}
\alpha (\xi )=\chi  _1(\xi )(M_\xi -M _{ \xi _0 })
\end{align}
then,
\begin{align*}
|\alpha (\xi )|\le |e^{-\xi/2 }-e^{-\xi _0/2}|\chi  _1 (\xi )=e^{-\xi _0/2}|1-e^{(\xi _0-\xi )/2}|\chi  _1 (\xi )\\
\le e^{-\xi _0/2}\frac {|\xi -\xi _0|} {2}\chi  _1 (\xi )\le  e^{-\xi _0/2}\frac {|\xi -\xi _0|} {2}\chi  _1 (\xi ).
\end{align*}
and,
\begin{align}
\label{S3.3Fa1}
e^{\frac {\xi _0} {2}}|\widehat \alpha (k )|\le \frac {\delta  _{ \xi _0 }} {2}\int _ {\xi _0-3\delta  _{ \xi _0 }}^{\xi _0+3\delta  _{ \xi _0 }}d\xi \le C\delta  _{ \xi _0 }^2.
\end{align}
On the other hand,  for all $k\not=0$,
\begin{align*}
|\widehat \alpha (k)|\le \frac {C} {|k|^2}\int  _{ \RR }|\alpha'' (\xi )|d\xi 
\end{align*}
with
\begin{align*}
\alpha ''(\xi )=\frac {1} {4}e^{-\frac {\xi } {2}}\chi  _1(\xi )-\frac {1} {2}e^{-\frac {\xi } {2}}\chi  _1'(\xi )+(e^{-\frac {\xi } {2}}-e^{-\frac {\xi _0} {2}})\chi  _1''(\xi )
\end{align*}
and,
\begin{align*}
\int  _{ \xi _0-3\delta  _{ \xi _0 } }^{\xi _0+3\delta  _{ \xi _0 }}e^{-\xi/2}|\chi  _1(x)|d\xi\le 
C\int  _{ \xi _0-3\delta  _{ \xi _0 } }^{\xi _0+3\delta  _{ \xi _0 }}e^{-\xi/2}d\xi
 = C e^{-\frac {\xi _0+3\delta  _{ \xi _0 }} {2}}(e^{3\delta  _{ \xi _0 }}-1)\le Ce^{-\frac {\xi _0} {2}}\delta  _{ \xi _0 }, 
\end{align*}

\begin{align*}
\int  _{ \xi _0-3\delta  _{ \xi _0 } }^{\xi _0+3\delta  _{ \xi _0 }}e^{-\xi/2}|\chi  _1'(x)|d\xi\le C\delta  _{ \xi _0 }^{-1}
\int  _{ \xi _0-3\delta  _{ \xi _0 } }^{\xi _0+3\delta  _{ \xi _0 }}e^{-\xi/2}d\xi
 = C \delta  _{ \xi _0 }^{-1}e^{-\frac {\xi _0+3\delta  _{ \xi _0 }} {2}}(e^{3\delta  _{ \xi _0 }}-1)\le Ce^{-\frac {\xi _0} {2}},
\end{align*}
\begin{align*}
\int  _{ \xi _0-3\delta  _{ \xi _0 } }^{\xi _0+3\delta  _{ \xi _0 }} |e^{-\frac {\xi } {2}}-e^{-\frac {\xi _0} {2}}||\chi  _1''(\xi )|d\xi\le 
C\delta  _{ \xi _0 }^{-2}e^{-\frac {\xi _0} {2}}\int  _{ \xi _0-3\delta  _{ \xi _0 } }^{\xi _0+3\delta  _{ \xi _0 }} |1-e^{ \frac {\xi -\xi _0} {2}}| d\xi\\
\le C\delta  _{ \xi _0 }^{-2}e^{-\frac {\xi _0} {2}}
\int  _{ \xi _0-3\delta  _{ \xi _0 } }^{\xi _0+3\delta  _{ \xi _0 }} |\xi -\xi _0|d\xi
 \le Ce^{-\frac {\xi _0 } {2}}. 
\end{align*}
Therefore,
\begin{align}
e^{\frac {\xi  _{ 0 }} {2}}|\widehat \alpha (k)|\le \frac {C} {|k|^2},\,\,\forall k\not =0,\label{S3.3Fa2}
\end{align}
and then,  by (\ref{S3.3Fa1}), (\ref{S3.3Fa2} 
\begin{align}
e^{\frac {\xi  _{ 0 }} {2}}|\widehat \alpha (k)|&\le \frac {C \delta^2  _{ \xi _0 }} { 1+\delta^2  _{ \xi _0 }k^2}\,\,\forall k\in \RR.\label{S3.3Fa21}
\end{align}

It then follows that Lemma \ref{SapPPCM0} may be applied to the function $\alpha _0=e^{\frac {\xi  _{ 0 }} {2}}\alpha $ and  there exists  two positive constants $K _{ \xi_0  } $, depending on $\xi _0$, but  independent of  $\delta  _{ \xi _0 }$,  and $C_{ \delta  _{ \xi _0 }}$, depending on  
$\delta  _{ \xi _0 }$ such that, in the right hand side of  (\ref{S3.3G3})  
\begin{align}
\kappa _0^{-2}||\psi (s)||^2 _{ H^\sigma  }=\kappa _0^{-2}||\alpha g(s)||^2 _{ H^\sigma  }\le K  _{ \xi _0 } \delta^2  _{ \xi _0 } ||\widetilde v(s)||^2 _{ H^\sigma  }+C_{ \delta  _{ \xi _0 }}
 ||v(s) || ^2_{ L^\infty(I_1) }.\label{lzm90}
\end{align}
and then, for all $\varepsilon _0>0$ arbitrarily small, there exists $\delta _{ \xi _0 }$ sufficiently small such that
\begin{align}
\label{S3.2EG3}
\int _0^1||g_3(t)||^2 _{ H^\sigma  }dt \le    \varepsilon _0\int _0^1||\widetilde v(s)|| ^2_{ H^\sigma  }ds+C_{ \delta  _{ \xi _0 }} \int _0^1||v(s) ||^2 _{ L^\infty(I_1) }ds.
\end{align}

\subsubsection{The term $g_4$.}
By (\ref{S7LapbcE1}) in Lemma \ref{S7Lapbc}, with $\eta(\xi )=\chi (\xi )(M_\xi -M _{ \xi _0 })\in C_c^\infty(\RR)$

\begin{align}
&||S _{ \xi _0 }(t-s)\left([\eta, P_0](\widetilde w(s) )\right)|| _{ H^\sigma  }^2\le C ||\widetilde w(s)  ||^2_{H ^{ \sigma -\rho ( (t-s))}} \nonumber\\
&||g_4(t)||^2 _{ H^\sigma  }=\int _0^t ||S _{ \xi _0 }(t-s)\left([\eta, P_0](\widetilde w(s) )\right)|| _{ H^\sigma  }^2\le C  \int _0^t || \widetilde w(s)  ||^2_{H ^{ \sigma -\rho ( (t-s))}}ds
\nonumber\\
&\int _0^1||g_4(t)||^2 _{ H^\sigma  }dt \le C\int _0^1 \int _0^t || \widetilde w(s)  ||^2_{H _{ \sigma -\rho (t-s)}}ds dt.\label{S3.2estG4}
\end{align}

A simple   interpolation argument  gives for all $\varepsilon >0$ the existence of a constant $C_\varepsilon >0$ such that
\begin{align*}
||\widetilde w(s)|| _{ H^{(\sigma -1)_+} }\le 
\varepsilon ||\widetilde w(s)|| _{ H^{\sigma  }} +C_\varepsilon  ||\widetilde w(s)||_{ L^2  }. 
\end{align*}
On the other hand, if $\sigma -\rho (t-s)<0$, then $L^2\subset H^{\sigma -\rho (t-s)}$ and 
\begin{align*}
||\widetilde w (s)|| _{ H^{\sigma -\rho (t-s)} }\le ||\widetilde w(s)||_2
\end{align*} 

If on the contrary $\sigma -\rho (t-s)>0$ then  $L^2\subset H^{\sigma -\rho ( t-s)}\subset H^\sigma $, and then by  interpolation,
\begin{align*}
&||\widetilde w(s)|| _{ H^{\sigma -\rho (t-s)} }\le 
||\widetilde w(s)|| _{ H^{\sigma  }}^{1-\theta(t, s)}||\widetilde w(s)||_2^{\theta(t,s)} \\
&\theta(t, s)=\frac {\rho (t-s)} {\sigma }
\end{align*}
and for all  $\varepsilon >0$, that may depend on $t$ and $s$,
\begin{align*}
||\widetilde w(s)|| _{ H^{\sigma -\rho (t-s)} }\le 
\varepsilon ||\widetilde w(s)|| _{ H^{\sigma  }} + \varepsilon ^{-  \frac {\sigma -\rho (t-s)} {\rho (t-s)}} ||\widetilde w(s)||_{ 2 }
\end{align*}
and then,
\begin{align*}
\int _0^1 \int _0^t || \widetilde w(s)  ||^2_{H _{ \sigma -\rho (t-s)}}ds dt=\int _0^1\int  _{s\in (0, t), \rho (t-s)>\sigma }\!\!\![\cdots]dsdt
+\int _0^1\int  _{s\in (0, t), \rho (t-s)<\sigma }\!\!\![\cdots]dsdt\\
\le C\int _0^1\int  _{s\in (0, t), \rho (t-s)>\sigma } ||  \widetilde w(s)||^2 _ 2dsdt
+ \int _0^1\int  _{s\in (0, t), \rho (t-s)<\sigma } \varepsilon ||\widetilde w(s)||^2 _{ H^{\sigma  }} dsdt\\
+ \int _0^1 \int _{s\in (0, t), \rho (t-s)<\sigma } \varepsilon ^{-  \frac {\sigma -\rho (t-s)} {\rho (
t-s)}} ||\widetilde w(s)|| _2dsdt.
\end{align*}
As we have seen in the estimate of $g_4$, we must  have $\tau\rho (t)<\kappa _0t$. Let us then take
\begin{align*}
&\rho (t)=\kappa_1 t,\,\,0<\kappa_1 <\kappa _0,\\
&\varepsilon =e^{-B(t-s)},\,\, B>0
\end{align*}
from where
\begin{align*}
 \varepsilon ^{-  \frac {\sigma -\rho (t-s)} {\rho (t-s)}}=e^{-B(t-s)}\,  \left(e^{-B(t-s)}\right)^{-\frac {\sigma } {\kappa _1(t-s)}}=e^{-B(t-s)}e^{\frac {\sigma B} {\kappa _1 }}
\end{align*}
 Then,
\begin{align*}
e^{\frac {\sigma B} { \kappa _1 }}&\int _0^1 \int _{s\in (0, t), \rho (t-s)<\sigma }  e^{-B(t-s)}||\widetilde w(s)|| _{ L^2(I_1) }dsdt\le \\
&\le e^{\frac {\sigma B} { \kappa _1 }}  \int _0^1 \int _0^t ||\widetilde w(s)|| _{ L^\infty(I_1) }  dsdt \le C e^{\frac {\sigma B} { \kappa _1 }} 
\int _0^1||\widetilde w(s)|| _{ L^\infty (I_1) }ds,
\end{align*}
and,
\begin{align*}
\int _0^1\int  _{s\in (0, t), \rho (t-s)<\sigma } \varepsilon ||\widetilde w(s)||^2 _{ H^{\sigma  }} dsdt\le 
\int _0^1\int _0^t e^{-B(t-s)} ||\widetilde w(s)||^2 _{ H^{\sigma  }} dsdt\\
=\int _0^1 ||\widetilde w(s)||^2 _{ H^{\sigma  }}\int _s^1e^{-B(t-s)}dtds\le \frac {1} {B}\int _0^1 ||\widetilde w(s)||^2 _{ H^{\sigma  }}ds
\end{align*}
It follows,
\begin{align*}
\int _0^1 \int _0^t || \widetilde w(s)  ||^2_{H _{ \sigma -\rho (t-s)}}ds dt
\le \frac {C} {B} \int _0^1\ ||\widetilde w(s)||^2 _{ H^{\sigma  }} dsdt+C e^{\frac {\sigma B} { \kappa _1 }} 
\int _0^1||\widetilde w(s)|| _{ L^\infty (I_1) }ds
\end{align*}
Then, for $B >0$ large enough,
\begin{align}
\label{S3.2SBrl}
 \int _0^1||g_4(t)||^2 _{ H^\sigma  }dt\le  \frac {C} {B} \int _0^1\ ||\widetilde w(s)||^2 _{ H^{\sigma  }} ds+C e^{\frac {\sigma B} { \kappa _1 }} 
\int _0^1||\widetilde w(s)|| _{ L^\infty (I_1) }ds.
\end{align}

\subsubsection{Estimate of $g_2$.}
By Lemma \ref{S3.3L3.6}
\begin{align*}
\left|\left|\int _0^t S _{ \xi _0 }(t-s) \kappa _0
\int  _{ \RR } (\chi (\xi )-\chi (\zeta ))\widetilde v(s, \zeta )\left(\frac {1} {|e^\xi-e^\zeta |}-\frac {1} {e^\xi+e^\zeta } \right) e^{\zeta } d\zeta ds\right|\right| 
_{ H^ {{\sigma +\frac {t \kappa _0} {2}  }} }\\
\le \frac {\kappa _0} {2}
\int _0^t  \left|\left|
\int  _{ \RR } (\chi (\xi )-\chi (\zeta ))\widetilde v(s, \zeta )\left(\frac {1} {|e^\xi-e^\zeta |}-\frac {1} {e^\xi+e^\zeta } \right) e^{\zeta } d\zeta \right|\right| 
_{ H^\sigma }ds\\
\le \frac {C \kappa _0} {2}\int _0^t||\widetilde w(s)|| _{ H^{(\sigma -1)_+} }ds
\end{align*}
and
\begin{align}
\label{S3.3EstG2}
\int _0^1||g_2(t)||^2 _{ H^\sigma  }dt\le C \kappa _0^2\int _0^t||\widetilde w(s)||^2 _{ H^{(\sigma -1)_+} }ds.
\end{align}
\underline{Addition of all the terms.} 
\begin{align}
\int _0^1||g(t)|| _{ H^\sigma  }^2dt\le C||w|| _{ L^\infty ((0, 1)\times  I_1)}^2+C(1+\kappa _0^{-2})\int _0^1||\chi \widetilde Q_2(s) ||^2 _{ H^\sigma _{ \log^{-1} }  }ds+\nonumber\\
+C \kappa _0^2\int _0^t||\widetilde w(s)|| _{ H^{(\sigma -1)_+} }ds+K'  \delta  _{ \xi _0 } ^2\int _0^t||\widetilde v(s)|| ^2_{ H^\sigma  }ds+C_{ \delta  _{ \xi _0 }}
\int _0^1||\widetilde v(s)||^2 _{ L^\infty  }ds+\nonumber\\
+\frac {C} {B} \int _0^1\ ||\widetilde w(s)||^2 _{ H^{\sigma  }} ds.\label{mdrd?}
\end{align}

Using a partition of the unity of the interval $I_2$,
\begin{align*}
\int _0^1||\widetilde v(t)|| _{ H^\sigma  }^2dt\le C||w|| _{ L^\infty ((0, 1)\times  I_1)}^2+C(1+\kappa _0^{-2})\int _0^1||\chi \widetilde Q_2(s) ||^2 _{ H^\sigma _{ \log^{-1} }  }ds+\\
+C \kappa _0^2\int _0^t||\widetilde w(s)|| _{ H^{(\sigma -1)_+} }ds+K'  \delta  _{ \xi _0 } ^2\int _0^t||\widetilde w(s)|| ^2_{ H^\sigma  }ds+C_{ \delta  _{ \xi _0 }}
\int _0^1||\widetilde v(s)||^2 _{ L^\infty(I_1)}ds+\\
+\frac {C} {B} \int _0^1\ ||\widetilde w(s)||^2 _{ H^{\sigma  }} ds.
\end{align*}
If $B>0$ is large enough and $\delta  _{ \xi _0 }>0$ is small enough
\begin{align*}
\int _0^1||\widetilde w(t)|| _{ H^\sigma  }^2dt\le C\int _0^1||w(s)|| _{ L^\infty ( I_1)}^2ds+C\int _0^1||  \widetilde Q_2(s) ||^2 _{ H^\sigma _{ \log^{-1} }  }ds+\\
+C \int _0^t||\widetilde w(s)|| ^2_{ H^{(\sigma -1)_+} }ds.
\end{align*}
A simple   interpolation argument  gives for all $\varepsilon >0$ the existence of a constant $C_\varepsilon >0$ such that
\begin{align*}
||\widetilde w(s)|| _{ H^{(\sigma -1)_+} }\le 
\varepsilon ||\widetilde w(s)|| _{ H^{\sigma  }} +C_\varepsilon  || w(s)||_{ L^\infty (I_1) }. 
\end{align*}
and then,
\begin{align*}
\int _0^1||\widetilde w(t)|| _{ H^\sigma  }^2dt\le Ce^{\frac {\sigma B} { \kappa _1 }}\int _0^1||w(s)|| _{ L^\infty ( I_1)}^2ds
+C  \int _0^1||  \chi _0 Q(s) ||^2 _{ H^\sigma _{ \log^{-1} }  }ds.
\end{align*}

\subsection{Proof of Theorem. \ref{S2Th3.1}, (ii).}
The notation is the same as in the previous Section, in particular the intervals $I_\ell, \ell=1, 2, 3, 4$, the functions  $\chi _0, \chi , \eta$ and $w, \widetilde w, g$. The function $g$ satisfies now,
\begin{align}
\frac {\partial g(t, \xi )} {\partial t}-&P _{ \xi _0 }(g(t))(\xi )+a(t)g(t, \xi ) =-\frac { \kappa _0} {2}
 \int  _{ \RR } (\chi (\xi )-\chi (\zeta ))\widetilde w(\zeta )\left(\frac {1} {|e^\xi-e^\zeta |}-\frac {1} {e^\xi+e^\zeta } \right) e^{\zeta } d\zeta+\nonumber\\
 &+\frac {\chi (\xi )(e^{-\frac {\xi } {2}}-e^{-\frac {\xi_0 } {2}})} {2}
 \int  _{ \RR } (\widetilde w(\zeta )-\widetilde w(\xi ) )\left(\frac {1} {|e^\xi-e^\zeta |}-\frac {1} {e^\xi+e^\zeta } \right) e^{\zeta } d\zeta+\chi  (\widetilde Q_1+\widetilde Q_2) 
\end{align}

The function $g$ is given by
\begin{align}
g(t, \xi )&= \int _0^tS _{ \xi _0 }(t-s) \chi  (\widetilde Q_1(s)+\widetilde Q_2(s))ds \nonumber\\
&-\int _0^t S _{ \xi _0 }(t-s)  \frac {\kappa _0} {2}\int  _{ \RR } (\chi (\xi )-\chi (\zeta ))\widetilde v(s, \zeta )\left(\frac {1} {|e^\xi-e^\zeta |}-\frac {1} {e^\xi+e^\zeta } \right) e^{\zeta } d\zeta ds+\nonumber\\
&+\int _0^tS _{ \xi _0 }(t-s)P_0\Big( \chi (M _{ \xi  }-M _{ \xi _0 }) \widetilde w(s)\Big)(\xi )ds+\nonumber\\
&+\int _0^t S _{ \xi _0 }(t-s)\Big[\chi (M _{ \xi  }-M _{ \xi _0 }),P_0 \Big](\widetilde w(s))(\xi )ds-\nonumber\\
&=g_1(t, \xi )+g_2(t, \xi )+g_3(t, \xi )+g_4(t, \xi ).\label{S3emc9}
\end{align}
\subsubsection{Estimate of $g_1$.}  For $T\in (0, 1]$ the same argument as in the proof of point (i) shows,
\begin{align*}
&\left(\int _T^{T+1}||g_1(t)||^2 _{ H^\sigma  }dt\right)^{1/2}\le \left(\int _0^2||g_1(t)||^2 _{ H^\sigma  }dt\right)^{1/2}\\
&\le  C\left(\int _0^2 ||w(t)|| ^2_{ L^\infty (I_1)}dt\right)^{1/2}+C\left(\int _0^2||\chi_0  Q(s) ||^2 _{ H^\sigma _{ \log^{-1} }  }ds\right)^{1/2}
\end{align*}
For $T>1$, use of the change of variables $t=(T-1)+\tau $ gives,

\begin{align*}
\left(\int _T^{T+1}||g_1(t)||^2 _{ H^\sigma  }dt\right)^{1/2}\le \sum _{ n=1 }^{[T]}\left(\int _T^{T+1}\left|\left | \int  _{ n-1 }^{n}
S _{ \xi _0 }(t-s)\chi (\widetilde Q(s))ds\right|\right| _{ H^\sigma  }dt \right)^{1/2}+\\
+\left(\int _T^{T+1}\left|\left |\int  _{ [T] }^{T+1}S _{ \xi _0 }(t-s)\chi (\widetilde Q(s))ds\right|\right| _{ H^\sigma  }dt \right)^{1/2}=I_1+I_2.
\end{align*}
The estimate of the term $I_2$ follows as in point (i) and gives,

\begin{align*}
I_2\le C \left(\int _T^{T+1}||w(t)||^2 _{ L^\infty (I_1)}dt\right)^{1/2}+C\left(\int _T^{T+1}||\chi_0 Q(s) ||^2 _{ H^\sigma _{ \log^{-1} }  }ds\right)^{1/2}.
\end{align*}
For the estimate of $I_1$ the change of variables $t=\tau +(T-n)$ gives,
\begin{align*}
I_1=\sum _{ n=1 }^{[T]}\left(\int _n^{n+1}\left|\left | \int  _{ n-1 }^{n}S _{ \xi _0 }(\tau +(T-n)-s)\chi (\widetilde Q(s))ds\right|\right| _{ H^\sigma  }d\tau  \right)^{1/2}.
\end{align*}
Since for every $n\in [1, [T]]$ and every $h\in H^\sigma $, $||S _{ \xi _0 }(T-n)h|| _{ H^\sigma  }\le  ||h|| _{ H^\sigma  }$,
\begin{align*}
\left|\left | \int  _{ n-1 }^{n}S _{ \xi _0 }(\tau +(T-n)-s)\chi (\widetilde Q(s))ds\right|\right| _{ H^\sigma  }=\\
=\left|\left | \int  _{ n-1 }^{n}S _{ \xi _0 }(\tau +(T-n)-s) \left(\chi (\widetilde Q(s))\right)ds\right|\right| _{ H^\sigma  }.
\end{align*}
Then,
\begin{align*}
&\int _n^{n+1}\left|\left | \int  _{ n-1 }^{n}S _{ \xi _0 }(\tau +(T-n)-s)\chi (\widetilde Q(s))ds\right|\right| _{ H^\sigma  }d\tau \le\\
&\le \int _n^{n+1}\left|\left | \int  _{ n-1 }^{n}S _{ \xi _0 }(\tau- s)\beta (n, s)\chi (\widetilde Q(s))ds\right|\right| _{ H^\sigma  }d\tau\\
&\le \int  _{ n-1 }^{n+1}\left|\left | \int  _{ n-1 }^{\tau }S _{ \xi _0 }(\tau- s)\1 _{ (n-1, n )}(s)\chi (\widetilde Q(s))ds\right|\right| _{ H^\sigma  }d\tau\\
&=\int  _{ 0 }^{2}\left|\left | \int  _{ 0 }^{\tau' }S _{ \xi _0 }(\tau'- s')\1 _{ (0, 1) }(s')\chi (\widetilde Q(n-1+s'))ds'\right|\right| _{ H^\sigma  }d\tau'
\end{align*}

By Lemma...,
\begin{align*}
\int _n^{n+1}\left|\left | \int  _{ n-1 }^{n}S _{ \xi _0 }(\tau +(T-n)-s)\chi (\widetilde Q(s))ds\right|\right| _{ H^\sigma  }d\tau \le\\
\le C \int _{n-1}^{n+1}\left|\left |\1 _{ (n-1, n) }(s)\chi (\widetilde Q(s))\right|\right| _{ H^{\sigma} _{ \log^{-1} }  }ds\\
\le C \int _{n-1}^{n}\left|\left |\chi \widetilde Q(s))\right|\right|  _{ H^{\sigma} _{ \log^{-1} }  }ds.
\end{align*}
We have then obtained,
\begin{align*}
&I_1+I_2\le C\sum _{ n=1 }^{[T]} \left(\int _{n-1}^{n}\left|\left |\chi \widetilde Q(s))\right|\right| ^2_{ H^{\sigma} _{ \log^{-1} }  }ds\right)^{1/2}+\\
&\hskip 2.5cm + C\left(\int _T^{T+1}||v(t) ||^2 _{ L^\infty(I_1) }dt\right)^{1/2}+C\left(\int  _{ [T] }^{T+1}\left|\left |\chi \widetilde Q(s))\right|\right| ^2 _{ H^{\sigma} _{ \log^{-1} }  }ds \right)^{1/2}\\
&\hskip 2cm\le C\left(\int _T^{T+1}||v(t) ||^2 _{ L^\infty(I_1) }dt\right)^{1/2}+C\sup _{ 0\le T\le T _{ \max } }\left(\int  _{ T }^{T+1}\left|\left |\chi \widetilde Q(s))\right|\right|^2 _{ H^{\sigma} _{ \log^{-1} }  }ds \right)^{1/2}
\end{align*}
and
\begin{align*}
&\sup _{ 0\le T\le T _{ \max } }\left(\int  _{ T }^{\min(T+1, T _{ +1\max })} ||g_1(s)||^2 _{ H^\sigma  }  ds \right)^{1/2}\le \\
&\le C\sup _{ 0\le T\le T _{ \max } }\left(\int  _{ T } ^{\min(T+1, T _{ +1\max })}\left|\left |\chi _0 Q(s))\right|\right|^2  _{ H^{\sigma} _{ \log^{-1} }  }ds \right)^{1/2}+
C\left(\int _T^{T+1}||v(t) ||^2 _{ L^\infty(I_1) }dt\right)^{1/2}
\end{align*}

\subsubsection{Estimate of $g_3$.}  For $T\in (0, 1]$ the same argument as in the estimate of $g_3$ in point  (i) (cf. proof of (\ref{S3.2EG3})) shows that for all $\varepsilon _0$ arbitrarily small, there exists a positive constant  $C>0$ such that  
\begin{align*}
\left(\int _T^{T+1}||g_3(t)|| _{ H^\sigma  }^2dt\right)^{1/2}&\le  \left(\int _0^{2}||g_3(t)|| _{ H^\sigma  }^2dt\right)^{1/2}\\
&\le    \varepsilon _0\left(\int _0^{2)}||\widetilde v(s)|| ^2_{ H^\sigma  }ds\right)^{1/2}+C\left(\int _T^{T+1}||v(t) ||^2 _{ L^\infty(I_1) }dt\right)^{1/2}.
\end{align*}

When $T>1$, using again the function  $ \psi  (s, \xi )=\chi  (\xi )(M_\xi -M _{ \xi _0 })\widetilde w(s, \xi )$ introduced in the proof of point (i), 
\begin{align*}
\left(\int _T^{T+1}||g_3(t)|| _{ H^\sigma  }^2dt\right)^{1/2}\le \sum _{ n=1 }^{[T]}\left(\int  _{ T }^{T+1}\left|\left|\int _{ n-1 }^n
 S _{ \xi _0 }(t-s)P_0\Big(\psi (s)\Big)ds  \right|\right| ^2_{ H^\sigma  }dt\right)^{1/2}+\\
+\left(\int  _{ T }^{T+1}\left|\left|\int _{ [T] }^t
 S _{ \xi _0 }(t-s)P_0\Big(\psi (s)\Big)ds  \right|\right| ^2_{ H^\sigma  }dt\right)^{1/2}=I_1+I_2.
\end{align*}

As in the case where $T\in (0, 1)$, for all $\varepsilon _0$ arbitrarily small, there exists a positive constant  $C>0$ such that
\begin{align}
\label{S4.1EG3}
I_2 \le    \varepsilon _0\left(\int _T^{\min (T+1, T _{ \max })}||\widetilde v(s)|| ^2_{ H^\sigma  }ds\right)^{1/2}+C_{ \delta  _{ \xi _0 }} \left(\int _T^{T+1}||v(t) ||^2 _{ L^\infty(I_1) }dt\right)^{1/2} .
\end{align}
As in the estimate of $g_1$, in order to estimate $I_1$ we first use the change of time variable $t=\tau +(T-n)$ to obtain,
\begin{align*}
I_1\le \sum _{ n=1 }^{[T]}\Bigg(\int _n^{n+1}&\Big|\Big|\int  _{ n-1 }^nS _{ \xi _0 }(\tau +(T-n)-s) P_0(\psi (s)ds\Big|\Big| _{ H^\sigma  }^2d\tau 
 \Bigg)^{1/2}.
\end{align*}
and then, since $T-n\ge 0$,
\begin{align*}
I_1\le \sum _{ n=1 }^{[T]}\Bigg(\int _n^{n+1}&\Big|\Big|\int  _{ n-1 }^nS _{ \xi _0 }(\tau-s) P_0(\psi (s)ds\Big|\Big| _{ H^\sigma  }^2d\tau 
 \Bigg)^{1/2}.
\end{align*}
For each $n\in [1, \cdots, [T]]$,
\begin{align*}
\int _n^{n+1}&\Big|\Big|\int  _{ n-1 }^nS _{ \xi _0 }(\tau-s) (P_0(\psi (s)ds\Big|\Big| _{ H^\sigma  }^2d\tau\\
\le \int  _{ n-1 }^{n+1}&\Big|\Big|\int  _{ n-1 }^\tau S _{ \xi _0 }(\tau-s) P_0\Big(\1 _{ (n-1, n) }(s)  \psi (s)\Big)ds\Big|\Big| _{ H^\sigma  }^2d\tau
\end{align*}
and the same argument as in the proof of (\ref{S3.2EG3}) shows that  for all $\varepsilon _0$ arbitrarily small, there exists a positive constant  $C>0$ such that
\begin{align*}
\int  _{ n-1 }^{n+1}\Big|\Big|\int  _{ n-1 }^\tau S _{ \xi _0 }(\tau-s) P_0\Big(\1 _{ (n-1, n) }(s)  \psi (s)\Big)ds\Big|\Big| _{ H^\sigma  }^2d\tau\nonumber\\
 \le   \left(\varepsilon _0 \int  _{ n-1 }^{n}||\widetilde v(s)|| ^2_{ H^\sigma  }ds +C_{ \delta  _{ \xi _0 }}\int  _{ n-1 }^n||g(s)||^2 _2ds \right).
\end{align*}
It follows
\begin{align}
&I_1\le \varepsilon _0\sum _{ n=1 }^{[T]}  \left(\int  _{ n-1 }^{n}||g(s)|| ^2_{ H^\sigma  }ds\right)^{1/2}+C\sum _{ n=1 }^{[T]}
\left(\int  _{ n-1 }^n||g(s)||^2 _{ 2 }\right)^{1/2}\nonumber\\
&\le \varepsilon  _{ 0 }\!\!\!\!\!\sup _{ 0\le T\le T _{ \max } }\left(\int _T^{\min(T+1, T _{ \max })}\!\!\!||g(s)|| _{ H^\sigma  }^2ds\right)^{1/2}
\!\!\!\!\!+C\!\!\!\sup _{ 0\le T\le T _{ \max } }\left(\int _T^{\min(T+1, T _{ \max })}\!\!\!||g(s)|| _2^2ds\right)^{1/2}\label{S4EI1ii}
\end{align}
then,
\begin{align*}
I_1+I_2\le   \varepsilon _0\left(\int _T^{\min (T+1, T _{ \max })}||g(s)|| ^2_{ H^\sigma  }ds\right)^{1/2}+C\left(\int _T^{T+1}||v(t) ||^2 _{ L^\infty(I_1) }dt\right)^{1/2}
\end{align*}
and,
\begin{align*}
&\sup _{ 0\le T\le T _{ \max } }\left(\int _T^{\min(T+1, T _{ \max })}||g_3(s)|| _{ H^\sigma  }^2ds\right)^{1/2}\le \\
&\le \varepsilon  _{ 0 }\sup _{ 0\le T\le T _{ \max } }\left(\int _T^{\min(T+1, T _{ \max })}||g(s)|| _{ H^\sigma  }^2ds\right)^{1/2}+C\left(\int _T^{T+1}||v(t) ||^2 _{ L^\infty(I_1) }dt\right)^{1/2}.
\end{align*}
\subsubsection{Estimate of $g_4$ and $g_2$.}
With the same method, and arguing as in the estimate  (\ref{S3.2estG4}) of the term $g_4$ in part (i) of the Theorem,
\begin{align*}
&\sup _{ 0\le T\le T _{ \max } }\left(\int _T^{\min(T+1, T _{ \max })}||g_4(t)||^2 _{ H^\sigma  }dt)\right)^{1/2}\le \\
&\le C\varepsilon \sup _{ 0\le T\le T _{ \max } }\left( \int  _{ T }^{\min(T+1, T _{ \max })}||\widetilde w(s)|| _{ \sigma  }^2dsdt\right)^{1/2}+ C\left(\int _T^{T+1}||v(t) ||^2 _{ L^\infty(I_1) }dt\right)^{1/2}
\end{align*}
\underline{Estimate of $g_2$.}
Same method combined with the arguments leading to (\ref{S3.3EstG2}) give
\begin{align*}
&\sup _{ 0\le T\le T _{ \max } }\left(\int _T^{\min(T+1, T _{ \max })}||g_2(t)||^2 _{ H^\sigma  }dt)\right)^{1/2}\le C\sup _{ 0\le T\le T _{ \max } }\left(\int _T^{\min(T+1, T _{ \max })}||\widetilde w(s)||^2 _{ H^{\sigma-1}  }dt)\right)^{1/2}
\end{align*}

Adding the estimates of all the different terms $g_1, \cdots, g_4$,

\begin{align*}
&\sup _{ 0\le T\le T _{ \max } }\left(\int _T^{\min(T+1, T _{ \max })}||g(t)||^2 _{ H^\sigma  }dt\right)^{1/2}\le C\left(\int _T^{T+1}||v(t) ||^2 _{ L^\infty(I_1) }dt\right)^{1/2}+ \\
&+ C\sup _{ 0\le T\le T _{ \max } }\left(\int _T^{\min(T+1, T _{ \max })}||\widetilde w(s)||^2 _{ H^{\sigma-1}  }dt)\right)^{1/2}+\\
&+C\varepsilon \sup _{ 0\le T\le T _{ \max } }\left( \int  _{ T }^{\min(T+1, T _{ \max })}  ||\widetilde w(s)|| _{H^ \sigma}^2dsdt\right)^{1/2}+\\
&+C\sup _{ 0\le T\le T _{ \max } }\left(\int  _{ T } ^{\min(T+1, T _{ +1\max })}\left|\left |\chi _0 Q(s))\right|\right|^2 _{ H^{\sigma} _{ \log^{-1} }  }ds \right)^{1/2}.
\end{align*}
Using a partition of the unity,
\begin{align}
&\sup _{ 0\le T\le T _{ \max } }\left(\int _T^{\min(T+1, T _{ \max })}||\widetilde w(t)||^2 _{ H^\sigma  }dt\right)^{1/2}\le C\left(\int _0^{T _{ \max }}||v(t) ||^2 _{ L^\infty(I_1) }dt\right)^{1/2}+ \nonumber\\
&+ C\sup _{ 0\le T\le T _{ \max } }\left(\int _T^{\min(T+1, T _{ \max })}||\widetilde w(s)||^2 _{ H^{\sigma-1}  }dt)\right)^{1/2}+\nonumber\\
&+C\varepsilon \sup _{ 0\le T\le T _{ \max } }\left( \int  _{ T }^{\min(T+1, T _{ \max })} ||\widetilde w(s)|| _{H^ \sigma }^2dsdt\right)^{1/2}+\nonumber\\
&+C\sup _{ 0\le T\le T _{ \max } }\left(\int  _{ T } ^{\min(T+1, T _{ +1\max })}\left|\left |\chi _0 Q(s))\right|\right|^2 _{ H^{\sigma} _{ \log^{-1} }  }ds \right)^{1/2}. \label{S4PUE}
\end{align}
and (\ref{S4PUE3}) follows.

\subsection{Proof of Theorem. \ref{S2Th3.1}, (iii).} 
The proof of  (\ref{S3emc8}) is based on the estimate of 
$T_1(g )$ using again (\ref{S3emc9}). 

\noindent
\subsubsection{Estimate of $g_1$.} In the first term, if $T<1$,  since,
$$
g_1(t)=\int _0^t  S _{ \xi _0 }(t-s) (\chi  \widetilde Q(s))ds
$$
use of Lemma \ref{SapPGCD0} gives,
\begin{align*}
\left(\int _T^{T+1}||T_1(g_1)(t)||^2 _{ H^\sigma}dt\right)^{1/2}\le \left(\int _0^{2}||T_1(g_1)(t)|| _{ H^\sigma}dt\right)^{1/2}\le C\left(\int _0^2||\widetilde Q||^2 _{ H^\sigma  }dt\right)^{1/2}
\end{align*}
When $T>1$, write,
\begin{align*}
&\left(\int _T^{T+1}||T_1 g_1(t)||^2 _{ H^\sigma  }dt\right)^{1/2}\le \sum _{ n=1 }^{[T]}\left(\int _T^{T+1}\left|\left | \int  _{ n-1 }^{n} T_1\left(S _{ \xi _0 }(t-s) (\chi \widetilde Q(s))\right)ds\right|\right| _{ H^\sigma  }dt \right)^{1/2}+\\
&+\left(\int _T^{T+1}\left|\left |\int  _{ [T] }^{T+1} T_1\left(S _{ \xi _0 }(t-s) (\chi \widetilde Q(s))\right)ds\right|\right| _{ H^\sigma  }dt \right)^{1/2}=I_1+I_2.
\end{align*}
The second term is estimated as,
\begin{align*}
I_2\le C\left(\int _T^{T+1}||\widetilde Q||^2 _{ H^\sigma  }dt\right)^{1/2}.
\end{align*}
In the term $I_1$ perform the change of variables $t=\tau +(T-n)$,
\begin{align*}
I_1=&\sum _{ n=1 }^{[T]}\Bigg(\int _n^{n+1}\Big|\Big | \int  _{ n-1 }^{n} T_1\left(S _{ \xi _0 }(\tau +(T-n)-s)(\chi \widetilde Q(s))\right)ds\Big|\Big| _{ H^\sigma  }d\tau  \Bigg)^{1/2}.
\end{align*}
Since for every $n\in [1, [T]]$ and every $h\in H^\sigma $, $||S _{ \xi _0 }(T-n)h|| _{ H^\sigma  }\le  ||h|| _{ H^\sigma  }$,
\begin{align*}
&\left|\left | \int  _{ n-1 }^{n}T_1\left(S _{ \xi _0 }(\tau +(T-n)-s) (\chi \widetilde Q(s))\right)ds\right|\right| _{ H^\sigma  }=\\
&\qquad =\left|\left | \int  _{ n-1 }^{n}T_1\left(S _{ \xi _0 }(\tau +(T-n)-s) \left( (\chi \widetilde Q(s))\right)\right)ds\right|\right| _{ H^\sigma  }.
\end{align*}
Then,
\begin{align*}
\int _n^{n+1}&\left|\left | \int  _{ n-1 }^{n}T_1\left(S _{ \xi _0 }(\tau +(T-n)-s)(\chi \widetilde Q(s))\right)ds\right|\right| _{ H^\sigma  }d\tau \le\\
&\le  \int _n^{n+1}\left|\left | \int  _{ n-1 }^{n}T_1\left(S _{ \xi _0 }(\tau- s)(\chi \widetilde Q(s))\right)ds\right|\right| _{ H^\sigma  }d\tau\\
&\le  \int  _{ n-1 }^{n+1}\left|\left | \int  _{ n-1 }^{\tau }T_1\left(S _{ \xi _0 }(\tau- s)\1 _{ (n-1, n )}(s)\chi (\widetilde Q(s))\right)ds\right|\right| _{ H^\sigma  }d\tau\\
&=\int  _{ 0 }^{2}\left|\left | \int  _{ 0 }^{\tau' }T_1\Big(S _{ \xi _0 }(\tau'- s')\1 _{ (0, 1) }(s')\chi (\widetilde Q(n-1+s'))\Big)ds'\right|\right| _{ H^\sigma  }d\tau'.
\end{align*}
By Lemma \ref{SapPGCD0} 
\begin{align*}
&\int _n^{n+1}\left|\left | \int  _{ n-1 }^{n} T_1\left(S _{ \xi _0 }(\tau +(T-n)-s)(\chi \widetilde Q(s))\right)ds\right|\right| _{ H^\sigma  }d\tau \le\\
&\le C   \int _{n-1}^{n+1}\left|\left |\1 _{ (n-1, n) }(s) (\chi \widetilde Q(s))\right|\right| _{ H^{\sigma}    }ds
\le C  \int _{n-1}^{n}\left|\left |\chi \widetilde Q(s))\right|\right|  _{ H^{\sigma} }ds.
\end{align*}
We have then obtained,
\begin{align*}
&I_1+I_2\le C\sum _{ n=1 }^{[T]}e^{-A(T-n)} \left(\int _{n-1}^{n}\left|\left |\chi \widetilde Q(s))\right|\right| ^2_{ H^{\sigma}  }ds\right)^{1/2}+\\
&\hskip 2cm +\left(\int _T^{T+1} ||w(t)||^2 _{ L^2( I_1) }dt\right)^{1/2}+C\left(\int  _{ [T] }^{T+1}\left|\left |\chi \widetilde Q(s))\right|\right| ^2 _{ H^{\sigma}    }ds \right)^{1/2}\\
&\hskip 1.5cm\le  \left(\int _T^{T+1} ||w(t)||^2 _{ L^2( I_1) }dt\right)^{1/2}+C\sup _{ 0\le T\le T _{ \max } }\left(\int  _{ T }^{T+1}\left|\left |\chi \widetilde Q(s))\right|\right|^2 _{ H^{\sigma}    }ds \right)^{1/2}
\end{align*}
and
\begin{align*}
&\sup _{ 0\le T\le T _{ \max } }\left(\int  _{ T }^{\min(T+1, T _{ +1\max })} ||T_1(g_1)(s)||^2 _{ H^\sigma  }  ds \right)^{1/2}\le \\
&\le C\!\!\sup _{ 0\le T\le T _{ \max } }\left(\int  _{ T } ^{\min(T+1, T _{ +1\max })}\left|\left |\chi _0 Q(s))\right|\right|^2  _{ H^{\sigma} _{ \log^{-1} }  }ds \right)^{1/2}
\!\!\!\!+
 \left(\int _0^{T _{ \max }} ||w(t)||^2 _{ L^2( I_1) }dt\right)^{1/2}.
\end{align*}

\noindent
\subsubsection{Estimate of $g_3$.} The estimate is proved first for $T\in (0, 1)$, using arguments similar to those of the estimate of $g_3$ in case (i) of the Theorem. Then the estimate is proved for $T>1$ using the same technique  as in previous steps. Suppose then first $T\in (0, 1)$ and 
denote  here again $\psi (t, \xi )=\chi (\xi )(M_\xi -M _{ \xi _0 })\widetilde w(t, \xi )$. By definition,

\begin{align*}
&|\widehat {T_1g_3(t)}(k )|^2 \le  \int _0^t\int _0^t  \mathscr Re (\rho _0(k))^2 \exp \left( - (t-s_1)\kappa _0\rho _0(k)\right)  \widehat{ P_0\psi (s_1)}(k)\times \\
&\times \exp \left( - (t-s_2) \kappa _0\overline{\rho _0(k)}\right) \overline{ \widehat{ P_0\psi (s_2)}(k)}ds_1ds_2\\
&\le \left(\int _0^t  \mathscr Re (\rho _0(k))\exp \Big( - (t-s) \kappa _0\mathscr Re(\rho _0(k))\Big)  \left|\widehat{ P_0\psi (s)}(k)\right|ds\right)^2\\
&=  \left(\int _0^t \mathscr Re (\rho _0(k))  \exp \Big( - (t-s) \kappa _0\mathscr Re(\rho _0(k))\Big)\left| \kappa _0 \rho _0(k)\right| |\widehat \psi (s, k)| ds\right)^2\\
&\le  \kappa _0^2\left(\int _0^t  \mathscr Re (\rho _0(k)) \exp \Big( - (t-s) \kappa _0\mathscr Re(\rho _0(k))\Big)\left| \rho _0(k)\right| |\widehat \psi (s, k) |ds\right)^2
\end{align*}

After multiplication by $(1+|k|^2)^\sigma $,
\begin{align*}
|\widehat {T_1g_3(t)} (k )|^2(1+|k|^2)^\sigma \le \kappa _0^2 \Bigg(\int _0^t  \mathscr Re (\rho _0(k))  \exp \Big( - (t-s)\kappa _0\mathscr Re(\rho _0(k))\Big)\times \\
\times \left| \rho _0(k)\right|
(1+|k|^2)^{\sigma /2}| \widehat  \psi (s, k)| ds\Bigg)^2.
\end{align*}
By Proposition \ref{S2PP1},
\begin{align*}
|\rho _0(k)|\le C\left(\mathscr Re(\rho _0(k))\1 _{ |k|\ge 1 }+\1 _{ |k|\le 1 }\right)
\end{align*}
and
\begin{align}
&|\widehat {T_1g_3(t)} (k )|^2  \le C   \kappa _0^2 \Bigg(\int _0^t \mathscr Re (\rho _0(k))  \exp \Big( - (t-s)  \kappa _0\mathscr Re(\rho _0(k))\Big) \times \\
&\times \mathscr F(T_1(M\psi (s)))(k) ds \Bigg)^2+C   \kappa _0^2\Bigg(\int _0^t  \mathscr Re (\rho _0(k))  \exp \Big( - (t-s)\kappa _0\mathscr Re(\rho _0(k))\Big) 
\times \nonumber\\
&\times  \mathscr F(T_2(M\psi (s)))(k)ds \Bigg)^2=r _1(t, k)+r _2(t, k).\label{S5Er1r2}
\end{align}

The first term $r _1(t, \xi )$ may be written
\begin{align*}
\left(\int _0^t    \mathscr Re (\rho _0(k))\exp \Big( - (t-s) \kappa _0\mathscr Re(\rho _0(k))\Big) 
  \mathscr F(T_1(M\psi (s)))(k)  ds \right)^2=\\
=\left(\int _0^t \mathscr F \left[ T_1 e^{\kappa _0  T_1( t-s)}
T_1(M\psi (s)))\right](k) ds \right)^2 
\end{align*}
from where,
\begin{align*}
\int  _{ \RR } r _1(t, k)(1+|k|^2)^\sigma dk=\left|\left|   \int _0^t  T_1\left( e^{\kappa _0  T_1(t-s)}
T_1(M\psi (s))\right) ds  \right|\right| ^2 _{ H^\sigma  }.
\end{align*}
By Lemma  \ref{SapPGCD0},
\begin{align*}
\int _0^2 \left|\left|   \int _0^t  T_1\left( e^{\kappa _0  T_1(t-s)}
T_1(M\psi (s))\right) ds  \right|\right| ^2 _{ H^\sigma  }dt\le C\int _0^2||T_1(M\psi (s))||^2 _{ H^\sigma  }ds
\end{align*}
Arguing as in the estimate of $g_3$ in point (i), using Lemma \ref{SapPPCM0},   for all $\varepsilon _0>0$ as small as desired, there exists a constant $C$ such that,
\begin{align}
\label{S5Er1}
\int _0^2 \left|\left|   \int _0^t  T_1\left( e^{\kappa _0  T_1(t-s)}
T_1(M\psi (s))\right) ds  \right|\right| ^2 _{ H^\sigma  }dt\le \nonumber  \\
\le \varepsilon \int _0^2||T_1 \widetilde w (s)||^2 _{ H^\sigma  }ds+C\int _0^1||w(t)||^2 _{ L^2(I_1)}dt
\end{align}
In the second term in the right hand side of (\ref{S5Er1r2}),
\begin{align*} 
\mathscr Re (\rho _0(k)) \mathscr F(T_2(M\psi (s)))(k) (1+|k|^2)^{\sigma /2} &\le (1+|k|^2)^{\sigma /2} \mathscr Re (\rho _0(k)) |\widehat \psi (s, k)|\1 _{ |k|\le 1 }\\
&\le  C  |\widehat \psi (s, k)|\\
\Longrightarrow   \int  _{ \RR }r _2(t, k)(1+|k|^2)^\sigma  dk & \le C t\,   \kappa _0^2\int _0^t ||\psi (s)||  _{ L^\infty(I_1) }^2 ds
\end{align*}
from where,
\begin{align}
&\rho _2(t, k)\le  C   \kappa _0^2 \left(\int _0^t 
(1+|k|^2)^{\sigma /2} |\widehat{ \psi (s, k)}|   ds \right)^2
\le C t\,  \kappa _0^2\int _0^t  |\widehat{ \psi (s, k)}|^2 ds\nonumber \\
&\Longrightarrow \int  _{ \RR }\rho _2(t, k) dk \le C t\,   \kappa _0^2\int _0^t ||\psi (s)|| _ 2^2 ds\le C t\,   \kappa _0^2\int _0^t ||\psi (s)|| _ \infty^2 ds, \label{S5Er2}
\end{align}
and, from (\ref{S5Er1}) and (\ref{S5Er2}),
\begin{align*}
\int _0^2||g_3(t)||^2 _{ H^\sigma  }dt \le    \varepsilon _0\int _0^2||T_1g (s)|| ^2_{ H^\sigma  }ds+C_{ \delta  _{ \xi _0 }} \int _0^2||w(t)||^2 _{ L^\infty(I_1)}dt.
\end{align*}
It immediately follows, for $T\in (0, 1)$,
\begin{align*}
\int _T^{T+1}||g_3(t)||^2 _{ H^\sigma  }dt \le    \varepsilon _0\int _0^2||\widetilde w(s)|| ^2_{ H^\sigma  }ds+C_{ \delta  _{ \xi _0 }}\int _0^2||w(t)||^2 _{ L^\infty(I_1)}dt.
\end{align*}
If $T>1$,  using again the function  $ \psi  (s, \xi )=\chi  (\xi )(M_\xi -M _{ \xi _0 })\widetilde w(s, \xi )$ introduced in the proof of point (i), 
\begin{align*}
&\left(\int _T^{T+1}||T_1g_3(t)|| _{ H^\sigma  }^2dt\right)^{1/2}\le\sum _{ n=1 }^{[T]}\left(\int  _{ T }^{T+1}\left|\left|\int _{ n-1 }^n
 T_1\left(S _{ \xi _0 }(t-s)P_0\Big(\psi (s)\Big)\right)ds  \right|\right| ^2_{ H^\sigma  }dt\right)^{1/2}+\\
&+\left(\int  _{ T }^{T+1}\left|\left|\int _{ [T] }^t
S _{ \xi _0 }(t-s)P_0\Big(\psi (s)\Big)ds  \right|\right| ^2_{ H^\sigma  }dt\right)^{1/2}=I_1+I_2.
\end{align*}
The  estimate of $I_2$ follows arguing  as for the case where $T\in (0, 1)$, 
\begin{align}
\label{S5Eg3iiiI2}
I_2 \le   \varepsilon _0\left(\int _T^{\min (T+1, T _{ \max })}||\widetilde w(s)|| ^2_{ H^\sigma  }ds\right)^{1/2}+C \int _0^2||w(t)||^2 _{ L^\infty(I_1)}dt.
\end{align}
In the first term, write first,
\begin{align*}
I_1\le \sum _{ n=1 }^{[T]}\Bigg(\int _n^{n+1}&\Big|\Big|\int  _{ n-1 }^n T_1S _{ \xi _0 }(\tau +(T-n)-s) P_0(\psi (s)ds\Big|\Big| _{ H^\sigma  }^2d\tau 
 \Bigg)^{1/2}.
\end{align*}
Arguing as in the proof of (\ref{S4EI1ii}), for each $n\in {1, \cdots, [T]}$, for all $\varepsilon >0$ there exists a constant $C>0$ such that,
\begin{align*}
\int _n^{n+1}&\Big|\Big|\int  _{ n-1 }^nT_1S _{ \xi _0 }(\tau +(T-n)-s) P_0(\psi (s)ds\Big|\Big| _{ H^\sigma  }^2d\tau \le \\
&\hskip 2cm\le \left(
 \varepsilon _0\left(\int  _{ n-1 }^ n ||T_1g(s)|| ^2_{ H^\sigma  }ds\right)^{1/2}+C\int _0^{T _{ \max }}||w(t)||^2 _{ L^\infty(I_1) }dt\right).
\end{align*}
Therefore,
\begin{align}
\label{S5Eg3iiiI1}
I_1&\le \varepsilon _0 \sum _{ n=1 }^{[T]} 
\left(\int  _{ n-1 }^ n ||T_1g(s)|| ^2_{ H^\sigma  }ds\right)^{1/2}+C\sum _{ n=1 }^{[T]} 
  ||w||^2 _{ L^2((0, 2)\times \RR) } \nonumber\\
&\le \varepsilon _0\left(\int _T^{\min (T+1, T _{ \max })}||T_1g|| ^2_{ H^\sigma  }ds \right)^{1/2}+ C\int _0^{T _{ \max }}||w(t)||^2 _{ L^\infty(I_1) }dt.
\end{align}
From (\ref{S5Eg3iiiI2}) and  (\ref{S5Eg3iiiI1})
\begin{align*}
I_1+I_2&\le \varepsilon _0\left(\int _T^{\min (T+1, T _{ \max })}||T_1g|| ^2_{ H^\sigma  }ds \right)^{1/2}+C\int _0^{T _{ \max }}||w(t)||^2 _{ L^\infty(I_1) }dt,
\end{align*}
and 
\begin{align*}
\sup _{ 0\le T\le T _{ \max } }\left(\int_T^{\min(T+1, T _{ \max })}||T_1 g_3(s)|| _{ H^\sigma  }^2ds \right)^{1/2}\le \varepsilon _0\left(\int _T^{\min (T+1, T _{ \max })}||T_1g|| ^2_{ H^\sigma  }ds \right)^{1/2}+\\
 +C\int _0^{T _{ \max }}||w(t)||^2 _{ L^\infty(I_1) }dt.
\end{align*}
 
\subsubsection{Estimate of $g_4$.}  In order to estimate the term $g_4$ it is enough to combine the same method,  with the argument used to prove (\ref{S3.2estG4}) with (\ref{S7LapbcE2}) instead of (\ref{S7LapbcE1}),

\begin{align*}
&\sup _{ 0\le T\le T _{ \max } }\left(\int_T^{\min(T+1, T _{ \max })}||T_1 g_4(s)|| _{ H^\sigma  }^2ds \right)^{1/2}\le \\
&\le  \varepsilon _0\left(\int _T^{\min (T+1, T _{ \max })}||T_1g|| ^2_{ H^\sigma  }ds \right)^{1/2}
 +C\left(\int _0^{T _{ \max }}||w(t)||^2 _{ L^\infty(I_1) }dt\right)^{1/2}.
\end{align*}

\subsubsection{Estimate of $g_2$.} Use of the same method  with Lemma \ref{SapPGCD0} and Lemma \ref{S3.3L3.6}, yield,
\begin{align*}
&\sup _{ 0\le T\le T _{ \max } }\left(\int_T^{\min(T+1, T _{ \max })}||T_1 g_2(s)|| _{ H^\sigma  }^2ds \right)^{1/2}\le \\
&\le   C\left(\int _T^{\min (T+1, T _{ \max })}||\widetilde w(s)|| ^2_{ H^{\sigma-1}  }ds \right)^{1/2}
+C\int _0^{T _{ \max }}||w(t)||^2 _{ L^\infty(I_1) }dt.
\end{align*}

Adding the estimates of $g_1,\cdots, g_4$, it follows
\begin{align*}
&\sup _{ 0\le T\le T _{ \max } }\left(\int _T^{\min(T+1, T _{ \max })} ||T_1g(s)|| _{ H^\sigma  }^2ds\right)^{1/2}\le \\
&\le \varepsilon \sup _{ 0\le T\le T _{ \max } }\left(\int _T^{\min(T+1, T _{ \max })} ||T_1 g (s)|| _{ H^\sigma  }^2ds\right)^{1/2}+\\
&+  C\left(\int _T^{\min (T+1, T _{ \max })}||\widetilde w(s)|| ^2_{ H^{\sigma-1}  }ds \right)^{1/2}+C\left(\int _0^{T _{ \max }}||\widetilde w(t)||^2 _{ L^\infty(I_1) }dt\right)^{1/2}+\\
&+ C\sup _{ 0\le T\le T _{ \max } }\left(\int  _{ T } ^{\min(T+1, T _{ +1\max })}\left|\left |\chi _0 Q(s))\right|\right|^2  _{ H^{\sigma} _{ \log^{-1} }  }ds \right)^{1/2}
\end{align*}
and then,
\begin{align*}
&\sup _{ 0\le T\le T _{ \max } }\left(\int _T^{\min(T+1, T _{ \max })} ||T_1g(s)|| _{ H^\sigma  }^2ds\right)^{1/2}\le \\
&\le   C\left(\int _T^{\min (T+1, T _{ \max })}||\widetilde w(s)|| ^2_{ H^{\sigma-1}  }ds \right)^{1/2}+C\left(\int _0^{T _{ \max }}||\widetilde w(t)||^2 _{ L^\infty(I_1) }dt\right)^{1/2}+\\
&+ C\sup _{ 0\le T\le T _{ \max } }\left(\int  _{ T } ^{\min(T+1, T _{ +1\max })}\left|\left |\chi _0 Q(s))\right|\right|^2  _{ H^{\sigma} _{ \log^{-1} }  }ds \right)^{1/2}.
\end{align*}
Use of  a partition of the unity gives then,
\begin{align*}
&\sup _{ 0\le T\le T _{ \max } }\left(\int _T^{\min(T+1, T _{ \max })} ||T_1\widetilde w(s)|| _{ H^\sigma  }^2ds\right)^{1/2}\le \\
&\le   C\left(\int _T^{\min (T+1, T _{ \max })}||\widetilde w(s)|| ^2_{ H^{\sigma-1}  }ds \right)^{1/2}+C\left(\int _0^{T _{ \max }}||\widetilde w(t)||^2 _{ L^\infty(I_1) }dt\right)^{1/2}+\\
&+ C\sup _{ 0\le T\le T _{ \max } }\left(\int  _{ T } ^{\min(T+1, T _{ +1\max })}\left|\left |\chi _0 Q(s))\right|\right|^2  _{ H^{\sigma} _{ \log^{-1} }  }ds \right)^{1/2}.
\end{align*}

\section{Back to the $\omega, X $ variables.}
\label{back}
\setcounter{equation}{0}
\setcounter{theo}{0}
This Section contains the proof of Theorem \ref{mtheo}. 
If $v(X)=w(\xi )$ with $X=e^\xi $ it follows, for all $k\in \CC$,
\begin{align*}
\widehat w(k)=\frac {1} {\sqrt{2\pi }}\int  _{ \RR }e^{-i\xi k}w(\xi )d\xi =\frac {1} {\sqrt{2\pi }}\int  _{ \RR }X^{-1-i k} v(X) dX
=\frac {1} {\sqrt{2\pi }}\mathscr M(v)(-i k)
\end{align*}
where $\mathscr M$ denotes the Mellin transform. Theorem \ref{S2Th3.1} may now be translated in terms of the function $v$. To this end, define,
\begin{align*}
M_\sigma (0, \infty)=\left\{v\in L^2(0, \infty); ||v|| _{ M_\sigma  } <\infty \right\}\\
||v||^2  _{ M_\sigma  }=\int  _{ \RR }(1+|k|^2)^\sigma |\mathscr M(v)(-i k)|^2dk<\infty
\end{align*}
and $v\in M_\sigma (0, \infty)$ if and only if $w\in H^\sigma (\RR)$.

Given an open  interval $J\subset (0, \infty)$  we will say that $v\in M_\sigma (J)$ if and only if the function $w$ defined as $w(\xi )=v(e^\xi )$ is such that $w\in H^\sigma (I)$ with $I=\log (J)$ where 
we denote
\begin{align}
&I=\log (J)=\{\xi \in \RR; \,\ \exists X\in J,\,\,\xi =\log X\} \label{STEIJ1}\\
&J=e^I=\{X>0; \,\ \exists \xi \in I,\,\,X =e^\xi \}\label{STEIJ2}\\
&\forall R>0,\,\,RJ=\left\{Rz,\,z\in J \right\}\label{STEIJ3}.
\end{align}
As usual, the norms of the  local Sobolev spaces $H^\sigma (I)$ are given by,
\begin{align*}
&||w||_{H^\sigma (I)}^2=\sum _{ m=0 }^{[\sigma  ]}||w^{(m)}||^2 _{ L^2(I) }+[w]^2 _{ \sigma , I } \\
&[w]^2 _{ \sigma , I }=\int_I \int_I \frac {|w^{(m)}(\xi )-w^{(m)}(\zeta )|^2} {|\xi -\zeta |^{2(\sigma-[\sigma ]) +1}}d\xi d\zeta
\end{align*}
where $[\sigma ]$ denotes the entire part of $\sigma $.

Moreover, for $\sigma =m\in \NN$,
\begin{align*}
||w|| ^2_{ H^m _{ \log }(I) }=||w||^2 _{ H^m(I) }+[w^{(m)}]^2 _{ 0 , I }\\
[h]^2 _{ 0 , I }=\int_I \int_I \frac {|h(\xi )-h(\zeta )|^2} {|\xi -\zeta |}d\xi d\zeta.
\end{align*}

\underline{Local version of the spaces $M_\sigma $.} 
As for $M_\sigma (0, \infty)$, given an open interval $J\subset (0, \infty)$ we define, 
\begin{align*}
&M_\sigma (J)=\left\{v\in L^2(J); ||v|| _{ M_\sigma(J)  } <\infty \right\}\\
&||v||  _{ M_\sigma(J)  }=||w|| _{ H^\sigma (\log J) }
\end{align*}
In particular the following properties are going to be needed. Given an open bounded interval $J$ such that $J\subset \overline J\subset (0, \infty)$ and $R>0$ we denote,
\begin{align*}
I=\{\xi \in \RR; e^\xi \in J\};\,\,\,
I+\log R=\{\xi +\log R,\,\xi \in I\},\,\,RJ=\{RX,\,\,X\in J\}.
\end{align*}
\begin{lem}
\label{lzm75}
Suppose that $\sigma \ge 0$ and $J$ is any open bounded interval such that $I\subset \overline I\subset (0, \infty)$ and $R>0$ fixed. Then,  the spaces $H^\sigma (RJ)$ and $M_\sigma (RJ)$ are the same (in the sense of the equivalence of the norms). In particular,  there exists two constant $C>0$ and $C'>0$ independent on $R$, but that may depend on $I$, such that if $\sigma \in (0, 1)$,

\begin{align}
C||v||^2 _{ M _{ \sigma  }(RJ)}
& \le  R^{-1}\int  _{ RJ }|v ( X  )|^2 dX  +
R^{2\sigma -1}\iint\limits_{( RJ)^2} \frac {|v  (X  )-v  (Y )|^2} {|X   -Y |^{2\sigma +1}}dXdY \le C'||v||^2 _{ M _{ \sigma  }(RJ) }. \label{pgcd1}
\end{align}
and, if $\sigma \in (1, 2)$, 
\begin{align}
C ||v||^2 _{ M _{ \sigma  }(RJ) }&\le  R^{-1}||v||^2 _{ L^2(RJ) }
+R ||\omega' ||^2 _{ L^2(RJ) }+\nonumber \\
&+R^{2\sigma -1} \iint _{ (RJ)^2 }
\frac {|\omega'(X)-v'(Y)|^2} {|X-Y|^{2\sigma -1}}dXdY\le  C ||v||^2 _{ M _{ \sigma  }(RJ) }.\label{pgcd2}
\end{align}
\end{lem}
\begin{proof} 
Denote $m=[\sigma ]$, $s=\sigma -m$ and  and $\sigma \in (0, 1)$ and suppose first that $v\in H^\sigma (RJ)$. Then $v\in H^m(RJ)$ and then, $w\in H^m(I+\log R)$ and for all $k=1, \cdots, m$
\begin{align}
\int  _{ I+\log R }|w^{(k)}(\xi )|^2d\xi \le \sum_{ j=1}^k C_j \int  _{ RJ } |X|^{2j}|v^{(j)}|^2\frac {dX} {X}\le
\sum_{ j=1}^k C_jR^{2j-1}\int  _{ RJ } ||v^{(j)}||^2 _{ L^2(RJ) }. \label{pnyx}
\end{align}
Moreover, for $m\ge 1$,
\begin{align*}
&\frac {|w^{(m)}(\xi )-w^{(m)}(\zeta  )|^2} {|\xi -\zeta |^{2s+1}}\le R[X, Y, v]\\
&R[X, Y, v]=\sum _{ k=1 }^{m}C_k
\frac {X^{2k}\big|(v^{(k)}(X)-\omega^{(k)}(Y))\big|^2} {|\log X-\log Y|^{2s+1}}+\sum _{ k=1 }^{m}C_k
\frac {|X^k-Y^k|^2 |\omega^{(k)}(Y)|^2} {|\log X-\log Y|^{2s+1}}
\end{align*}
There exists a constant $C>0$ such that for $X\in RJ$ and $Y\in RJ$, $|X^k-Y^k|\le C|X-Y|R^{k-1}$. Using
\begin{align}
\frac {\min (x, y)} {|x-y|} \le \frac{1}{ |\log x-\log y|} \le \frac {\max (x, y)}{|x-y|}  \label{S6Elog}
\end{align}
it follows,
\begin{align}
\iint  _{ (RJ)^2 }\frac {|X^k-Y^k|^2 |\omega^{(k)}(Y)|^2} {|\log X-\log Y|^{2s+1}}\frac {dX} {X}\frac {dY} {Y}\le CR^{2s+2k-3}
\iint  _{ (RJ)^2 }\frac {|X-Y|^2 |\omega^{(k)}(Y)|^2} {|X-Y|^{2s+1}}dXdY\nonumber\\
\le CR^{2k-1}\int  _{ RJ }|v^{(k)}(Y)|^2dY. \label{pnyx60}
\end{align}
On the other hand, 
\begin{align}
\label{pnyx57C}
\iint _{ (RJ)^2 } \frac { \big|(v^{(m)}(X)-\omega^{(m)}(Y))\big|^2} {| X- Y|^{2s+1}}dXdY\le C||v||^2 _{ H^\sigma (RJ) }
\end{align}
from where,
\begin{align*}
&\iint  _{ (RJ)^2 }
\frac {X^{2m}\big|(v^{(m)}(X)-\omega^{(m)}(Y))\big|^2} {|\log X-\log Y|^{2s+1}}\frac {dX} {X}\frac {dY} {Y}
\le CR^{2\sigma -1}||v||^2 _{ H^\sigma (RJ) }.
\end{align*}
If $m\ge 2$,  for $k=1,\cdots, m-1$
\begin{align*}
\iint _{ (RJ)^2 } \frac { \big|(v^{(k)}(X)-\omega^{(k)}(Y))\big|^2} {| X- Y|^{2s+1}}dXdY&\le
\begin{cases}
 CR^{-2s+3}||v||^2 _{ H^\sigma (RJ) },\,\text{if}\,\,\,m+\frac {1} {2}<\sigma <m+1\nonumber\\
 CR^{-2s+2}||v^{(k+1)}||^2 _{ L^2(RJ) },\,\text{if}\,\,\,\sigma <m+\frac {1} {2} 
\end{cases}
\nonumber\\
&\le C(1+R)R^{-2s+2}||v||^2 _{ H^\sigma (RJ) } 
\end{align*}
from where for $k=1,\cdots, m-1$
\begin{align}
&\iint  _{ (RJ)^2 }
\frac {X^{2k}\big|(v^{(k)}(X)-\omega^{(k)}(Y))\big|^2} {|\log X-\log Y|^{2s+1}}\frac {dX} {X}\frac {dY} {Y}\le \nonumber\\
&\le CR^{2k+2s-1}\iint  _{ (RJ)^2 }
\frac { \big|(v^{(k)}(X)-\omega^{(k)}(Y))\big|^2} {| X- Y|^{2s+1}}dXdY\le CR^{2k+1}(1+R) ||\omega|| ^2_{ H^{\sigma  }(RJ) }.\label{pnyx57}
\end{align}
From  (\ref{pnyx60}), (\ref{pnyx57C}) and (\ref{pnyx57})
\begin{align}
\iint _{ (I+\log R)^2 } \frac {|w^{(m)}(\xi )-w^{(m)}(\zeta  )|^2} {|\xi -\zeta |^{2s+1}}d\xi d\zeta \le 
C\left(R^{2\sigma -1}+R(1+R)\sum _{ k=0 }^{m-1}
R^{2k}\right) ||\omega|| ^2_{ H^{\sigma }(RJ) }+\nonumber \\
+C\sum _{ k=0 }^m R^{2k-1}||v^{(k)}|| ^2_{ L^2(RJ) }.\label{pnyxB}
\end{align}
It follows from (\ref{pnyx}) and  (\ref{pnyxB}) that $w\in H^\sigma (I+\log R)$ and then $v\in M_\sigma (RJ)$, and
\begin{align}
||v||^2 _{ M_\sigma ( RJ) } &\le   \sum_{ k=1 }^m \sum_{ j=1}^k C_jR^{2j-1}\int  _{ RJ } ||v^{(j)}||^2 _{ L^2(RJ) }+
\nonumber\\
&+C\left(R^{2\sigma -1}+R(1+R)\sum _{ k=1 }^{m-1}
R^{2k}\right) ||\omega|| ^2_{ H^{\sigma }(RJ) }
+C\sum _{ k=0 }^m R^{2k-1}||v^{(k)}|| ^2_{ L^2(RJ) } \label{pgcd2.21}
\end{align}
On the other hand, if $v\in M_\sigma (RJ)$, then $w\in H^\sigma (I+\log R)$, $m\ge 1$, $k=1, \dots, m$:
\begin{align*}
\int  _{ RJ } |v^{(k)}(X)|^2dX\le \sum _{ j=1 }^k D_j\int  _{ I+\log R }e^{-2k\xi }|w^{(j)}(\xi )|^2e^\xi d\xi 
\le   \sum _{ j=1 }^k D'_jR^{-2k+1}||w^{(j)} ||^2 _{ L^2(I+\log R) }.
\end{align*}
Moreover, using that for $m\ge 2$
\begin{align*}
v^{(m)}(X)-v^{(m)}(X)=e ^{-m\xi }\sum _{ k=1 }^m\left(w ^{(k)}(\xi )-w ^{(k)}(\zeta ) \right)
\end{align*}
A similar argument as in the previous step gives, for $m\ge 2$,
\begin{align*}
&\iint _{ (RJ)^2 } \frac {|v^{(m)}(X )-v^{(m)}(Y  )|^2} {|X-Y |^{2s+1}}dX dY \le C
R^{-2m-2s+1} ||w|| ^2_{ H^{\sigma }(I+\log R) }+\\
&+C\sum _{ k=1 }^m R^{-2k+1}||w ^{(k)}|| ^2_{ L^2(I+\log R) }\le C\left(R^{-2m-2s+1}+ \sum _{ k=1 }^m R^{-2k+1}\right)
 ||w|| ^2_{ H^{\sigma }(I+\log R) }
\end{align*}
from where  for $m\ge 2$,
\begin{align}
||v||^2 _{ H^\sigma (RJ) }& \le 
C\sum_{ j=1}^m  \left(\frac {R^{2j-1}-R^{2m+1}} {1-R^2}\right)  ||v^{(j)}||^2 _{ L^2(RJ) }+\nonumber\\
&+C\left(R^{-2m-2s+1}+ 
\frac {R^{2m-1}} {(R^2-1)R^{2m-1}}\right)
 ||w|| ^2_{ H^{\sigma }(I+\log R) }.
 \label{pgcd2.22}
\end{align}
The equivalence of $||v|| _{ H^\sigma  }(0, \infty)$ and $||v|| _{ M_\sigma  }(0, \infty)$ follows then from  (\ref{pgcd2.21}) and (\ref{pgcd2.22}). Calculations are slightly simplified when $\sigma \in (0, 2)$ to give (\ref{pgcd1}) and (\ref{pgcd2}).
\end{proof}

\begin{cor}
\label{S6Cor1}
Suppose that $v$ is the solution to  (\ref{S2Ewxi2L}), (\ref{S2Ewxi2Lb}) given by Corollary \ref{S7cor2} where $\nu$ satisfies (\ref{mtheoE1}). Then, 

\begin{align*}
(a)\qquad (i)&\sup _ {\substack{ 0< t_0 < T^*\\ 0<R<1}}\Bigg(\int  _{ t_0}^{\min(t_0+R^{1/2}, T_*)} || v(t)||^2 _{ M_\sigma (RJ_3)}dt \Bigg)^{1/2}\le \nonumber\\
&\le C\sup _{ 0<R<1 } \left(  \int  _{ 0}^{  T_*}||v(t)||^2 _{ L^\infty(RJ_1 ) }dt \right)^{1/2} +\nonumber\\
&+C  \sup _ {\substack{ 0< t_0 < T^*\\ 0<R<1}}\left(  R \int  _{ t_0 }^{\min(t_0+R^{1/2}, T_*)}|| \eta _{ 0, R } \nu (t)||^2 _{ M _{ \sigma , \log^{-1} }  }dt  \right)^{1/2}. 
\end{align*}
\begin{align*}
\qquad (ii) \,\,&\sup _{ R>1 } \Bigg(\int _0^{ T_*} ||v(t)||^2 _{ M_\sigma(RJ_3)}dt \Bigg)^{1/2}\le C  \left(\int  _{ 0}^{  T_*}||v(t)||^2 _{ L^\infty(RJ_1) }dt \right)^{1/2}  +\nonumber\\
 &+C\sup _{ R>1 }  \left(R \int  _{ 0 } ^{T_* } \left|\left |\eta _{ 0, R }\nu(t ))\right|\right|^2 _{ M _{ \sigma , \log^{-1} }}dt  \right)^{1/2}.
\end{align*}

\begin{align*}
(b) & \sup _ {\substack{ 0< t_0 < T^*\\ 0<R<1}}\Bigg(\int  _{ t_0}^{\min(t_0+R^{1/2}, T_*)} N _{ R, \sigma  }[v(t)]dt \Bigg)^{1/2}\le C\sup _{ 0<R<1 } \left(  \int  _{ 0}^{  T_*}||v(t)||^2 _{ L^\infty(RJ_1 ) }dt \right)^{1/2} +\\
&+C  \sup _ {\substack{ 0< t_0 < T^*\\ 0<R<1}}\left(  R \int  _{ t_0 }^{\min(t_0+R^{1/2}, T_*)}|| \eta _{ 0, R } \nu (t)||^2 _{ M _{ \sigma  }  }dt  \right)^{1/2}.\\ 
&\text{where},\,(N _{ R, \sigma  }[v(t) ])^2=\frac {1} {R}\int  _{ \RR }\Big| \mathscr M\left( \eta_0\, v_R(t)\Big)(-ik)\right|^2(1+|k|^2)^\sigma \times \\
&\hskip 7cm \times \left(1 _{ |k|<1 }+\1 _{ |k|>1 }\rho_0 (k)\right)dk
\end{align*}
with,  for all $x>0$, $v_R(t, x)=v(t, Rx)$.
\end{cor}
\begin{rem}
Since no general inner description of generalized Liouville spaces is known  for $\sigma\not \in \NN$ it does not seems easy  to deduce an estimate for  some norm of the function $v$ itself only  depending of its values on the interval $RJ_3$. 
\end{rem}
\begin{proof}
\noindent
(a)\,\,\,\underline{Scaled version of $v$.}

For all open and bounded interval $J\subset (0, \infty)$, $R>0$ and $T_*>0$ consider $X\in RJ$, $Y\in RJ$, $t\in (0, T_*$. We then define the new variable $X=x R$, $Y=yR$, $t=\tau \sqrt R$ such that $x\in J$, $y\in J$ and $\tau \in (0, T_*R^{-1/2})$. Define also  $\Psi (\tau , x)=v(t, X)$ and $\tilde \nu (\tau , x)=\nu(t, X)$. Then,

\begin{align*}
&||\Psi   (\tau )||^2 _{ M_\sigma(J)}=||v (t )||^2 _{ M_\sigma(RJ)}\\
&\left|\left |\widetilde \nu(\tau ))\right|\right|^2 _{ M_{\sigma } (J) }= \left|\left | \nu(t ))\right|\right|^2 _{ M_{\sigma} ( RJ) }.
\end{align*}
Let us  also denote,
\begin{align*}
x=e^\zeta,\,\, \Omega (\tau , \zeta )=\Psi (\tau , e^\zeta ), \\
X=e^{\xi },\,\,\, W(t, \xi )=v(t, e^\xi )
\end{align*}
By a simple change of variables,

\begin{align*}
\left|\left |\Omega (\tau )\right|\right|^2 _{ L^2  (I) }&=\int  _{ e^{I} }|\Psi (\tau , x)|^2\frac {dx} {x}=\int  _{ Re^{I} }|v (t, X)|^2\frac {dX} {X}\\
\left[\Omega (\tau )\right]^2 _{ \sigma , I }&=\int  _{ e^{I} }\int  _{ e^{I} }\frac {|\Psi (\tau , x)-\Psi(\tau , y) |^2} {|\log x-\log y|^{2\sigma +1}}\frac {dx} {x}\frac {dy} {y}=\int  _{ Re^{I} }\int  _{ Re^{I} }\frac {|v(t, X)-v(t, Y) |^2} {|\log X-\log Y|^{2\sigma +1}}\frac {dX} {X}\frac {dY} {Y} \\
&\ge \int  _{ Re^{I} }\int  _{ Re^{I} }\frac {|v(t, X)-v(t, Y) |^2} {|\log X-\log Y|^{2\sigma +1}}\frac {dX} {X}\frac {dY} {Y}
\end{align*}
If $v$ satisfies,
\begin{align*}
\frac {\partial v(t, X)} {\partial \tau }&=\mathscr L(v(t))+\nu(t, X)\\
&= \int _0^\infty (v(t, Y)-v(t, X))\left(\frac {1} { | X -Y |}-\frac {1} { (X +Y ) } \right) \frac {dY} {X  ^{1/2} }+\nu (t, X)
\end{align*}
and
\begin{align*}
\frac {\partial \Psi (\tau , x)} {\partial \tau }&=\frac {\partial v(t, X)} {\partial t}\frac {dt} {d\tau }=R^{1/2}\frac {\partial v(t, X)} {\partial t}\\
\mathscr L(\Psi  (\tau ))&= \int _0^\infty (\Psi  (\tau , y)-\Psi  (\tau , x))\left(\frac {1} { | x -y |}-\frac {1} { (x +y ) } \right) \frac {dy} {x  ^{1/2} }\\
& = \int _0^\infty (v(t, Y)-v(t, X))\left(\frac {R} { | X -Y |}-\frac {R} { (X +Y ) } \right) \frac {R^{-1}dY} {R^{-1/2}X  ^{1/2} }\\
&= R^{1/2}\int _0^\infty (v(t, Y)-v(t, X))\left(\frac {1} { | X -Y |}-\frac {1} { (X +Y ) } \right) \frac { dY} { X  ^{1/2} }=R^{1/2}\mathscr L(v(t))
\end{align*}
then,
\begin{align*}
&\frac {\partial \Psi (\tau , x)} {\partial \tau }=\mathscr L(\Psi  (\tau ))+R^{1/2}\widetilde \nu (\tau , x)\\
&\widetilde \nu (\tau , x)=\nu(t, X).
\end{align*}
Denote, also
\begin{align*}
Q(t, \xi )=V(\tau , \zeta ),\,\,\xi =\zeta +\log R.
\end{align*}
and then
\begin{align*}
&X=xR\Longrightarrow \xi =\zeta+\log R\\
&\Psi (\tau, x)=v(t, X)\Longrightarrow \Omega(\tau , \zeta )=w(t, \xi )=w(t, \zeta+\log R)
\end{align*}
Suppose now that $\nu$ satisfies (\ref{mtheoE1}). Denote $\tau  _A$ the operator acting on a function $\varphi $ as $\tau  _A[\varphi] (\xi )=\varphi (\xi +A)$
Then by definition $\chi _R(X)\nu(t, X)=\chi _0(\log (X/R))\nu(t, X)=\tau   _{ -\log R }[\chi _0](\xi)Q(t, \xi )$, and $\tau   _{ -\log R }[\chi _0]Q(t)\in  L^2((0, 2); H^\sigma  _{ \log^{-1} })$.

On the other hand,  for all $x>0$, $R>0$ and $\zeta =\log x$,
\begin{align*}
\chi _0(\log x)\widetilde \nu(\tau , x)&=\chi _0(\log x) \nu(t, Rx)=\chi _0(\log (x))Q(t, \log (Rx)\\
&=\chi _0(\zeta )Q(t, \zeta +\log R)=\chi _0(\zeta )\, \tau   _{ \log R }[Q(t)](\zeta )=\chi _0(\zeta )V(\tau , \zeta ).
\end{align*}
Since 
\begin{equation*}
\mathscr F\Big(\chi _0\, \tau   _{ \log R }[Q(t)] \Big)(k)=R^{ik}\mathscr F\Big( \tau   _{ -\log R }[\chi _0]Q(t)\Big)(k),\,\,\forall k\in \RR,
\end{equation*}
it follows that $\chi _0 V(\tau ) \in  L^2((0, 2); H^\sigma  _{ \log^{-1} })$.

Point (ii) in  Theorem \ref{S2Th3.1} may then be applied and  the function $\Omega (\tau , \zeta  )$ satisfies,
\begin{align}
\sup& _{ 0\le T\le  T_*R^{-1/2} }\Bigg(\int _T^{\min(T+1,  T_*R^{-1/2})} ||\chi _0\Omega  (\tau )||^2 _{ H^\sigma}d\tau \Bigg)^{1/2}\le C\left(\int _0^{T_*R^{-1/2}}|| \Omega (\tau )||^2 _{ L^\infty (I_1) }dt\right)^{1/2} +\nonumber\\
&+C\sup _{ 0\le T\le  T_*R^{-1/2} }\left(\int  _{ T } ^{\min(T+1,  T_*R^{-1/2})}\left|\left |R^{1/2}\chi _0 V(\tau ))\right|\right|^2 _{ H^{\sigma} _{ \log^{-1} } }d\tau  \right)^{1/2}. \label{S4PUE3B0}
\end{align}

Since $\chi _0\equiv 1$ on the interval $I_3$,
\begin{align*}
||\Omega  (\tau )||^2 _{ H^\sigma (I_3) }\le C ||\chi _0\Omega  (\tau )||^2 _{ H^\sigma}
\end{align*}
and then,
\begin{align}
\sup& _{ 0\le T\le  T_*R^{-1/2} }\Bigg(\int _T^{\min(T+1,  T_*R^{-1/2})}   ||\Omega  (\tau )||^2 _{ H^\sigma (I_3) }d\tau \Bigg)^{1/2}\le  C\left(\int _0^{T_*R^{-1/2}}|| \Omega (\tau )||^2 _{ L^\infty (I_1) }dt\right)^{1/2} +\nonumber\\
&+C\sup _{ 0\le T\le  T_*R^{-1/2} }\left(\int  _{ T } ^{\min(T+1,  T_*R^{-1/2})}\left|\left |R^{1/2}\chi _0 V(\tau ))\right|\right|^2 _{ H^{\sigma} _{ \log^{-1} } }d\tau  \right)^{1/2}. \label{S4PUE3B}
\end{align}

Since,
\begin{align*}
||\Omega(\tau ) ||^2 _{ L^\infty ( I_1) }=||v(t ) ||^2 _{ L^\infty ( RJ_1) }
\end{align*}
 then, for all $T\in (0, T_*R^{-1/2})$,
\begin{align}
\label{S5Emc2}
\int _0^{ T_*R^{-1/2}} ||\Omega(\tau ) ||^2 _{ L^\infty ( I_1) }d\tau\le R^{-1/2}\int  _{ TR^{1/2} }^{  T_*} ||v(t ) ||^2 _{ L^\infty ( RJ_1) }dt.
\end{align}
By the same argument, since  $\eta _{ 0, R }(X)=\eta_0(X/R)=\chi_0\left(\log (X/R) \right)$
\begin{align}
\label{S5Emc4}
\int \limits _{ T } ^{\min(T+1,  T_*R^{-1/2})}&\left|\left |R^{1/2}\chi _0 V(\tau ))\right|\right|^2 _{ H^{\sigma} _{ \log^{-1} }   }d\tau=R^{1/2}  \int\limits  _{ TR^{1/2} }^{\min((T+1)R^{1/2}, T_*)}  ||\eta _{ 0, R } \nu (t)||^2 _{ M _{ \sigma , \log^{-1} }  }dt.
\end{align}
By definition,
\begin{align*}
&||v||^2 _{ M _{ \sigma  }(RJ_3) }\le C||\Omega  (\tau )||^2 _{ H^\sigma (I_3) }
\end{align*}
and it then follows from  (\ref{S4PUE3B}), for all $R\in (0, 1)$ and $T\in (0, T_*)$
\begin{align*}
 \Bigg( \int\limits  _{ TR^{1/2} }^{\min((T+1)R^{1/2}, T_*)}\hskip -0.5cm||v(t)|| ^2_{ M_\sigma(RJ_3) }dt \Bigg)^{1/2}\le 
 \left(\int  _{ TR^{1/2} }^{  T_*}||v(t)||^2 _{ L^\infty(RJ_1) }dt \right)^{1/2}+\nonumber\\
+C \left(R\int  _{ TR^{1/2} }^{\min((T+1)R^{1/2}, T_*)}|| \eta _{ 0, R } \nu (t)||^2 _{ M _{ \sigma , \log^{-1} }   }dt  \right)^{1/2}.
\end{align*}

If now  $t_0=TR^{1/2}$, then for all $t_0\in (0, T^*)$,
\begin{align*}
\Bigg(\int  _{ t_0}^{\min(t_0+R^{1/2}, T_*)} ||v(t)||^2 _{ M_\sigma(RJ_3)}dt \Bigg)^{1/2}\le 
\left( \int  _{ t_0}^{  T_*}||v(t)||^2 _{ L^\infty(RJ_1) }dt \right)^{1/2}+\nonumber\\
+C\left( R  \int  _{ t_0 }^{\min(t_0+R^{1/2}, T_*)}||\chi _R \nu (t)||^2 _{ M _{ \sigma , \log^{-1} }  }dt  \right)^{1/2},
\end{align*}
and,
\begin{align}
\sup _ {\substack{ 0< t_0 < T^*\\ 0<R<1}}\Bigg(&\int  _{ t_0}^{\min(t_0+R^{1/2}, T_*)} || v(t)||^2 _{ M_\sigma (RJ_3)}dt \Bigg)^{1/2}\le C\sup _{ 0<R<1 } \left(  \int  _{ 0}^{  T_*}||v(t)||^2 _{ L^\infty(RJ_1 ) }dt \right)^{1/2} +\nonumber\\
&+C  \sup _ {\substack{ 0< t_0 < T^*\\ 0<R<1}}\left(  R \int  _{ t_0 }^{\min(t_0+R^{1/2}, T_*)}||\eta _{ 0, R } \nu (t)||^2 _{ M _{ \sigma , \log^{-1} }  }dt  \right)^{1/2}. \label{S4PUE7}
\end{align}

Suppose on the other hand that $R>1$.
\begin{align*}
&\frac {\partial \Psi (\tau , x)} {\partial \tau }= \mathscr L(\Psi  (\tau ))+R^{1/2}\widetilde \nu (\tau , x)\\
&\widetilde \nu (\tau , x)=\nu(t, X).
\end{align*}
By point (i) in Theorem \ref{S2Th3.1}   applied to the function $\Omega (t, \zeta )$

\begin{align*}
\Bigg(\int _0^{T_*R^{-1/2}} ||\chi _0\Omega  (\tau )||^2 _{ H^\sigma }d\tau \Bigg)^{1/2}\le C\left(\int _0^{T_*R^{-1/2}}|| \Omega(\tau ) ||^2 _{ L^\infty ((I_1)}d\tau\right)^{1/2} +\\
 +C\left(\int  _{ 0 } ^{ T_*R^{-1/2}}\left|\left |R^{1/2} \chi  _0\widetilde V(\tau ))\right|\right|^2 _{ H^{\sigma} _{ \log^{-1} }  }d\tau  \right)^{1/2},
\end{align*}
and then, arguing as for (\ref{S5Emc2})-(\ref{S4PUE7}),
\begin{align}
\sup _{ R>1 } \Bigg(\int _0^{ T_*} ||v(t)||^2 _{ M_\sigma(RJ_3)}dt \Bigg)^{1/2}\le &C  \left(\int  _{ 0}^{  T_*}||v(t)||^2 _{ L^\infty(RJ_1) }dt \right)^{1/2}  +\nonumber\\
 &+C\sup _{ R>1 }  \left(R \int  _{ 0 } ^{T_* } \left|\left |\eta _{ 0, R }\nu(t ))\right|\right|^2 _{ M _{ \sigma , \log^{-1} }}dt  \right)^{1/2}. \label{S4PUE5}
\end{align}
\vskip 0.5cm 
\noindent
{\it Proof of part }(b).  By (iii) of Theorem  \ref{S2Th3.1}  and arguing as fin the Proof of  (\ref {S4PUE3B0})  it follows, for $R\in (0, 1)$,
 \begin{align}
&\sup _{ 0\le T\le  T_*R^{-1/2} }\Bigg(\int _T^{\min(T+1,  T_*R^{-1/2})} ||T_1\chi _0\Omega  (\tau )||^2 _{ H^\sigma}d\tau \Bigg)^{1/2}\le \nonumber \\
&\le C\left(\int _0^{T_*R^{-1/2}}|| \Omega (\tau )||^2 _{ L^\infty (I_1) }dt\right)^{1/2} +\nonumber\\
&+C\sup _{ 0\le T\le  T_*R^{-1/2} }\left(\int  _{ T } ^{\min(T+1,  T_*R^{-1/2})}\left|\left |R^{1/2}\chi _0 V(\tau ))\right|\right|^2 _{ H^{\sigma}   }d\tau  \right)^{1/2}. \label{S4PUE3B5}
\end{align}
Then, by (\ref{S4PUE3B}) and (\ref{S4PUE3B5})
\begin{align}
&\sup _{ 0\le T\le  T_*R^{-1/2} }\Bigg(\int _T^{\min(T+1,  T_*R^{-1/2})} ||(I+T_1)\chi _0\Omega  (\tau )||^2 _{ H^\sigma}d\tau \Bigg)^{1/2}\le \nonumber \\
&\le C\left(\int _0^{T_*R^{-1/2}}|| \Omega (\tau )||^2 _{ L^\infty (I_1) }dt\right)^{1/2} +\nonumber\\
&+C\sup _{ 0\le T\le  T_*R^{-1/2} }\left(\int  _{ T } ^{\min(T+1,  T_*R^{-1/2})}\left|\left |R^{1/2}\chi _0 V(\tau ))\right|\right|^2 _{ H^{\sigma} }d\tau  \right)^{1/2}. \label{S4PUE3B9}
\end{align}
The right hand side is now bounded as in the proof of the point (a). In the left hand side,
\begin{align*}
||(I+T_1)\chi _0\Omega  (\tau )||^2 _{ H^\sigma}&=\int  _{ \RR }\left| \mathscr F\left(\chi _0 \Omega (\tau ) \right)(k)\right|^2(1+|k|^2)^\sigma \left(1 _{ |k|<1 }+\1 _{ |k|>1 }\rho (k)\right)dk,\\
\mathscr F\left(\chi _0 \Omega (\tau ) \right)(k)&=\int  _{ \RR }e^{-ik\zeta  }\chi _0(\zeta )\Omega (\tau , \zeta )d\zeta =R^{-1}\int _0^\infty\chi _0(\log x)v(t, R x)x^{-ik-1}dx\\
&=R^{-1}\mathscr M\left( \eta_0\, v_R(t)\right)(-ik)
\end{align*}
where $v_R(t, x)=v(t, Rx)$ and $\eta_0(x)=\chi _0(\log x)$, from where,
\begin{align*}
||(I+T_1)\chi _0\Omega  (\tau )||^2 _{ H^\sigma}&=R^{-1}\int  _{ \RR }\left| \mathscr M\left( \eta_0\, v_R(t)\right)(-ik)\right|^2(1+|k|^2)^\sigma \left(1 _{ |k|<1 }+\1 _{ |k|>1 }\rho (k)\right)dk.
\end{align*}
Arguing as in the proof of (a),
 \begin{align*}
 \Bigg( \int\limits  _{ TR^{1/2} }^{\min((T+1)R^{1/2}, T_*)}\hskip -0.5cm N _{ R, \sigma  }[v(t)]dt \Bigg)^{1/2}\le 
 \left(\int  _{ TR^{1/2} }^{  T_*}||v(t)||^2 _{ L^\infty(RJ_1) }dt \right)^{1/2}+\nonumber\\
+C \left(R\int  _{ TR^{1/2} }^{\min((T+1)R^{1/2}, T_*)}||\eta _{ 0, R } \nu (t)||^2 _{ M _{ \sigma , }   }dt  \right)^{1/2}.
\end{align*}
from where, as for (\ref{S4PUE7})
\begin{align}
\sup _ {\substack{ 0< t_0 < T^*\\ 0<R<1}}\Bigg(&\int  _{ t_0}^{\min(t_0+R^{1/2}, T_*)} N _{ R, \sigma  }[v(t)]dt \Bigg)^{1/2}\le C\sup _{ 0<R<1 } \left(  \int  _{ 0}^{  T_*}||v(t)||^2 _{ L^\infty(RJ_1 ) }dt \right)^{1/2} +\nonumber\\
&\hskip 3cm +C  \sup _ {\substack{ 0< t_0 < T^*\\ 0<R<1}}\left(  R \int  _{ t_0 }^{\min(t_0+R^{1/2}, T_*)}||\eta _{ 0, R } \nu (t)||^2 _{ M _{ \sigma }  }dt  \right)^{1/2}. \label{S4PUE77}
\end{align}
and the proof of (b) concludes as for (a) using (\ref{crdb1}).
\end{proof}

\begin{proof}
[\upshape\bfseries{Poof of Theorem \ref{mtheo}}]
Under the assumptions of Theorem \ref{mtheo}, 
by Corollary \ref{S7cor1}, and Holder's inequality,
\begin{align*}
||v(t)|| _{ \infty }\le C\left( t^{-2\theta}||\nu(t)|| _{ \theta, \rho  }+\int _0^t||\nu(s)|| _{ \theta, \rho  }(t-s)^{-2\theta}ds\right)\\
\le C\left( t^{-2\theta}||\nu(t)|| _{ \theta, \rho  }+t^{\frac {1-4\theta} {2}}\left(\int _0^t||\nu(s)|| _{ \theta, \rho  } ^2ds\right)^{1/2}\right)
\end{align*} 
and using Fubini's Theorem,
\begin{align}
\int _0^{T_*}||v(t)||^2 _{ \infty }dt\le C\int _0^{T_*} t^{-4\theta} ||\nu(t)|| ^2_{ \theta, \rho  } dt+ 
\left(\int _0^t||\nu(s)|| _{ \theta, \rho  } ^2ds\right)\int _0^{T^*}t^{1-4\theta}dt \nonumber \\
\le  C\int _0^{T_*} t^{-4\theta} ||\nu(t)|| ^2_{ \theta, \rho  } dt+CT_*^{(2-4\theta)}\int _0^{T_*} ||\nu(s)||^2 _{ \theta, \rho  }ds\nonumber\\
\le C\left(1+T_*^{(2-4\theta)}\right)\int _0^{T_*} \left(1+t^{-4\theta}\right) ||\nu(t)|| ^2_{ \theta, \rho  } dt.
\label{crdb1}
\end{align} 
Properties (a) and (b) of  Theorem \ref{mtheo} follow from Corollary \ref{S6Cor1}. Suppose $\sigma= 0$, by Proposition \ref{S2PP1}, and the definition of $H^0  _{ \log }(\RR)$ one has $||\eta_0w_R(t)||^2 _{  H ^0  _{ \log }(\RR)}\le C(N _{ R, 0   }[v(t) ])^2$.
Then, by Lemma \ref{S7Alocloc}, 
\begin{align*}
\iint  _{ I_3\times I_3 }\frac {| w_R(t, \xi )-w_R(t, \zeta )|^2} {|\xi -\zeta |}\le C(N _{ R, \sigma   }[v(t) ])^2
\end{align*}
and,
\begin{align*}
&\iint  _{ (RJ_3)^2 }\frac {|v(X)-v(Y)|^2} {|X-Y|}dXdY=R\iint  _{ J_3^2 }\frac {|\omega_R (X)-\omega_R (Y)|^2} {|X-Y|}dXdY\\
&\le CR^2\iint  _{ J_3^2 }\frac {|v(X)-v(Y)|^2} {|\log X- \log Y|}\frac {dX} {X}\frac {dY} {Y}=
CR^2\iint  _{ I_3^2 }\frac {| w_R(t, \xi )-w_R(t, \zeta )|^2} {|\xi -\zeta |}\\
&\le CR^2(N _{ R, \sigma   }[v(t) ])
\end{align*}
and then,
\begin{align}
\label{S6Cor2E1}
&R^{-1}||v(t)|| _{ L^2(RJ_3) }+R^{-2}\iint  _{ (RJ_3)^2 }\frac {|v(X)-v(Y)|^2} {|X-Y|}dXdY
\le C(N _{ R, 0  }[v(t) ])^2
\end{align}
On the other hand,
\begin{align*}
||\chi _R \nu(t)|| ^2_{ M _{ 0} }&= ||\chi _0 V(\tau )||^2 _{ L^2 }=\int  _{ \RR }\left|\widehat{\chi _0 V(\tau )}(k) \right|^2dk\\
& = \int  _{ \RR }\left|\widehat{\chi _0 V(\tau )}(k) \right|^2 dk= \int  _{ \RR }\left| \chi _0 (\zeta ) V(\tau )(\zeta ) \right|^2 d\zeta \\
&\le \int  _{I_2  }\left| V(\tau, \zeta ) \right|^2 dk=\int  _{ RJ_2 }|\nu(t, X)|^2\frac {dX} {X}\le \frac {C} {R}\int  _{ RJ_2 }|\nu(t, X)|^2dX.
\end{align*}
Then for $R\in (0, 1)$,
\begin{align*}
& \sup _ {\substack{ 0< t_0 < T^*\\ 0<R<1}}\Bigg(\int  _{ t_0}^{\min(t_0+R^{1/2}, T_*)}  \left(R^{-1}||v(t)||^2 _{ L^2(RJ_3) }+R^{-2}[[v (t )]]^2 _{RJ_3 }\right)dt \Bigg)^{1/2} \le \\
&\le C\left(1+T_*^{2(1-\theta)}\right)\left(\int _0^{T_*} \left(1+t^{-4\theta}\right) ||\nu(t)|| ^2_{ \theta, \rho  } dt\right)^{1/2}
+\\
&+C \sup _ {\substack{ 0< t_0 < T^*\\ 0<R<1}}  \left(  \int  _{ t_0 }^{\min(t_0+R^{1/2}, T_*)}||  \nu (t)||^2 _{ L^2(RJ_2)}dt  \right)^{1/2}
\end{align*}
and (\ref{sgfrd1}) follows.
\end{proof}
\begin{rem}
\label{theta14}
If $\nu \in L^\infty ((0, T); X _{ \theta, \rho  })$ then Theorem \ref{S7cor2} may be used instead of Corollary \ref{S7cor1} to obtain $||v(t)|| _{ \infty  }\le C\sup _{ 0\le s\le t }||\nu(s)|| _{ \theta, \rho  }(1+t^{-2\theta})$ and then, for $\theta<1/2$, the first term in the right hand sides of (\ref{mtheoE2a}), (\ref{mtheoE2b}) and (\ref{mtheoE2c}) may be substituted by $C(T^*+T*^{1-2\theta})\sup _{ 0\le s\le T^* }||\nu(s)||^2 _{ \theta, \rho  }$.
\end{rem}
\section{Appendix: Some Lemmas.}
\label{somelemmas}
\setcounter{equation}{0}
\setcounter{theo}{0}
Let us define,
\begin{align*}
&\widehat {-T_1(h)}(k)=\hat h(k)\mathscr R e(\rho _0(k))\1 _{ |k|>1 }\\
&\widehat {-T_2(h)}(k)=\hat h(k)\1 _{ |k|\le 1 }\\
&\widehat {\mu _\sigma h}(k)=|k|^{\sigma }\widehat h(k),\,\,\widehat {\mu  _{ \sigma , {\rho } }h}(k)=|k|^\sigma (1+\log (1+|k|)^\rho \widehat h(k)\\
&\mathscr F(\mu  _{ \sigma , \log^{-\rho } }h)(k)=|k|^\sigma (1+\log (1+|k|)^{-\rho  },\\
&\nu(k)=|k|^2\1 _{ |k|<1 }+\log |k|\,\1 _{ |k|>1 },
\end{align*}

The first Lemma is a small variation of Lemma 3.7  in   \cite{EscVelTransAMS}.

\begin{lem}
\label{SapPPCM0}
 Suppose $\alpha _0\in C^\infty_c(\RR)$ is such that, for some constants $C>0$ and $\varepsilon >0$,
\begin{align}
|\widehat \alpha_0 (k)|&\le \frac {C \delta^2  _{ \xi _0 }} { 1+\delta^2  _{ \xi _0 }k^2}\,\,\forall k\in \RR.\label{S3.3Fa21}
\end{align}
Then there exists positive constants $K$ and $C$, with $K$ independent of $\varepsilon $ such that,
\begin{align*}
||\alpha _0 f|| _{ H^\sigma  (\RR)}+||T_1\alpha _0 f|| _{ H^\sigma  (\RR)}\le K\varepsilon ||f|| _{ H^\sigma  (\RR)}+C||f|| _{ L^\infty(\RR) }
\end{align*}
\end{lem}
\begin{proof}
By hypothesis, condition (3.39) of Lemma 3.7 in \cite{EscVelTransAMS} is fulfilled. The argument may then proceed unchanged until the end where the sup norm may be replaced by the norm in $L^2$.
\end{proof}

\begin{lem}
\label{S7Lapbc}
For all $\eta\in C_c^\infty(\RR)$, there exists a constant $C$ such that
\begin{align}
&||S _{ \xi _0 }(t-s)\left([\eta, P_0](\varphi )\right)|| _{ H^\sigma  }^2\le C || \varphi  ||^2_{H _{ \sigma -\rho (t-s)}} \label{S7LapbcE1}\\
&\int _0^1\left|\left| \int _0^t T_1e^{\kappa _0(t-s)T_1} \left([\eta, P_0](\varphi )\right) \right|\right| _{ H^\sigma  }^2ds 
\le C \int _0^1\int _0^t||T_1\varphi (s)  ||^2 _{ H _{ \sigma -\rho(t-s)}}dsdt.\label{S7LapbcE2}
\end{align}
for $\rho (t)\in (0, \kappa _0 t)$.
\end{lem}
\begin{proof}

By definition,
\begin{align}
\label{Subsub3E1}
||S _{ \xi _0 }(t)[\eta, P_0](\varphi )|| _{ H^\sigma  }^2&=\int  _{ \RR } |\mathscr F\left( S _{ \xi _0 }(t)[\eta,  P_0](\varphi )\right)(k)|^2 
(1+|k|^2)^\sigma dk\nonumber\\
&=\int  _{ |k|<1 }[\cdots]dk+\int  _{ |k|>1 }[\cdots]dk=A_1+A_2.
\end{align}
with,
\begin{align*}
\mathscr F\left( \Big[\eta, P_0 \Big]( \varphi ) \right)(k)&=\widehat \eta \ast  \widehat{ P_0\phi }(k)- \rho _0 \widehat{\eta \varphi }=
 \left(\widehat \eta \ast (\rho _0 \widehat \varphi )(k)-\rho _0 (\widehat\eta \ast \widehat  \varphi)(k)\right)\\
& \int  _{ \RR }\left( \widehat \eta(k-k')\rho _0(k')\widehat \varphi (k')-\varphi _0(k)\widehat \eta(k-k')\widehat \varphi (k') \right)dk'\\
&= \int  _{ \RR }\left( \rho _0(k')-\rho _0(k)\right) \widehat \eta(k-k')\widehat \varphi (k')dk'
\end{align*}
It will be also be used that by hypothesis, if $K(k)=\widehat \eta(k)$ then  for all $m\in \NN$ there exists a constant $C_m>0$ such that, 
\begin{align*}
|K(k)|\le \frac {C_m} {1+|k|^m},\,\,\forall k\in \RR.
\end{align*}
In the first integral $A_1$ in (\ref{Subsub3E1}), where $|k|<1$, we may use that
\begin{align}
|\rho _0(k)|\le  C\log(1+|k|),\,\,\forall k\in \RR \label{Subsub3E2}
\end{align}
\begin{align}
\label{S7EA1}
\int  _{|k|<1 } e^{-t \kappa _0\rho _0(k)}(1+|k|^2)^\sigma\left| \int  _{ \RR }\left( \rho _0(k')-\rho _0(k)\right)  K(k-k')\widehat \varphi (k') dk'\right|^2dk\nonumber\\
\le  C\int  _{|k|<1 }\left( \int  _{ \RR }\frac {1+\left|\log(1+|k'|)\right)} {1+|k-k'|^m} \widehat \varphi (k')dk'\right)^2dk\le C||\varphi ||_2^2.
\end{align}
The second integral is split again as follows,
\begin{align*}
A_2=\int  _{ |k|>1,\,|k'|<1 }[\cdots]dk'dk+\int  _{ |k|>1,\,|k'|>1 }[\cdots]dk'dk=J_1+J_2.
\end{align*}
Property (\ref{Subsub3E2}) is used again in $J_1$,
\begin{align}
\label{S7EJ1}
J_1\le C\int  _{|k|>1 }(1+|k|^2)^\sigma\left|\int  _{ |k'|<1 }  \frac {1+\log(1+|k|)} {1+|k-k'|^m}\widehat \varphi (k') dk'\right|^2dk \nonumber\\
\le C\int  _{|k|>1 }\frac {1+\log(1+|k|)} {(1+|k|^2)^{m-\sigma }}\left|\int  _{ |k'|<1 }  \widehat \varphi (k') dk'\right|^2dk\le C||\varphi ||_2^2.
\end{align}
In the term $J_2$, for $\rho =\rho (t)$ to be chosen,
\begin{align*}
J_2\le ||\varphi ||^2 _{ H _{ \sigma -\rho  } }\int  _{|k|>1 }\exp\left(-t \kappa _0 \mathscr Re (\rho _0(k) )\right)(1+|k|^2)^\sigma\times \\
\times   \int  _{ |k'|>1 }\frac {\left| \rho _0(k')-\rho _0(k)\right|^2 }  {(1+|k-k'|^m)^2(1+|k'|^2)^{(\sigma -\rho )}} dk' dk
\end{align*} 
For $|k|>1$ and $|k'|>1$, 
\begin{align*}
| \rho _0(k')-\rho _0(k)|\le C| \rho _0(|k'|)-\rho _0(|k|)|\le C|k-k'|\rho_0 '(z)\le C\frac {|k-k'|} {z},
\end{align*}
for some $z>\min (|k|, |k'|)$, and by Proposition \ref{S2PP1},
\begin{align*}
\exp\Big(-t \kappa _0 \mathscr Re (\rho _0(k) )\Big)\le C \exp\left(- t \kappa _0 \log |k|\right)=C|k|^{-\kappa _0 t }
\end{align*}
from where,
\begin{align}
\label{S7EJJJ2}
J_2\le ||\varphi ||^2 _{ H _{ \sigma -\rho  } } \int  _{ |k'|>1 }\frac {dk'}  {(1+|k'|^2)^{(\sigma -\rho )}} 
\int  _{|k|>1 }  \frac {|k|^{2\sigma -\kappa _0 t }dk } {z(1+|k-k'|^{2(m-1)})}.
\end{align} 
Notice now that, for all $|k'|>1$,
\begin{align*}
\int  _{|k|>1 }  \frac {|k|^{2\sigma -\kappa _0 t } dk} {z(1+|k-k'|^{2(m-1)})}&=\1 _{ |k'|>|k| }\int  _{ 1<|k|<|k'| }[\cdots]dk+\\
+\int  _{ |k'|<|k|<8|k'| }[\cdots]dk&+\int  _{  k|>8|k'| }[\cdots]dk=I_1+I_2+I_3.
\end{align*}
In $I_1$, $z>|k|$ and then,
\begin{align*}
\int  _{1<|k|<|k'| }  \frac {|k|^{2\sigma -\kappa _0 t } dk} {z(1+|k-k'|^{2(m-1)})}\le
C |k'|^{-\kappa _0 t -1}\int  _{1<|k|<|k'| }  \frac {|k|^{2\sigma} dk} {(1+|k-k'|^{2(m-1)})}\\
\le C |k'|^{-\kappa _0 t -1}\int  _{1<|k|<|k'| }  \frac {|k|^{2\sigma} dk} {1+(|k'|-|k|)^{2(m-1)}}\\
=C |k'|^{-\kappa _0 t -1}\int  _{1<|k|<|k'| }  \frac {|k|^{2\sigma} d|k|} {1+(|k'|-|k|)^{2(m-1)}}\\
=
C |k'|^{-\kappa _0 t -1}\int _0^{|k'|-1}  \frac {(|k'|-r)^{2\sigma} dr} {(1+r)^{2(m-1)}}\\
\le C |k'|^{-\kappa _0 t -1+2\sigma }\int _0^{|k'|-1}  \frac { dr} {(1+r)^{2(m-1)}}\le C |k'|^{-\kappa _0 t -1+2\sigma }.
\end{align*} 
It follows, for $\rho =\rho (t)<\kappa _0 t$,
\begin{align*}
&\int  _{ |k'|>1 }\frac {dk'}  {(1+|k'|^2)^{(\sigma -\rho )}} 
\int  _{1<|k|<|k'| }  \frac {|k|^{2\sigma -\kappa _0 t }dk } {z(1+|k-k'|^{2(m-1)})}\le \\
&\le C \int  _{ |k'|>1 }\frac { |k'|^{-\kappa _0 t -1+2\sigma }dk'}  {|k'|^{2(\sigma -\rho )}}=
 C \int  _{ |k'|>1 } |k'|^{-\kappa _0 t -1+\rho  }dk'<\infty.
\end{align*}
In $I_2$,  the argument is similar although $z>|k'|>1$, 
\begin{align*}
\int  _{|k'|<|k| <8|k'|}  \frac {|k|^{2\sigma -\kappa _0 t } dk} {z(1+|k-k'|^{2(m-1)})}\le
C |k'|^{-\kappa _0 t -1}\int  _{|k'|<|k| }  \frac {|k|^{2\sigma} dk} {(1+|k-k'|^{2(m-1)})}\\
\le C |k'|^{-\kappa _0 t -1+2\sigma }.
\end{align*} 
In $I_3$, for $m>0$ large enough,
\begin{align*}
\int  _{|k|>8|k'| }  \frac {|k|^{2\sigma -\kappa _0 t } dk} {z(1+|k-k'|^{2(m-1)})}\le
C |k'|^{ -1}\int  _{|k|>8|k'| }  \frac {|k|^{2\sigma-\kappa _0t} dk} {|k|^{2(m-1)}}\\
\le C |k'|^{ -1}\int  _{|k|>8|k'| }  |k|^{2\sigma -\kappa _0t-2m+2}= C |k'|^{2\sigma -\kappa _0t-2m+2}
\end{align*} 
and,
\begin{align*}
\int  _{ |k'|>1 }\frac {dk'}  {(1+|k'|^2)^{(\sigma -\rho )}} \int  _{|k|>8|k'| }  \frac {|k|^{2\sigma -\kappa _0 t } dk} {z(1+|k-k'|^{2(m-1)})}\le
C\int  _{ |k'|>1 }\frac { |k'|^{2\sigma -\kappa _0t-2m+2}dk'}  {|k'|^{2(\sigma -\rho )}}. 
\end{align*}
It then follows for $\rho (t)<\kappa _0 t$,
 \begin{align}
\label{S7EJJJ2B}
J_2\le C ||\varphi ||^2 _{ H _{ \sigma -\rho(t)  } }
\end{align} 
and by (\ref{S7EA1}), (\ref{S7EJ1}) and (\ref{S7EJJJ2B}), for $\rho (t)\in (0, \kappa _0 t)$.
\begin{align*}
||S _{ \xi _0 }(t)\left([\eta, P_0](\varphi  )\right)|| _{ H^\sigma  }^2\le C ||\varphi ||^2_{H _{ \sigma -\rho (t)}}.
\end{align*}
On the other hand, the left hand of (\ref{S7LapbcE2}) is estimated first as,
\begin{align*}
\int _0^1\int _0^t\int  _{ \RR }e^{2\kappa _0(t-s)\mathscr Re (\rho _0(k)}|\mathscr Re (\rho _0(k)|^2(1+|k|^2)^\sigma \times \\
\times \left|\int  _{ \RR }(\rho _0(k')-\rho _0(k))\widehat \eta (k-k')\widehat \varphi (s, k')dk' \right|^2dkdsdt.
\end{align*}
Then, the same argument as in the proof of (\ref{S7EA1}) and (\ref{S7EJ1}) shows,
\begin{align*}
&\int _0^1\int _0^t\int  _{ \RR }e^{2\kappa _0(t-s)\mathscr Re (\rho _0(k)}|\mathscr Re (\rho _0(k)|^2(1+|k|^2)^\sigma \times \\
&\times \left|\int  _{ \RR }\1 _{ \min(|k|, |k'|)<1 }(\rho _0(k')-\rho _0(k))\widehat \eta (k-k')\widehat \varphi (s, k')dk' \right|^2dkdsdt\le C\int _0^1||h(s)|| _{2 }^2ds.
\end{align*}
When $|k|>1$ and $|k'|>1$, by  Proposition \ref{S2PP1}
\begin{align*}
\int _0^1&\int _0^t\int  _{ |k|>1 }e^{2\kappa _0(t-s)\mathscr Re (\rho _0(k)}|\mathscr Re (\rho _0(k)|^2(1+|k|^2)^\sigma \times \\
&\times \left|\int  _{ |k'|>1 }|\rho _0(k')-\rho _0(k))\widehat \eta (k-k')\widehat \varphi (s, k')dk' \right|^2dkdsdt
\le I_1+I_2
\end{align*}
with
\begin{align*}
I_1&=\int _0^1\int _0^t \int  _{ |k|>1 }e^{2\kappa _0(t-s)\mathscr Re (\rho _0(k)}(1+|k|^2)^\sigma \times \\
&\times \left|\int  _{ |k'|>1 }|\rho _0(k')-\rho _0(k)| |\mathscr Re \rho _0(k')|\widehat \eta (k-k')\widehat \varphi (s, k')dk' \right|^2dkdsdt
\end{align*}
and
\begin{align*}
I_2&=\int _0^1\int _0^t\int  _{ |k|>1 }e^{2\kappa _0(t-s)\mathscr Re (\rho _0(k)}(1+|k|^2)^\sigma \times \\
&\times \left|\int  _{ |k'|>1 }|\rho _0(k')-\rho _0(k)|^2\widehat \eta (k-k')\widehat \varphi (s, k')dk' \right|^2dkdsdt
\end{align*}
In the term $I_1$, one obtains first, 
\begin{align*}
&\int  _{ |k|>1 }e^{2\kappa _0(t-s)\mathscr Re (\rho _0(k)}(1+|k|^2)^\sigma\times \\
&\times   \left|\int  _{ |k'|>1 }|\rho _0(k')-\rho _0(k)| |\mathscr Re\rho _0(k')|\widehat \eta (k-k')\widehat \varphi (s, k')dk' \right|^2dk\\
&\le C \int  _{ |k|>1 }|k|^{2\kappa _0(t-s)}(1+|k|^2)^\sigma \times \\
&\times  \left|\int  _{ |k'|>1 }|\rho _0(k')-\rho _0(k)| \widehat \eta (k-k')|\mathscr Re\rho _0(k')|\widehat \varphi (s, k')dk' \right|^2dk
\end{align*}
then, as for  (\ref{S7EJJJ2B}), the existence of a constant $C>0$ such that for every $t\in (0, 1)$, $s\in (0, t)$ and $\rho (t-s)\in (0, \kappa _0(t-s))$,
\begin{align*}
&\int  _{ |k|>1 }e^{2\kappa _0(t-s)\mathscr Re (\rho _0(k)}(1+|k|^2)^\sigma\times \\
&\times   \left|\int  _{ |k'|>1 }|\rho _0(k')-\rho _0(k)| |\mathscr Re\rho _0(k')|\widehat \eta (k-k')\widehat \varphi (s, k')dk' \right|^2dk\\
&\le C  \int  _{ \RR }(1+|k'|^2)^{\sigma -\rho (t-s)} |\mathscr Re\rho _0(k')|\widehat \varphi (s, k')dk'=C||T_1 \varphi ||^2 _{ H^{\sigma -\rho (t-s)} }
\end{align*}
and then, after integration in $s$ and $t$,
\begin{align*}
I_1
\le C \int _0^1\int _0^t||T_1\varphi (s)  ||^2 _{ H _{ \sigma -\rho(t-s)}}dsdt.
\end{align*}
A similar argument gives in $I_2$, for $m>0$ as large as desired, the existence of a constant $C>0$ such that if $z\ge \min (|k|, |k'|)$, $t\in (0, 1)$ and $s\in (0, t)$,
\begin{align*}
\int  _{ |k|>1 }e^{2\kappa _0(t-s)\mathscr Re (\rho _0(k)}(1+|k|^2)^\sigma  \left|\int  _{ |k'|>1 }|\rho _0(k')-\rho _0(k)|^2\widehat \eta (k-k')\widehat \varphi (s, k')dk' \right|^2dk\\
\le ||T_1\varphi (s)|| _{ H^{\sigma -\rho (t-s)} }^2\int  _{ |k'|>1 }\frac {|\mathscr Re(\rho _0(k')|^{-2}dk'}  {(1+|k'|^2)^{(\sigma -\rho )}} 
\int  _{|k|>1 }  \frac {|k|^{2\sigma -\kappa _0 t }dk } {z(1+|k-k'|^{2(m-1)})}.
\end{align*}
Proceeding as for the estimate (\ref{S7EJJJ2B}) we deduce
\begin{align*}
I_2
\le C \int _0^1\int _0^t||T_1\varphi (s)  ||^2 _{ H _{ \sigma -\rho(t-s)}}dsdt.
\end{align*}
and (\ref{S7LapbcE2}) follows.
\end{proof}

\begin{lem}
\label{SapPGCD0}
\begin{align*}
\int _0^1 \left|\left|\int _0^t  T_1e^{\kappa _0 T_1(t-s)}h(s)ds\right|\right| ^2_{ H^\sigma  }dt\le C \int _0^1||h(t)||^2 _{ H^\sigma }dt
\end{align*}
\end{lem}
\begin{proof}
 define,
\begin{align*}
\varphi (t, \xi )=\int _0^t   e^{\kappa _0  T_1(t-s)} h(s)ds
\end{align*}
this function satisfies,
\begin{align*}
\frac {\partial \varphi } {\partial t}=\kappa _0 T_1(\varphi )+h,\,\,\varphi (0, \xi )=0.
\end{align*}
Multiplication of  both sides by  $-  T_1 (M _{ 2\sigma  }\varphi )(\xi )$ and integration over $\RR$  gives,
\begin{align*}
-\frac {1 } {2}\frac {d} {dt}\int _{ \RR } T_1 (M _{ 2\sigma  }\varphi )(\xi )\varphi  (\xi )d\xi 
+\kappa _0\int  _{ \RR }T_1(M _{ 2\sigma  }\varphi )(\xi ) T_1(\varphi )(\xi )d\xi =\\
=-  \int  _{ \RR }T_1(M _{ 2\sigma  }(\varphi )(\xi )h(t, \xi )d\xi. 
\end{align*}
By Plancherel's Theorem,
\begin{align*}
&\int _{ \RR }(- T_1 (M _{ 2\sigma  }\varphi )(\xi ))\varphi  (\xi )d\xi =\int  _{ \RR } \mathscr F(-T_1 (M _{ 2\sigma  }\varphi ))(k )\overline{ \mathscr F(\varphi  )(k )}dk\\
&=\int  _{ |k|\ge 1 }|k|^{2\sigma }\mathscr Re(\rho _0(k)) \left|\mathscr F(\varphi  )(k )\right|^2dk=\int  _{ |k|\ge 1}\left||k|^{\sigma } \mathscr Re(\rho _0(k))^{1/2} \hat \varphi (k)\right|^2dk\\
&=\int  _{ \RR }\left|M_\sigma (\sqrt{-T_1}(\varphi ))(\xi )\right|^2d\xi, 
\end{align*}

\begin{align*}
&\int _{ \RR }(T_1 (M _{ 2\sigma  }\varphi )(\xi ))T_1(\varphi  )(\xi )d\xi 
=\int  _{ \RR } \mathscr F(T_1 (M _{ 2\sigma  }\varphi ))(k )\overline{ \mathscr F(T_1(\varphi)  )(k )}dk\\
&=\int  _{ |k|\ge 1 }|k|^{2\sigma }\mathscr Re(\rho _0(k))^2 \left|\mathscr F(\varphi  )(k )\right|^2dk=\int  _{ |k|\ge 1 }\left||k|^{\sigma } \mathscr Re(\rho _0(k)) \hat \varphi (k)\right|^2dk\\
&=\int  _{ \RR }\left|M_\sigma T_1(\varphi )(\xi )\right|^2d\xi, 
\end{align*}

\begin{align*}
&- \int  _{ \RR }T_1(M _{ 2\sigma  }(\varphi )(\xi )h(t, \xi )d\xi=- \int  _{ |k|\ge 1 }|k|^{2\sigma }\mathscr Re(\rho _0(k))\widehat {\varphi} (k) \overline {\widehat h} (t, k)dk\\
&\Longrightarrow \left|  \int  _{ \RR }T_1(M _{ 2\sigma  }(\varphi )(\xi )h(t, \xi )d\xi\right|\le \left(\int  _{ |k|\ge 1 } |k|^{2\sigma }\mathscr Re(\rho _0(k))^2 |\widehat \varphi (k)|^2dk\right)^{1/2}\times \\
&\hskip 7cm \times \left(\int  _{ |k|\ge 1 } |k|^{2\sigma }  |\widehat h (t, k)|^2dk\right)^{1/2}\\
&=\left( \int  _{ \RR }|M_\sigma T_1(\varphi )|^2d\xi \right)^{1/2}\left( \int  _{ \RR }|M_\sigma  (h(t, \xi ) )|^2d\xi \right)^{1/2}.
\end{align*}
It then follows,
\begin{align*}
\frac {1 } {2}\frac {d} {dt}\left|\left|M_\sigma \sqrt{-T_1}(\varphi) \right|\right|^2_2+\kappa _0\left|\left|M_\sigma  T_1(\varphi )\right|\right|^2_2
\le \left|\left|M_\sigma  T_1(\varphi )\right|\right|_2||M_\sigma (h(t))||_2\\
\le \frac { \kappa _0} {2}\left|\left|M_\sigma  T_1(\varphi )\right|\right|_2^2+ \frac {1} {2\kappa _0}||M_\sigma (h(t))||_2^2
\end{align*}
and,
\begin{align*}
\frac {1 } {2}\frac {d} {dt}\left|\left|M_\sigma \sqrt{-T_1}(\varphi) \right|\right|^2_2+\frac {\kappa _0} {2}\left|\left|M_\sigma  T_1(\varphi )\right|\right|^2_2
\le   \frac {1} {2\kappa _0}||M_\sigma (h(t))||_2^2.
\end{align*}
On the other hand, multiplication by $- T_1(\varphi )$ and similar arguments give,
 \begin{align*}
&-\frac {1 } {2}\frac {d} {dt}\int _{ \RR } T_1 (\varphi )(\xi )\varphi  (\xi )d\xi 
+\kappa _0\int  _{ \RR }T_1(\varphi )(\xi ) T_1(\varphi )(\xi )d\xi =-  \int  _{ \RR }T_1(\varphi )(\xi )h(t, \xi )d\xi\\
&\Longleftrightarrow\\
&\frac {1 } {2}\frac {d} {dt} ||\sqrt {-T_1}(\varphi )||_2^2+  \kappa _0 ||T_1(\varphi )||_2^2\le 
\left( ||T_1(\varphi )||_2\right)||h(t)||_2\\
&\le \frac {\kappa _0} {2}\left|\left|  T_1(\varphi )\right|\right|_2^2+  \frac {1} {2\kappa _0}||  h(t)||_2^2
\end{align*}
and,
\begin{align*}
\frac {1 } {2}\frac {d} {dt}\left|\left| \sqrt{-T_1}(\varphi) \right|\right|^2_2+\frac {\kappa _0} {2}\left|\left| T_1(\varphi )\right|\right|^2_2
\le   \frac {1} {2\kappa _0}|| h(t)||_2^2.
\end{align*}
Both estimates show,
\begin{align*}
\frac {1 } {2}\frac {d} {dt}\left|\left| \sqrt{-T_1}(\varphi(t)) \right|\right|^2 _{ H^\sigma  }+\frac { \kappa _0} {2}\left|\left| T_1(\varphi (t))\right|\right|^2 _{ H^\sigma  }
\le   \frac {1} {2\kappa _0}|| h(t)|| _{ H^\sigma  }^2
\end{align*}
and integration in time, using also that $\varphi (0, \xi )=0$ gives
\begin{align*}
\frac {1 } {2}\left|\left| \sqrt{-T_1}(\varphi)(t) \right|\right|^2 _{ H^\sigma  }+\frac { \kappa _0} {2}\int _0^t\left|\left| T_1(\varphi(s) )\right|\right|^2 _{ H^\sigma  }ds
\le  \frac {1} {2\kappa _0} \int _0^t || h(s)|| _{ H^\sigma  }^2dt.
\end{align*}
\vskip -0.5cm
\end{proof}

\begin{lem}
\label{SapLPGCD1}
\begin{align*}
\int _0^1 ||\int _0^t e^{\kappa _0 T_1(t-s)}h(s)ds|| ^2_{ H^\sigma  }dt\le C \int _0^1||h(t)||^2 _{ H^\sigma  _{ \log^{-1} } }dt
\end{align*}
\end{lem}
\begin{proof}
Use of the notation $e^{-\frac {\xi _0} {2}}=\kappa _0$ gives,
\begin{align*}
&\left|\left|\int _0^tS _{ \xi _0 }(t-s)h(s) ds\right|\right|^2 _{ H^\sigma  }=
\left|\left|\int _0^t e^{\kappa _0P_0(t-s)} h(s) ds \right|\right| _{ H^\sigma  }^2\\
&=\int  _{ \RR }\left(\int _0^t e^{-\kappa _0 (t-s_1)\rho _0(k)} \widehat h(s_1, k)ds_1\int _0^t e^{-\kappa _0 (t-s_2)\overline {\rho _0}(k)} \overline{\widehat h(s_2, k)}ds_2\right)(1+|k|^2)^\sigma dk\\
&\le \int  _{ \RR }\left(\int _0^t e^{-\kappa _0 (t-s_1)\mathscr Re \rho _0(k)} |\widehat h(s_1, k)|^2ds_1\right)^2(1+|k|^2)^\sigma dk.
\end{align*}
Consider
\begin{align*}
\varphi (t, \xi )=\int _0^t   e^{\kappa _0 T_1(t-s)} h(s)ds
\end{align*}
this function satisfies,
\begin{align*}
\frac {\partial \varphi } {\partial t}= \kappa _0  T_1(\varphi )+h,\,\,\varphi (0, \xi )=0.
\end{align*}
Multiplication of  both sides by  $(\mu  _{ 2\sigma , \log^{-2} }T_1(\varphi )(\xi ))$ and integration over $\RR$  gives,
\begin{align*}
\frac {d} {dt}\int  _{ \RR } |\widehat \varphi (t, k)|^2 \frac {(1+|k|^2)^\sigma \mathscr Re \rho _0(k)} {(1+\log (1+|k|)^2}dk
+\frac {\kappa _0} {2}\int  _{ \RR } \frac {(1+|k|^2)^\sigma | \widehat \varphi (k)|^2} {(1+\log (1+|k|))^2} (\mathscr Re\rho _0(k))^2dk =\\
=\int  _{ \RR } \frac {(1+|k|^2)^\sigma } {(1+\log (1+|k|))^2}  \mathscr Re(\rho _0(k)  \widehat \varphi (k) \widehat h(t, k)dk 
\end{align*}
where by Proposition \ref{S2PP1}, there exists a constant $C>0$ such that  the right hand side is bounded as,
\begin{align*}
\int  _{ \RR } \frac {(1+|k|^2)^\sigma } {(1+\log (1+|k|))^2}  \mathscr Re(\rho _0(k)  \widehat \varphi (k) \widehat h(t, k)dk\le\\
\le C\int  _{ \RR } \frac {(1+|k|^2)^\sigma  \widehat \varphi (k) \widehat h(t, k)} {(1+\log (1+|k|))^2} \nu(k) ^2dk
\end{align*}
On the other hand, by Proposition \ref{S2PP1},
\begin{align*}
\int  _{ \RR } \frac {(1+|k|^2)^\sigma | \widehat \varphi (k)|^2} {(1+\log (1+|k|))^2} (\mathscr Re\rho _0(k))^2dk\ge 
C \int  _{ \RR } \frac {(1+|k|^2)^\sigma  \widehat \varphi (k) \widehat h(t, k)} {(1+\log (1+|k|))^2} \nu(k) ^2dk\end{align*}
It follows
\begin{align*}
\frac {d} {dt}\int  _{ \RR } |\widehat \varphi (t, k)|^2 \frac {(1+|k|^2)^\sigma \mathscr Re \rho _0(k)} {(1+\log (1+|k|)^2}dk
+C\kappa _0\int  _{ \RR } \frac {(1+|k|^2)^\sigma | \widehat \varphi (k)|^2} {(1+\log (1+|k|))^2} \nu(k)^2dk\le \\
\le \int  _{ \RR } \frac {(1+|k|^2)^\sigma  \widehat \varphi (k) \widehat h(t, k)} {(1+\log (1+|k|))^2} \nu(k) ^2dk
\end{align*}
and
\begin{align*}
\frac {d} {dt}\int  _{ \RR } |\widehat \varphi (t, k)|^2 \frac {(1+|k|^2)^\sigma \mathscr Re \rho _0(k)} {(1+\log (1+|k|)^2}dk
+\frac {C\kappa _0} {2}\int  _{ \RR } \frac {(1+|k|^2)^\sigma | \widehat \varphi (k)|^2} {(1+\log (1+|k|))^2} \nu(k)^2dk\le \\
\le \frac {1} {2C\kappa _0}\int  _{ \RR } \frac {(1+|k|^2)^\sigma |\widehat h(t, k)|^2} {(1+\log (1+|k|))}dk\le C\kappa _0^{-1}||h ||^2 _{ H^\sigma _{ \log^{-1} }  }
\end{align*}
Integration in time and use of $\varphi (0, \xi )=0$ for all  $\xi \in \RR$ give
\begin{align*}
\int _0^t \int  _{ \RR } \frac {(1+|k|^2)^\sigma | \widehat \varphi (k)|^2} {(1+\log (1+|k|))^2} \nu(k)^2dkds\le C\kappa _0^{-2} \int _0^t||h ||^2 _{ H^\sigma _{ \log^{-1} }  }ds
\end{align*}
from where,
\begin{align*}
\int _0^t \int  _{ |k|\ge1 } (1+|k|^2)^\sigma | \widehat \varphi (k)|^2dkds\le C\kappa _0^{-2} \int _0^t||h ||^2 _{ H^\sigma _{ \log^{-1} }  }ds.
\end{align*}

On the other hand, for $s\in (0, t)$,
\begin{align*}
\int  _{ |k|<1 }|\widehat \varphi (s, k)|^2(1+|k|^2)^\sigma dk
\le  \int  _{ |k|<1 } \left|\int _0^s e^{-\kappa _0\mathscr Re(\rho _0(k))}\widehat h(\tau , k)d\tau \right|^2(1+|k|^2)^\sigma dk\\
\le  s\int  _{ |k|<1 }\int _0^s  \left| e^{-\kappa _0\mathscr Re(\rho _0(k))}\right|^2 |\widehat h(\tau , k)| ^2d\tau (1+|k|^2)^\sigma dk\\
\le s\int _0^s \int  _{ |k|<1 } |\widehat h(\tau , k)| ^2(1+|k|^2)^\sigma dk d\tau \\
\le Cs\int _0^s \int  _{ |k|<1 } |\widehat h(\tau , k)| ^2(\frac {1+|k|^2)^\sigma } {(1+\log (1+|k|))}dk ds
= Cs\int _0^s ||h(\tau )||^2 _{ H^\sigma  _{ \log^{-1} } }d\tau. 
\end{align*}
We deduce,
\begin{align*}
\int _0^t \int  _{ |k|<1 }|\widehat \varphi (s, k)|^2(1+|k|^2)^\sigma dkds
\le C\int _0^ts\int _0^s ||h(\tau )||^2 _{ H^\sigma  _{ \log^{-1} } }d\tau ds\\
\le C\int _0^t(t-s) ||h(s )||^2 _{ H^\sigma  _{ \log^{-1} } }ds
\end{align*}
and,
\begin{align*}
\int _0^t ||\varphi || ^2_{ H^\sigma  }d\sigma \le C\int _0^t(t-s) ||h(s )||^2 _{ H^\sigma  _{ \log^{-1} } }ds
+C\kappa _0^{-2} \int _0^t||h(s) ||^2 _{ H^\sigma _{ \log^{-1} }  }ds.
\end{align*}
\vskip -0.5cm 
\end{proof}

\begin{lem}
\label{SapLPGCD2}
Given $\sigma >0$ and $\beta >1$, there exists a constant $C>0$ such that, for all $Q\in H^\sigma  _{ \log^{-1} }$ and  $m\in H^{\sigma +\beta }\cap H^{1+\beta }$, 
\begin{align*}
||mQ|| _{ H^\sigma  _{ \log^{-1} } }\le C||Q|| _{ H^\sigma  _{ \log^{-1} } }\left(||m||  _{ H^{\sigma +\beta }}+||m||  _{ H^{1 +\beta }} \right)
\end{align*}
\end{lem}
\begin{proof}
By definition
\begin{align*}
||mQ|| _{ H^\sigma  _{ \log^{-1} } }^2=\int  _{ \RR }|\widehat {mQ}(k)|^2\frac {(1+|k|^2)^{\sigma }dk} {(1+\log(1+|k|))}\\
\end{align*}
\begin{align*}
\frac {(1+|k|^2)^{\sigma /2}} {(1+\log(1+|k|))^{1/2}}\le C\frac {(1+|k-\ell|^2)^{\sigma /2}+(1+|\ell|^2)^{\sigma /2}} {(1+\log(1+|k|))^{1/2}}\\
\le C (1+|k-\ell|^2)^{\sigma /2}+C\frac {(1+|\ell|^2)^{\sigma /2}} {(1+\log(1+|k|))^{1/2}}
\end{align*}
Since,
\begin{align*}
\sup _{ (k, \ell)\in \RR^2 }C\frac {(1+|k-\ell|^2)^{-1 /2} (1+\log(1+|\ell|))^{1/2}} {(1+\log(1+|k|))^{1/2}}=C'<\infty\\
\Longrightarrow C\frac {(1+|\ell|^2)^{\sigma /2}} {(1+\log(1+|k|))^{1/2}}\le C'\frac {(1+|\ell|^2)^{\sigma /2}} {(1+\log(1+|\ell|))^{1/2}}(1+|k-\ell|^2)^{1 /2}
\end{align*}
then,
\begin{align*}
\frac {(1+|k|^2)^{\sigma /2}} {(1+\log(1+|k|))^{1/2}}\le C\frac {(1+|k-\ell|^2)^{\sigma /2}+(1+|\ell|^2)^{\sigma /2}} {(1+\log(1+|k|))^{1/2}}\\
\le C (1+|k-\ell|^2)^{\sigma /2}+C'\frac {(1+|\ell|^2)^{\sigma /2}} {(1+\log(1+|\ell|))^{1/2}}(1+|k-\ell|^2)^{1 /2}.
\end{align*}
It follows,
\begin{align*}
|\widehat m(k-\ell)||\widehat Q(\ell)|\frac {(1+|k|^2)^{\sigma /2}} {(1+\log(1+|k|))^{1/2}}\le C (1+|k-\ell|^2)^{\sigma /2}|\widehat m(k-\ell)||\widehat Q(\ell)|+\\
+C'(1+|k-\ell|^2)^{1 /2}|\widehat m(k-\ell)| \frac {(1+|\ell|^2)^{\sigma /2}} {(1+\log(1+|\ell|))^{1/2}}|\widehat Q(\ell)|
\end{align*}
and,
\begin{align*}
||mQ|| _{ H^\sigma  _{ \log^{-1} } }^2\le C||h_1\ast \widehat Q||^2_2+C'||h_2\ast h_3||^2_2
\end{align*}
where,
\begin{align*}
&h_1(k)=(1+|k|^2)^{\sigma /2}|\widehat m(k),\,\,\,h_2(k)=(1+|k|^2)^{1 /2}|\widehat m(k)\\
&h_3(k)= \frac {(1+|\ell|^2)^{\sigma /2}} {(1+\log(1+|\ell|))^{1/2}}|\widehat Q(\ell)|.
\end{align*}
By hypothesis $\widehat Q\in L^2(\RR)$. On the other hand since $m\in H^{\sigma +\beta }$ for $\beta >1/2$ and $m\in H^{1 +\beta }$ for $\beta >1/2$, $h_1\in L^1(\RR)$ and $h_2\in L^1(\RR)$. And since $Q\in H^\sigma  _{ \log^{-1} }$, $h_3\in L^2(\RR)$. Therefore,
\begin{align*}
||mQ|| _{ H^\sigma  _{ \log^{-1} } }^2&\le C||h_1||_1^2||\widehat Q||_2^2+C'||h_2||_1^2||h_3||_2^2\\
&\le C||m||^2 _{ H^{\sigma +\beta } }||Q||^2_2+C||m||^2 _{ H^{1 +\beta } }||Q||^2 _{ H^\sigma  _{ \log ^{-1} } }.
\end{align*}
\vskip -0.5cm
\end{proof}
Let us prove now the following auxiliary Lemma,
\begin{lem}
\label{S3.3L3.6}
For all $\chi\in C_c^\infty(\RR)$, for all $\sigma >1$ there exists a constant $C>0$ such that for all $h\in H^\sigma $,
\begin{align*}
\left|\left| \int  _{ \RR } (\chi (\xi )-\chi (\zeta )) h(s, \zeta ) \left(\frac {1} {|e^{\xi-\zeta }-1 |}-\frac {1} {e^{\xi-\zeta }+ 1} \right)  d\zeta \right|\right| _{ H^\sigma  }\le C||h|| _{ H^{\sigma -1} }
\end{align*}
\end{lem}
\begin{proof}
By a simple change of variables,
\begin{align*}
&\int  _{ \RR } (\chi (\xi )-\chi (\zeta )) h(s, \zeta ) \left(\frac {1} {|e^{\xi-\zeta }-1 |}-\frac {1} {e^{\xi-\zeta }+ 1} \right)  d\zeta=\\
&\int  _{ \RR } (\chi (\xi )-\chi (\xi -\zeta  )) h(s, \xi -\zeta  ) \left(\frac {1} {|e^{\zeta  }-1 |}-\frac {1} {e^{\zeta  }+ 1} \right)  d\zeta=:B(h)(\xi )
\end{align*}
and,
\begin{align*}
||B(h)||^2 _{ H^\sigma  }=\int  _{ \RR }\left|\widehat {B(h)}(k)\right|^2(1+|k|^2)^\sigma dk
\end{align*}
with 
\begin{align*}
\widehat {B(h)}(k)&= \int  _{ \RR }e^{-ik\xi }\int  _{ \RR } (\chi (\xi )-\chi (\xi -\zeta  )) h(s, \xi -\zeta  ) \left(\frac {1} {|e^{\zeta  }-1 |}-\frac {1} {e^{\zeta  }+ 1} \right)  d\zeta d\xi\\
&=\frac {1} {2\pi }  \int  _{ \RR } \int  _{ \RR }  \int  _{ \RR }e^{-ik\xi }
(\chi (\xi )-\chi (\xi -\zeta  )) h(s, \xi -\zeta  )\times \\
&\hskip 4cm \times  \left(\frac {1} {|e^{\zeta  }-1 |}-\frac {1} {e^{\zeta  }+ 1} \right)e^{i\ell (\xi -\zeta ) }\widehat h(t, \ell)d\zeta  d\xi d\ell\\
&=\int  _{ \RR }\widehat P(k-\ell, \ell)\widehat h(\ell)d\ell
\end{align*}
where,
\begin{align*}
\widehat P(k, \ell)&=\int  _{ \RR }\int  _{ \RR }e^{-i(\xi k+\zeta \ell)}(\chi (\xi )-\chi (\xi -\zeta  )) \left(\frac {1} {|e^{\zeta  }-1 |}-\frac {1} {e^{\zeta  }+ 1} \right)d\zeta d\xi \\
&=\int  _{ \RR }\int  _{ \RR }e^{-i(\xi k+\zeta \ell)}M(\xi , \zeta)  Q(\zeta )d\zeta d\xi\\
P(\xi , \zeta )&=M(\xi , \zeta)Q(\zeta )\\
M(\xi , \zeta)&=\frac {(\chi (\xi )-\chi (\xi -\zeta  ))} {\zeta }\\
Q(\zeta )&= \left(\frac {1} {|e^{\zeta  }-1 |}-\frac {1} {e^{\zeta  }+ 1} \right)\zeta 
\end{align*}
By the regularity properties of $\chi$, for all $n\in \NN$,
\begin{align*}
\frac {\partial^n M} {\partial \xi ^n}(\xi, \zeta )&=\frac {\chi ^{(n)}(\xi )-\chi ^{(n)}(\xi-\zeta  )} {\zeta }\\
&= \int _0^1\chi ^{(n+1)} \left(\theta \xi +(1-\theta)(\xi -\zeta )\right)d\theta\le 
\left|\frac {\partial^n M} {\partial \xi ^n}(\xi, \zeta )\right|\le 1
\end{align*}
from where, for some constant $C_n>0$,
\begin{align*}
\left|\frac {\partial^n M} {\partial \xi ^n}(\xi, \zeta )\right|\le \frac {C_n} {1+|\zeta |},\,\,\forall \zeta \in \RR.
\end{align*}
A similar argument shows, for some other constant $C>0$,
\begin{align*}
|Q(\zeta )|\le C|\zeta |e^{-|\zeta |},\,\,\forall \zeta \in \RR.
\end{align*}
Then, since for all $\zeta \in \RR$,  $\text{supp}(M(\cdot , \zeta )\subset I(\zeta )$ with $I(\zeta )= \zeta +(\xi _0-2\delta  _{ \xi _0,} \xi _0+2\delta  _{ \xi _0 } )$, by Fubini's Theorem
\begin{align*}
\int  _{ \RR } \int  _{ \RR } \left|M(\xi , \zeta )\right|\left|Q(\zeta )\right|d\xi d\zeta \le C\int  _{ \RR} \frac {|\zeta |e^{-|\zeta |}} {1+|\zeta |} 
\int _{ I(\zeta ) }d\xi d\zeta = C\delta  _{ \xi 0 }\int  _{ \RR} \frac {|\zeta |e^{-|\zeta |}} {1+|\zeta |}d\zeta <\infty
\end{align*}
and $P\in L^1(\RR ^2)$ from where $\widehat P(k, \ell)$ is well defined for all $(k, \ell)\in \RR^2$. We obtain now decaying properties of $|\widehat P(k, \ell)|$ as $|k|+|\ell|\to \infty$. On the one hand, by integration by parts,
\begin{align*}
\widehat P(k, \ell)=\frac {1} {2\pi }(i\, k)^{-n}\int  _{ \RR }e^{-i\zeta \ell}S_n(k, \zeta )Q(\zeta )d\zeta \\
S_n(k, \zeta )=\int  _{ \RR }e^{-i\xi k}\frac {\partial ^n M(\xi , \zeta )} {\partial \xi ^n}d\xi. 
\end{align*}
If $|\zeta |\le 1$, use now, for all $n\in \NN$ and all $m\in \NN$ and $p\in \NN$,
\begin{align*}
\frac {\partial ^{n+m+p}M(\xi , \zeta )} {\partial \xi ^{n+p}\partial \zeta ^m}&=(-1)^m \int _0^1 \chi ^{(n+m+p+1)}\left( \theta \xi +(1-\theta)(\xi -\zeta )\right)(1-\theta)^m d\theta \\
&\Longrightarrow \left|\frac {\partial ^{n+m+p}M(\xi , \zeta )} {\partial \xi ^{n+p}\partial \zeta ^m} \right|\le 1.
\end{align*}
We deduce first for all $k$,
\begin{align}
\frac {\partial ^m S_n(k, \zeta )} {\partial \zeta ^m}&= \int  _{ \RR }e^{-i\xi k}\frac {\partial ^{n+m}M(\xi , \zeta )} {\partial \xi ^{n+p}\partial \zeta ^m}d\xi \label{lndn1} \\
\left|\frac {\partial ^m S_n(k, \zeta )} {\partial \zeta ^m}\right|&\le  \int  _{  I(\zeta ) }\left|\frac {\partial ^{n+m}M(\xi , \zeta )} {\partial \xi ^{n}\partial \zeta ^m}\right|d\xi \le 4\delta  _{ \xi _0 }.\nonumber
\end{align}
But for $|k|>1$, an  integration by parts  gives,
\begin{align*}
\frac {\partial ^m S_n(k, \zeta )} {\partial \zeta ^m}&=\frac {1} {(-ik)^p}\int  _{ \RR }e^{-i\xi k}\frac {\partial ^{n+m+p}M(\xi , \zeta )} {\partial \xi ^{n+p}\partial \zeta ^m}d\xi \\
\left|\frac {\partial ^m S_n(k, \zeta )} {\partial \zeta ^m}\right|&\le |k|^{-p}\int  _{  I(\zeta ) }\left|\frac {\partial ^{n+m+p}M(\xi , \zeta )} {\partial \xi ^{n+p}\partial \zeta ^m}\right|d\xi \le 4 |k|^{-p}\delta  _{ \xi _0 }
\end{align*}
Similarly, when $|\zeta |>1$, by the hypothesis on $\chi $,  there exists $C>0$, depending on $m, n$ and $p$ such that
\begin{align*}
\left| \frac {\partial ^{m+n+p} M(\xi , \zeta )} { \partial \xi ^{n+p} \partial \zeta ^m}\right|\le \frac {C} {|\zeta |},\,\,\forall \zeta \in \RR,\, |\zeta |>1
\end{align*}
Then, using (\ref{lndn1} ),
\begin{align*}
\left|\frac {\partial ^m S_n(k, \zeta )} {\partial \zeta ^m}\right|\le \int  _{  I(\zeta ) }\left|\frac {\partial ^{n+m}M(\xi , \zeta )} {\partial \xi ^{n}\partial \zeta ^m}\right|d\xi
\le C\delta  _{ \xi _0 }
\end{align*}
and, again by integration by parts if $|k|>1$,
\begin{align*}
\frac {\partial ^m S_n(k, \zeta )} {\partial \zeta ^m}=\frac {(-1)^m} {(-ik)^p}\int  _{ \RR }e^{-i\xi k}\frac {\partial ^{n+m+p}M(\xi , \zeta )} {\partial \xi ^{n+p}\partial \zeta ^m}d\xi \\
\left|\frac {\partial ^m S_n(k, \zeta )} {\partial \zeta ^m}\right|\le \int  _{  I(\zeta ) }\left|\frac {\partial ^{n+m+p}M(\xi , \zeta )} {\partial \xi ^{n+p}\partial \zeta ^m}\right|d\xi
\le \frac {C} {|k |^p|\zeta |}\delta  _{ \xi _0 }.
\end{align*}
On has then,
 \begin{align*}
\left|\frac {\partial ^m S_n(k, \zeta )} {\partial \zeta ^m}\right|\le \frac {C} {(1+|k|^p)(1+|\zeta |)}.
\end{align*}
If we consider now $\widehat P(k, \ell)$,
\begin{align*}
\widehat P(k, \ell)= \frac {1} {2\pi }(i\, k)^{-n}\int  _{ \RR }e^{-i\zeta \ell}S_n(k, \zeta )Q(\zeta )d\zeta 
\end{align*}
an integration by parts would give,
\begin{align*}
\widehat P(k, \ell)= \frac {1} {2\pi }(i\, k)^{-n}(i\, \ell)^{-1})\int  _{ \RR }e^{-i\zeta \ell}\frac {\partial } {\partial  \zeta }\left(S_n(k, \zeta )Q(\zeta )\right)d\zeta 
\end{align*}
and  simple calculus gives,
\begin{align*}
\left|\frac {\partial } {\partial  \zeta }\left(S_n(k, \zeta )Q(\zeta )\right) \right|\le \frac {C(1+|\zeta |)e^{-|\zeta |}} {(1+|k|^p)(1+|\zeta |)}
\end{align*}
and then, for all $n\in \NN$ there exists a constant $C>0$ such that
\begin{align*}
\left|\widehat P(k, \ell)\right|\le  \frac {C} {(1+|k|^n)(1+|\ell|)},\,\,\forall k\in \RR, \forall \ell \in \RR.
\end{align*}
It follows,
\begin{align*}
||B(h)|| _{ H^{\sigma   }}^2 &\le \int  _{ \RR }(1+|k|^{2})^{\sigma }\left|\int  _{ \RR }\widehat P(k-\ell, \ell)\widehat h(\ell)d\ell \right|^2dk\\
&\le C \int  _{ \RR }(1+|k|^{2})^{ \sigma }\left|\int  _{ \RR } \frac {\widehat h(\ell)} {(1+|k-\ell|^n)(1+|\ell|)}d\ell \right|^2dk\\
&\le C \int  _{ \RR }(1+|k|^{2})^{ \sigma }\left|\int  _{|\ell|\le \frac {|k |} {2} } \frac {\widehat h(\ell)} {(1+|k-\ell|^n)(1+|\ell|)}d\ell \right|^2dk+\\
&+C \int  _{ \RR }(1+|k|^{2})^{ \sigma }\left|\int  _{|\ell|\ge \frac {|k |} {2} } \frac {\widehat h(\ell)} {(1+|k-\ell|^n)(1+|\ell|)}d\ell \right|^2dk=B_1+B_2.
\end{align*}
Since $|\ell-k|>|k|/2>|\ell|$ in $B_1$ one has  for $n$ large enough,
\begin{align*}
B_1&\le C\int  _{ \RR }\frac {(1+|k|^{2})^{ \sigma }} {(1+|k|^{2})^{ n}}\left| \int  _{ |\ell|\le  \frac {|k|} {2} } \frac {\widehat h(\ell)} {(1+|\ell|^{n/2})(1+|\ell|)}d\ell\right|^2dk\\
&\le C\int  _{ \RR }\frac {(1+|k|^{2})^{ \sigma }} {(1+|k|^{2})^{ n}}dk \int  _{ \RR } \frac {d\ell } {(1+|\ell|^{n})(1+|\ell|^2)}\int  _{  \RR }|\widehat h(\ell)|^2 d\ell\le ||h||_2^2.
\end{align*}
In the term $B_2$, use of $|\ell|\ge |k|/2$ gives

\begin{align*}
B_2&\le C \int  _{ \RR }\frac{(1+|k|^{2})^{ \sigma }}{(1+|k|^{2})^{ \sigma }}\left|\int  _{|\ell|\ge \frac {|k |} {2} } \frac {(1+|\ell|^2)^{\sigma /2}\widehat h(\ell)} {(1+|k-\ell|^n)(1+|\ell|)}d\ell \right|^2dk\\
&\le C \int  _{ \RR }\left|\int  _{|\ell|\ge \frac {|k |} {2} } \frac {(1+|\ell|^2)^{\frac {\sigma -1} {2}}\widehat h(\ell)} {(1+|k-\ell|^n)}d\ell \right|^2dk\\
&\le C \int  _{ \RR } \int  _{\RR}  (1+|\ell|^2)^{\sigma -1}|\widehat h(\ell)|^2d\ell\,\left| \int  _{ \RR }\frac {d\ell } {(1+|k-\ell|^n)}\right|^2dk
\le C||h|| ^2_{ (\sigma -1)^+ }.
\end{align*}
\vskip -0.5cm 
\end{proof}

\begin{lem}
\label{S7Alocloc}
For all open bounded interval $I\subset \RR$ and all $\chi \in C_c^1(\RR)$ such that $\chi =1$ on $I$, there exists a constant $C_I>0$ such that if $m=[\sigma ]$,
\begin{align*}
&\int  _{I }\int  _{ I}\frac {|w (\xi )-w (\zeta  )|^2} {|\xi -\zeta |} d\zeta  d\xi
\le C_I ||\chi w||^2 _{ H^0  _{ \log }(\RR) } 
\end{align*}
\end{lem}
\begin{proof} 
It readily follows from  property (\ref{S1EHl2})) that for all bounded interval $I'\subset \RR$, there exists a constant such that
\begin{align*}
\int  _{ \RR }\int  _{ h\in I' }\frac {|w(\xi )-w(\xi -h )|^2} {|\xi -\zeta |} dh d\xi \le C ||w|| ^2_{ H^0 _{ \log } }
=\int  _{ \RR }\int  _{ \xi +I' }\frac {|w(\xi )-w(\zeta  )|^2} {|\xi -\zeta |} d\zeta  d\xi \le C ||w|| ^2_{ H^0 _{ \log } }.
\end{align*}
after the change of variable $h=\xi -\zeta$ .
If we suppose now that $I=(\xi _0, \xi _1)$ for some $\xi _0<\xi _1$ and consider $I'=(-(\xi _1-\xi _0), \xi _1-\xi _0)$, then, 
$I\times I\subset \left\{(\xi , \zeta );\,\,\xi \in I,\,\zeta \in \xi +I' \right\}$
and,
\begin{align*}
&\int  _{I }\int  _{ I}\frac {|w(\xi )-w(\zeta  )|^2} {|\xi -\zeta |} d\zeta  d\xi=
\int  _{I }\int  _{ I}\frac {|\chi (\xi )w(\xi )-\chi (\zeta )w(\zeta  )|^2} {|\xi -\zeta |} d\zeta  d\xi\\
&\le \int_I\int  _{ \xi +I' }\frac {|\chi (\xi )w(\xi )-\chi (\zeta )w(\zeta  )|^2} {|\xi -\zeta |} d\zeta  d\xi
\le  \int _{ \RR }\int  _{ \xi +I' }\frac {|\chi (\xi )w(\xi )-\chi (\zeta )w(\zeta  )|^2} {|\xi -\zeta |} d\zeta  d\xi
\end{align*}
from where Lemma follows.
\end{proof}

\textbf{Acknowledgments.}
The research of the author is supported by grant PID2020-112617GB-C21 of MCIN, grant  IT1247-19 of the Basque Government and grant RED2022-134784-T funded by MCIN/AEI/10.13039/501100011033.

\end{document}